\documentclass[secnum,leqno]{rims-bessatsu-a}
\usepackage{amsmath,amssymb,graphicx,latexsym,xcolor,pstricks,multicol}
\usepackage[utf8]{inputenc}
\usepackage{newunicodechar}
\usepackage[all]{xy}
\newtheorem{thm}{Theorem}[section]
\newtheorem{crl}[thm]{Corollary}

\newtheorem{lmm}[thm]{Lemma}

\newtheorem{prp}[thm]{Proposition}

\theoremstyle{definition}
\newtheorem{dfn}[thm]{Definition}
\newtheorem{exa}[thm]{Example}
\theoremstyle{remark}
\newtheorem*{rem}{Remark}
\allowdisplaybreaks

\title{Duality between box-ball systems of finite box and/or carrier capacity}
\author{David A.\ \textsc{Croydon}\footnote{Department of Advanced Mathematical Sciences, Graduate School of Informatics, Kyoto University, Sakyo-ku, Kyoto 606--8501, Japan.\newline e-mail: \texttt{croydon@acs.i.kyoto-u.ac.jp}}
          ~and Makiko \textsc{Sasada}\footnote{Graduate School of Mathematical Sciences, University of Tokyo, 3-8-1, Komaba, Meguro-ku, Tokyo, 153--8914, Japan. \endgraf e-mail: \texttt{sasada@ms.u-tokyo.ac.jp}}}
\AuthorHead{D.\ A.\ Croydon and M.\ Sasada}
\support{DC was supported by the JSPS Grant-in-Aid for Research Activity Start-up, 18H05832, and MS was supported by the JSPS Grant-in-Aid for Scientific Research (B), 16KT0021.}

\begin{document}

\maketitle

\begin{abstract}
We construct the dynamics of the box-ball system with box capacity $J$ and carrier capacity $K$, which we abbreviate to BBS($J$,$K$), in the case of infinite initial configurations, and show that this system is dual to the analogous BBS($K$,$J$) model. Towards this end, we build on previous work for the original box-ball system, that is BBS($1$,$\infty$), to show that when the box capacity $J$ and carrier capacity $K$ satisfy $J<K$ the dynamics can be represented by a Pitman-type transformation. These ideas are applied in the case of random initial configurations to show that the distributional properties of spatial stationarity and invariance under the BBS dynamics are dual. Moreover, for independent and identically distributed configurations, we derive a characterisation of invariant measures in terms of a detailed balance equation, which captures the duality of the system locally; this is used to find all invariant measures in this class. Finally, we deduce the speed of a tagged particle, and show that this also satisfies a natural duality relation.
\end{abstract}

\section{Introduction}

The box-ball system (BBS) with box capacity $J\in\mathbb{N}\cup\{\infty\}$ and carrier capacity $K\in\mathbb{N}\cup\{\infty\}$, which we will henceforth abbreviate to BBS($J$,$K$), was introduced in \cite{TakaMatsu}, where for $J<K$ it was shown to arise via a limiting procedure from the discrete modified Korteweg-de Vries equation. The dynamics of the system can be described as follows. The initial configuration is represented as a sequence $\eta=(\eta_n)_{n\in\mathbb{Z}}$ in the configuration space $\mathcal{C}_{J,K}:=\{0,1,\dots,J\}^\mathbb{Z}$ (where we interpret $\{0,1,\dots,J\}$ as $\mathbb{Z}_+$ for $J=\infty$), with $\eta_n$ denoting the number of balls in the $n$th box. To begin our exposition, we consider the case when there is a finite number of balls in the system, i.e.\ $\sum_{n\in\mathbb{Z}}\eta_n<\infty$. The evolution of the BBS($J$,$K$) is then given by an operator $T\equiv T_{J,K}$ that maps $\eta$ to another configuration $T\eta$ in $\{0,1,\dots,J\}^\mathbb{Z}$ that is characterised by setting:
\begin{eqnarray}
(T\eta)_n&=&\eta_n+\min\left\{\sum_{m=-\infty}^{n-1}\left(\eta_m-(T\eta)_m\right),J-\eta_n\right\}\nonumber\\
&&\hspace{100pt}-\min\left\{\eta_n,K-\sum_{m=-\infty}^{n-1}\left(\eta_m-(T\eta)_m\right)\right\},\label{tndef}
\end{eqnarray}
where we fix $(T\eta)_n=0$ for $n<\inf\{m:\:\eta_m\neq 0\}$ (with the convention that $\inf\emptyset=\infty$), so that the sums in the above formula are well-defined. In more transparent terms, this means that, considering one particle at a time from the left to the right (the choice of order within each box is unimportant), each particle moves to the nearest empty space on its right, unless this would mean more than $K$ particles cross from one location to the next, in which case the move is denied and the particle stays in its current location. See Figure \ref{dynamicsfig} for an example.

\begin{figure}[t]
\begin{center}
{\includegraphics[width=0.9\textwidth]{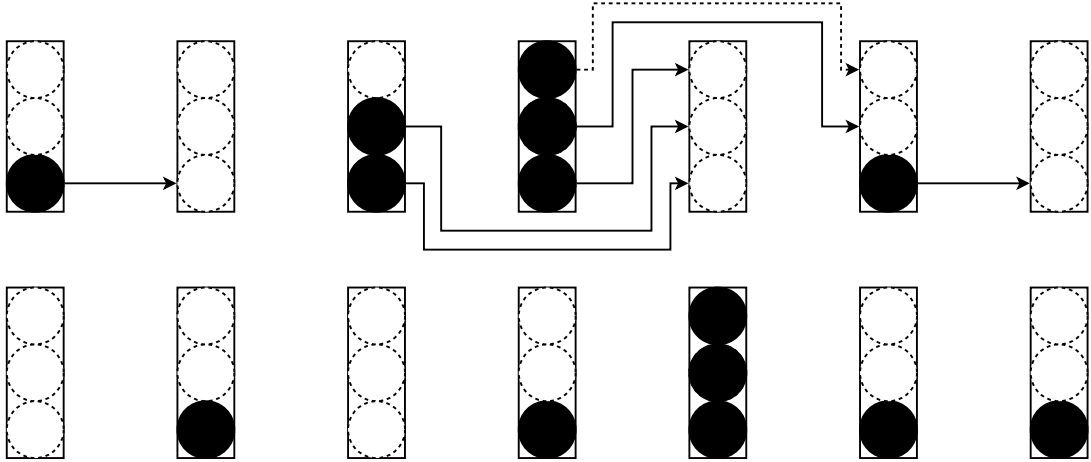}}
\vspace{-15pt}
\end{center}
\caption{Dynamics of the BBS($3$,$4$). The top diagram shows the initial configuration and attempted moves, with the move shown by a dashed line being denied. The bottom diagram shows the new configuration.}\label{dynamicsfig}
\end{figure}

Importantly in what follows, one has an alternative intuitive description of the system in terms of a ‘carrier’. The carrier, which is initially not carrying any balls, moves along $\mathbb{Z}$ from left to right (that is, from negative to positive). When it reaches a box containing $a\in\{0,\dots,J\}$ balls and is carrying $b\in\{0,\dots,K\}$ balls, it picks up as many balls as it has spare capacity for (i.e.\ $\min\{a,K-b\}$), and places down as many balls as fit in the initially empty space in the box (i.e.\ $\min\{b,J-a\}$). It then proceeds to repeat this procedure at the next box to the right, and so on. (Since we are assuming the finiteness of the initial configuration, it is apparent that we only need consider the action of the carrier from the first non-empty box.) If we associate with the carrier a process $W=(W_n)_{n\in\mathbb{Z}}$ by setting $W_n$ to be the number of balls that the carrier holds after having visited box $n$, then we can rewrite the equation at \eqref{tndef} more concisely:
\begin{equation}\label{carrierdef}
(T\eta)_n=\eta_n+\min\left\{W_{n-1},J-\eta_n\right\}-
\min\left\{\eta_n,K-W_{n-1}\right\}.
\end{equation}
Figure \ref{bbs75} illustrates the evolution at one site of the BBS($J$,$K$) in terms of the associated carrier, as well as some of the symmetries of the model that will feature in this article.

\begin{figure}[t]
\begin{center}
\scalebox{0.6}{
\includegraphics[width=0.9\textwidth]{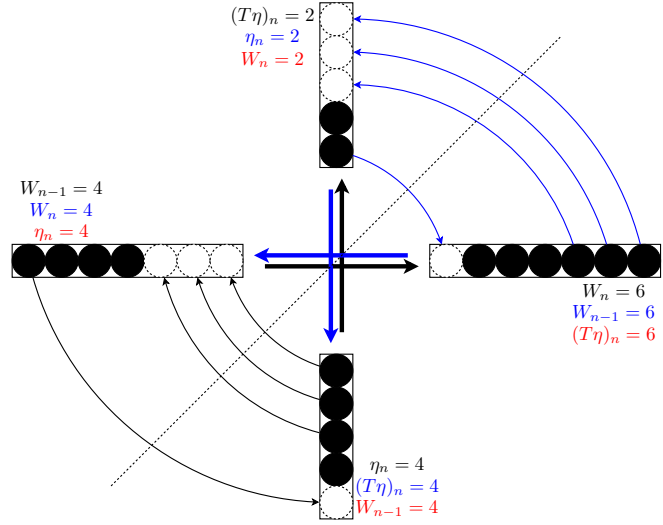}
\rput(-5,1){$\eta_n=4$}
\rput(-5,0.6){\color{blue}$(T\eta)_{n}=4$}
\rput(-5,0.2){\color{red}$W_{n-1}=4$}
\rput(-11.2,6.1){$W_{n-1}=4$}
\rput(-11.2,5.7){\color{blue}$W_n=4$}
\rput(-11.2,5.3){\color{red}$\eta_n=4$}
\rput(-1.0,4.2){$W_{n}=6$}
\rput(-1.0,3.8){\color{blue}$W_{n-1}=6$}
\rput(-1.0,3.4){\color{red}$(T\eta)_n=6$}
\rput(-7.3,9.3){$(T\eta)_n=2$}
\rput(-7.3,8.9){\color{blue}$\eta_{n}=2$}
\rput(-7.3,8.5){\color{red}$W_{n}=2$}}
\vspace{-15pt}
\end{center}
\caption{Black arrows and annotations show an example evolution of the BBS($J$,$K$) with $J=5$ and $K=7$, in terms of the carrier. Blue arrows and annotations illustrate the involutive property of the dynamics. Reflecting in the dotted line yields the corresponding picture for the dual model BBS($K$,$J$), which is described by the black arrows and red annotations. (There is also an involutive property for the latter model, which is not annotated.)
}\label{bbs75}
\end{figure}

A fundamental question is whether the above definition of the BBS($J$,$K$) can be extended from finite initial configurations to a larger subset of $\mathcal{C}_{J,K}$ in a natural way. Of course, as soon as $\sum_{n=-\infty}^0\eta_n=\infty$, then the original formulation at \eqref{tndef} is no longer well-defined. However, it is still reasonable to ask if one can find a carrier $(W_n)_{n\in \mathbb{Z}}$ that takes values in $\{0,1,\dots,K\}^\mathbb{Z}$ and is consistent with $(\eta_n)_{n\in\mathbb{Z}}$ in the sense that
\begin{equation}\label{wdyn}
W_n=W_{n-1}-\min\left\{W_{n-1},J-\eta_n\right\}+
\min\left\{\eta_n,K-W_{n-1}\right\},
\end{equation}
as is required by the dynamics; this can then be used to define $T\eta$ via \eqref{carrierdef}. For future reference, we note that summing \eqref{carrierdef} with \eqref{wdyn} gives the conservation of mass formula:
\begin{equation}\label{consmass}
(T\eta)_n+W_n=\eta_n+W_{n-1}.
\end{equation}

As it transpires, a carrier does not necessarily exist, and, if it does, then it is not necessarily unique (Section \ref{carriersec} contains a detailed exploration of this issue). However, in Definition \ref{cancarr} below, we introduce the notion of a `canonical carrier', which is unique if it exists. We argue from basic physical considerations that this choice is natural. Indeed, similarly to \cite[Remark 2.11]{CKST}, which covered  BBS($1$,$\infty$) and explained how the choice of carrier of that article ruled out the transport into the system of particles from $-\infty$, our choice means that a canonical carrier $W$ can always be determined endogenously, i.e.\ from the current state of the system -- as we demonstrate in Lemma \ref{unique} below, it is actually the case that $W_n$ is a function of $(\eta_m)_{m\leq n}$. Additionally, as Lemma \ref{keylemma}, we show that non-canonical carriers lead to a certain type of degenerate behaviour. In the case $J<K$, we show this means that if a random configuration has a distribution that is invariant under the dynamics induced by a non-canonical carrier, then the resulting dynamics are almost-surely trivial, see Proposition \ref{p33}. (NB. The distinction between configurations admitting a canonical carrier or only non-canonical carriers parallels that between the sub-critical and critical classes of configurations considered in \cite{CKST}.) We are able to completely characterise the set of configurations that admit a canonical carrier, $\mathcal{C}^{can}_{J,K}$, for all choices of $J$ and $K$, see Propositions \ref{ch1}, \ref{ch2}, \ref{ch3} and \ref{ch4}. Thus we arrive at our extension of the BBS($J$,$K$) to infinite initial configurations.

A significant motivation for considering infinite configurations comes in the search for random configurations that are invariant in distribution under the box-ball system dynamics. Indeed, beyond the trivial case of there being no particles, the natural transience of the system means that invariant random configurations must comprise an infinite number of particles. (Another motivation for studying infinite configurations is that it allows us to treat a periodic version of the model, cf.\ \cite[Remark 1.13]{CKST} and \cite{CS}, though we do not pursue this issue here.) Furthermore, in the study of the invariant random configurations, one is naturally led to look for configurations for which one can extend the dynamics to all times, both forwards and backwards.
For the original box-ball system introduced in \cite{takahashi1990}, that is BBS($1$,$\infty$) in our notation, the set of such configurations was completely characterised in \cite{CKST}. Here, we do not attempt to repeat this program for BBS($J$,$K$) for general $J$ and $K$, but nonetheless it will still be important to give an abstract description of the invariant set, and be able to describe a suitably rich subset of it. To set-out the relevant aspects of this discussion more precisely, we first introduce the spatially reversed configuration $R\eta=((R\eta)_n)_{n\in\mathbb{Z}}$, as given by
\begin{equation}\label{rdef}
(R\eta)_n=\eta_{1-n},
\end{equation}
and define the set
\begin{equation}\label{crevdef}
\mathcal{C}_{J,K}^{rev}:=\left\{\eta\in\mathcal{C}^{can}_{J,K}:\:R\eta\in\mathcal{C}^{can}_{J,K}\right\}
\end{equation}
of configurations such that both $\eta$ and $R\eta$ admit a canonical carrier; the `rev' here is a contraction of `reversible', which we use in a dynamical systems sense, since the dynamics have a natural inverse on this set. Indeed, as we will demonstrate in Proposition \ref{tinvprp} below, for $\eta\in \mathcal{C}^{rev}_{J,K}$, we have that $RTRT\eta=TRTR\eta=\eta$, and so we can consider $T^{-1}:=RTR$ to be the inverse for $T$ on the set in question. This characterisation can be interpreted as meaning that the inverse of $T$ is given by running the carrier in the reverse direction, from right to left, which is a familiar description in the finite particle case. Moreover, we note that it is a natural global consequence of the involutive property of the local dynamics given by the map $(\eta_n,W_{n-1})\mapsto ((T\eta)_n,W_{n})$, as is illustrated in Figure \ref{bbs75}. Given the set $\mathcal{C}_{J,K}^{rev}$, we then define an invariant set
\begin{equation}\label{cinvdef}
\mathcal{C}_{J,K}^{inv}:=\left\{\eta\in\mathcal{C}_{J,K}:\:T^t\eta\in\mathcal{C}^{rev}_{J,K},\:\forall t\in\mathbb{Z}\right\},
\end{equation}
upon which all the laws of invariant random configurations studied here are supported.

We now come to the presentation of our main results concerning duality and invariant measures. To begin with the first of these issues, observe that if Figure \ref{bbs75}, which shows the dynamics of a BBS($J$,$K$) system, is reflected in the dotted line, then it shows the dynamics of a BBS($K$,$J$) system; this is a simple consequence of the symmetry with respect to reversing the roles of $J$ and $K$, and the roles of $(\eta_n,(T\eta)_n)$ and $(W_{n-1},W_n)$, in the equations at \eqref{carrierdef} and \eqref{wdyn}. (For a more formal description of this relation between BBS($J$,$K$) and BBS($K$,$J$), see \eqref{dualityrel} below.) The following result can be seen as an extension of this local picture to a global one.  Specifically it shows a duality between the initial particle configurations $\eta=(\eta_n)_{n\in\mathbb{Z}}\in\mathcal{C}^{inv}_{J,K}$ and the corresponding current of particles crossing the origin, as described by $((T^tW)_0)_{t\in\mathbb{Z}}$, where $T^tW$ is the canonical carrier of $T^t\eta$. For its statement, we define a duality map $\mathcal{D}_{J,K}: \mathcal{C}_{J,K}^{inv} \to \mathcal{C}_{K,J}$ by setting
\begin{equation}\label{dualitymap}
\mathcal{D}_{J,K}(\eta)=((T^{t}W)_0)_{t \in \mathbb{Z}}.
\end{equation}

\begin{thm}\label{mainthm1} (a) Fix $J,K\in\mathbb{N}\cup\{\infty\}$ such that either $J\geq K$ or $J<K<\infty$. For any $\eta \in \mathcal{C}_{J,K}^{inv}$, it is the case that $((T^{t}W)_0)_{t \in \mathbb{Z}} \in \mathcal{C}_{K,J}^{inv}$. Moreover, for each $n\in\mathbb{Z}$, $((T^{t+1}\eta)_{n+1})_{t \in \mathbb{Z}}$ is the canonical carrier for $(T^{t}W_n)_{t \in \mathbb{Z}}$ with respect to the BBS($K$,$J$) dynamics.\\
(b) Fix $J,K\in\mathbb{N}\cup\{\infty\}$. If we define
\[\tilde{\mathcal{C}}_{J,K}^{inv}:=\left\{
                           \begin{array}{ll}
                             \mathcal{D}_{K,J}(\mathcal{C}_{K,J}^{inv}), & \hbox{when }J<K=\infty,\\
                             \mathcal{C}_{J,K}^{inv}, & \hbox{otherwise},                           \end{array}
                         \right.\]
then $\mathcal{D}_{J,K}$ is a bijection between $\tilde{\mathcal{C}}_{J,K}^{inv}$ and $\tilde{\mathcal{C}}_{K,J}^{inv}$, with inverse given by
\begin{equation}
\mathcal{D}_{J,K}^{-1}:=\theta^{-1}\circ\mathcal{D}_{K,J}\circ\theta^{-1},\label{dualinv}
\end{equation}
where $\theta$ is the usual left-shift on doubly infinite sequences, i.e.\ $\theta((x_n)_{n\in\mathbb{Z}})=(x_{n+1})_{n\in\mathbb{Z}}$.
\end{thm}

\begin{rem} We show in Example \ref{notbijection} below that in the case $J<\infty$ there exists a configuration $\eta\in\mathcal{C}_{J,\infty}^{inv}\backslash\tilde{\mathcal{C}}_{J,\infty}^{inv}$ such that $\mathcal{D}_{J,\infty}(\eta)\not\in \mathcal{C}_{\infty,J}^{inv}$. Hence $\mathcal{D}_{J,\infty}$ is not a bijection between $\mathcal{C}^{inv}_{J,\infty}$ and $\mathcal{C}^{inv}_{\infty,J}$.
\end{rem}

\begin{rem} The inclusion of the shift map in the result at \eqref{dualinv} is simply by convention. Indeed, suppose that we instead define the carrier process to be given by $\tilde{W}=(\tilde{W}_n)_{n\in\mathbb{Z}}$, where $\tilde{W}_n$ represents the number of balls being carried by the carrier when it arrives at location $n$, i.e. $\tilde{W}_n=W_{n-1}$, and define the duality map by setting $\tilde{\mathcal{D}}_{J,K}(\eta)=((T^{t}\tilde{W})_0)_{t \in \mathbb{Z}}$, then we have the rather more elegant statement that, for each $n\in\mathbb{Z}$, $((T^{t}\eta)_{n})_{t \in \mathbb{Z}}$ is the canonical carrier for $(T^{t}\tilde{W}_n)_{t \in \mathbb{Z}}$, and \eqref{dualinv} becomes $\tilde{\mathcal{D}}_{J,K}^{-1}=\tilde{\mathcal{D}}_{K,J}$. However, we choose the index of the carrier $W$ as we do because it aligns with that chosen in the earlier work \cite{CKST}, and also because it will be convenient when it comes to our description of the dynamics in terms of certain path encodings (see Subsection \ref{J<K}).
\end{rem}

\begin{rem}
Box-ball systems with finite capacity are known to arise from quantum integrable systems by a procedure called crystallization. Precisely, these models form a class of two-dimensional multi-state vertex models that can be constructed from the six-vertex model by a fusion procedure. The duality of the box-ball system can be understood as a certain symmetry of these two-dimensional multi-state vertex models. For more detailed background, see \cite{IKT}.
\end{rem}

The strong link between the initial configuration and particle current was already applied in \cite{CKST}, where the properties of invariance and ergodicity of a random configuration $\eta$ under $T$ were shown to be equivalent to the corresponding properties for the current $((T^{t}W)_0)_{t \in \mathbb{Z}}$ under the spatial shift $\theta$. In the first main probabilistic result of this paper, we apply the deterministic relation of Theorem \ref{mainthm1} to extend such parallels to more general box-ball systems. At the heart of the proof is the observation that
\begin{equation}\label{djkident}
\mathcal{D}_{J,K}\circ T_{J,K}=\theta\circ\mathcal{D}_{J,K}
\end{equation}
on $\mathcal{C}_{J,K}^{inv}$, which is straightforward to check from the definition of $\mathcal{D}_{J,K}$. Towards stating the result, we introduce $\mathcal{P}_{J,K}, \mathcal{P}_{J,K}^{rev}, \mathcal{P}_{J,K}^{inv}, \tilde{\mathcal{P}}_{J,K}^{inv}$ for the collections of probability measures supported on $\mathcal{C}_{J,K}, \mathcal{C}_{J,K}^{rev}, \mathcal{C}_{J,K}^{inv}, \tilde{\mathcal{C}}_{J,K}^{inv}$, respectively. We moreover note that the map $\mathcal{D}^{P}_{J,K}:\mathcal{P}_{J,K}^{inv}\rightarrow\mathcal{P}_{K,J}$ given by
\begin{equation}\label{djkpdef}
\mathcal{D}^{P}_{J,K}\left(P_{J,K}\right):=P_{J,K} \circ \mathcal{D}_{J,K}^{-1}\circ\theta
\end{equation}
is a bijection between $\tilde{\mathcal{P}}_{J,K}^{inv}$ and $\tilde{\mathcal{P}}_{K,J}^{inv}$, with inverse given by $\mathcal{D}^{P}_{K,J}$.

\begin{thm}\label{dualthm} Fix $J,K\in\mathbb{N}\cup\{\infty\}$.\\
(a) If $P_{J,K} \in \mathcal{P}_{J,K}^{rev}$ and $P_{J,K} \circ  T_{J,K}^{-1} = P_{J,K}$, then  $P_{J,K} \in \tilde{\mathcal{P}}_{J,K}^{inv}$.\\
(b) If $P_{J,K} \in \mathcal{P}_{J,K}^{inv}$ and $P_{K,J}:= \mathcal{D}^{P}_{J,K}(P_{J,K})$, then $P_{J,K} \circ T_{J,K}^{-1} = P_{J,K}$ if and only if $P_{K,J} \circ \theta^{-1} = P_{K,J}$.\\
(c) If the transform that appears in one of the sides of (b) is ergodic for the relevant measure, then so is the transform that appears in the other side.
\end{thm}

We next turn our attention to measures in $\mathcal{P}_{J,K}$ of product form, or in other words, random configurations such that the elements of $(\eta_n)_{n\in\mathbb{Z}}$ are independent and identically distributed (i.i.d.). The possible marginals of such measures will be denoted by $\mathcal{M}_{J,K}$, i.e.\ this is the collection of probability measures $\mu_{J,K}$ supported on $\{0,1,2,\dots,J\}$. We will moreover write
\begin{align}
\mathcal{M}_{J,K}^{rev}&:=\left\{\mu_{J,K} \in \mathcal{M}_{J,K}:\:\mu_{J,K}^{\otimes \mathbb{Z}}\in\mathcal{P}_{J,K}^{rev}\right\},\label{mrevdef}\\
\mathcal{M}_{J,K}^{inv}&:=\left\{\mu_{J,K} \in \mathcal{M}_{J,K}:\:\mu_{J,K}^{\otimes \mathbb{Z}}\in\mathcal{P}_{J,K}^{inv}\right\};\label{minvdef}
\end{align}
actually, we will show in Proposition \ref{mrevprp} that these two sets are the same. Now, from Theorem \ref{dualthm}, we immediately observe that arbitrary elements of $\mathcal{M}_{J,K}^{rev}$ yield invariant measures for the dual model BBS($K$,$J$), as we make precise in the following corollary.

\begin{crl} Fix $J,K\in\mathbb{N}\cup\{\infty\}$. Let $\mu_{J,K} \in \mathcal{M}_{J,K}^{rev}$, and define $P_{K,J}:=  \mathcal{D}^{P}_{J,K}(\mu_{J,K}^{\otimes \mathbb{Z}})$. It then holds that $P_{K,J} \circ T_{K,J}^{-1}= P_{K,J}$. Moreover, $T_{K,J}$ is ergodic for $P_{K,J}$.
\end{crl}

Whilst it is interesting to consider what invariant measures arise in this way, our main focus in this part of the article will be on another basic problem: for what choices of $\mu_{J,K}\in\mathcal{M}_{J,K}^{rev}$ is the corresponding product measure $\mu_{J,K}^{\otimes \mathbb{Z}}$ invariant under $T_{J,K}$? And, in Theorem \ref{iidclass} below, we give a complete answer to this question. Whilst we postpone the statement of this result to avoid setting out the necessary technical preparations here, we will introduce the key idea for its proof, which is a certain local duality property that, by analogy with Markov chain terminology, we will call detailed balance, see \eqref{dbalance} below. As above, we appeal to the relation between the configuration and the current, now defining a map $\mathcal{D}^\mu_{J,K}:\mathcal{M}_{J,K}^{rev}\rightarrow \mathcal{M}_{K,J}$ by setting
\begin{equation}\label{dmudef}
\mathcal{D}^\mu_{J,K}\left(\mu_{J,K}\right):=\mu_{J,K}^{\otimes \mathbb{Z}}\circ W_0^{-1}.
\end{equation}
So, if $\eta\sim \mu_{J,K}^{\otimes \mathbb{Z}}$ and we set $\mu_{K,J}:=\mathcal{D}^\mu_{J,K}(\mu_{J,K})$, then, since $W_{n-1}$ is $(\eta_{m})_{m\leq n-1}$ measurable (by Lemma \ref{unique}), this means that
\begin{equation}\label{prodmeasure}
(\eta_n,W_{n-1})\sim\mu_{J,K}\times\mu_{K,J}.
\end{equation}
And, as is made precise in the following theorem, we will show that the invariance of $\mu_{J,K}^{\otimes \mathbb{Z}}$ under $T_{J,K}$ is equivalent to the invariance of the above law under the map $(\eta_n,W_{n-1})\mapsto((T\eta)_n,W_{n})$ given by \eqref{carrierdef} and \eqref{wdyn}, together with a certain consistency condition between the ranges of the support of the measures $\mu_{J,K}$ and $\mu_{K,J}$ holding. To state the result in a concise way and for later use,  we set
\begin{equation}\label{iinvdef}
\mathcal{I}_{J,K}^{inv}:=\left\{\mu_{J,K} \in \mathcal{M}^{rev}_{J,K}:\:\mu_{J,K}^{\otimes \mathbb{Z}} \circ T_{J,K}^{-1} = \mu_{J,K}^{\otimes \mathbb{Z}}\right\}.
\end{equation}
We also define a map $(a,b)\mapsto F_{J,K}(a,b):=(F^{(1)}_{J,K}(a,b),F^{(2)}_{J,K}(a,b))$ from $\{0,1,\dots,J\}\times \{0,1,\dots,K\}$ to itself by setting
\begin{align}
F^{(1)}_{J,K}(a,b) &=a+\min\{b,J-a\}-\min\{a,K-b\},\nonumber\\
F^{(2)}_{J,K}(a,b) &=b-\min\{b,J-a\}+\min\{a,K-b\},\label{fdef}
\end{align}
so that $((T\eta)_n,W_{n})=F_{J,K}(\eta_n,W_{n-1})$. Moreover, for a measure $\mu_{J,K} \in \mathcal{M}_{J,K}$, we set
\[\underline{r}(\mu_{J,K}):=\min_{a:\:\mu_{J,K}(a)>0} a,\]
\begin{equation}\label{rjkdef}
r(\mu_{J,K}):=\min_{a:\:\mu_{J,K}(a)>0}\min\left\{a,\sigma_J(a)\right\},
\end{equation}
where $\sigma_J(a):=J-a$. In practice, we will use the detailed balance equation \eqref{dbalance} to identify dual pairs of invariant i.i.d.\ measures.

\begin{thm}\label{iidthm} Fix $J,K\in\mathbb{N}\cup\{\infty\}$.  Let $\mu_{J,K}\in\mathcal{M}^{rev}_{J,K}$. It is then the case that $\mu_{J,K}\in\mathcal{I}_{J,K}^{inv}$ if and only if there exists a $\mu_{K,J} \in \mathcal{M}_{K,J}$ such that
\begin{equation}\label{dbalance}
\mu_{J,K}\times\mu_{K,J}\circ F_{J,K}^{-1}=\mu_{J,K}\times\mu_{K,J},
\end{equation}
$r(\mu_{J,K})=r(\mu_{K,J})$, and also $\underline{r}(\mu_{J,K})=\underline{r}(\mu_{K,J})$ when either $J=\infty$ or $K=\infty$. Moreover, if the above conditions hold, then $\mu_{K,J}=\mathcal{D}^\mu_{J,K}(\mu_{J,K})$ and $\mu_{K,J} \in \mathcal{M}_{K,J}^{rev}$.
\end{thm}

As an immediate corollary of this result, we find that there is a bijection between invariant i.i.d.\ measures under the duality map  $\mathcal{D}^{\mu}_{J,K}$.  We note that it is easy to check that $\mathcal{D}^\mu_{J,K}$ is not a bijection between $\mathcal{M}_{J,K}^{rev}$ and $\mathcal{M}_{K,J}^{rev}$ in general.

\begin{crl} The map $\mathcal{D}^{\mu}_{J,K}$ is a bijection between $\mathcal{I}_{J,K}^{inv}$ and $\mathcal{I}_{K,J}^{inv}$.
\end{crl}

As a simple consequence of our general arguments, we further obtain the ergodicity of invariant i.i.d.\ measures. For BBS(1,$\infty$), the result was previously derived in \cite{CKST}.

\begin{crl}\label{iidcor} Fix $J,K\in\mathbb{N}\cup\{\infty\}$. If $\mu_{J,K}\in\mathcal{I}_{J,K}^{inv}$, then $\mu_{J,K}^{\otimes \mathbb{Z}}$ is ergodic for $T_{J,K}$.
\end{crl}

As an application of our results concerning i.i.d.\ invariant measures we are able to derive the asymptotic speed of a tagged particle. To state our theorem in this direction precisely, we first need to assign an order to particles in a configuration $\eta$; we do this by assigning an order (left to right) to particles in each box, and then using the natural order of $\mathbb{Z}$ to induce an order on all particles. Moreover, we use a slightly different description of the local dynamics to that given above in that we suppose that the carrier collects and deposits particles in such a way that the particle order is preserved by the dynamics. In particular, to achieve this, when the carrier passes a location, it leaves the same number of particles as determined by \eqref{carrierdef}, but does so in a way to ensure that those left behind have a lower index in the order than those it transports onwards. In the case of the BBS($J$,$\infty$), to use the terminology of queueing theory, this is simply a first-in-first-out scheme. However, when $K<\infty$, the carrier might also swap balls in the box with balls it is carrying. Nonetheless, although the action on individual balls is different to that illustrated by Figure \ref{dynamicsfig}, the final configuration is the same; see Figure \ref{dynamicsfig1} for an example of the algorithm considered here. With this viewpoint, we then consider the progress of the particle that initially is the left-most particle located at a spatial location in $\mathbb{N}$; we write $X=(X_{J,K}(t))_{t\in\mathbb{Z}}$, where $X_{J,K}(t)$ is the location of the particle in question after $t$ evolutions of the BBS($J$,$K$). In particular, we are able to prove the following strong law of large numbers. Again, for BBS(1,$\infty$), the result was already known \cite{CKST}.

\begin{figure}[t]
\begin{center}
{\includegraphics[width=0.9\textwidth]{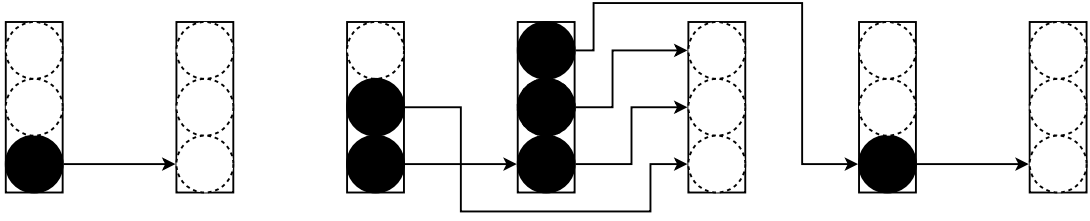}}
\vspace{-15pt}
\end{center}
\caption{Dynamics of the BBS($3$,$4$) under the algorithm used to study the tagged particle, starting from the same initial configuration as in Figure \ref{dynamicsfig}. NB. Within each box, balls are ordered from bottom to top. Note that, though the motion of individual balls is different, the resulting configuration is the same as the lower diagram in the latter figure.}\label{dynamicsfig1}
\end{figure}

\begin{thm}\label{speedthm} Fix $J,K\in\mathbb{N}\cup\{\infty\}$ with $J\neq K$. If $\mu_{J,K}\in\mathcal{I}_{J,K}^{inv}$ is such that $\mu_{J,K}(0)\neq 1$, then
\begin{equation}\label{speedconv}
\frac{X_{J,K}(t)}{t}\rightarrow \frac{m_{K,J}}{m_{J,K}},\qquad \mbox{$\mu_{J,K}^{\otimes \mathbb{Z}}$-a.s.},
\end{equation}
as $t\rightarrow\infty$, where $m_{J,K}:=\sum_xx\mu_{J,K}(x)$ and $m_{K,J}:=\sum_xx\mu_{K,J}(x)$, with $\mu_{K,J}:=\mathcal{D}^\mu_{J,K}(\mu_{J,K})$, are both constants taking values in $(0,\infty)$.
\end{thm}

\begin{rem} When $J=K$, the triviality of the dynamics mean that $X_{J,K}(t)-X_{J,K}(0)=t$ for each $t$, and so the same result is true with limit equal to 1 (even if $m_{J,K}=m_{K,J}$ is no longer finite).
\end{rem}

\begin{rem} Under the weaker assumptions that $P_{J,K}\in\mathcal{P}^{rev}_{J,K}$ is invariant and ergodic under both $\theta$ and $T_{J,K}$, then the obvious adaptation of the proof of Theorem \ref{speedthm} yields that the convergence of \eqref{speedconv} holds in probability under $P_{J,K}$, where $m_{J,K}:=\sum_xxP_{J,K}(\eta_0=x)$ and $m_{K,J}:=\sum_xxP_{J,K}(W_0=x)$ are both non-zero. (Since we are supposing that $J\neq K$, at most one of the moments is infinite, and so the limit is well-defined in $[0,\infty]$ in this case.)
\end{rem}

Before concluding our introduction, we highlight one further noteworthy aspect of our study, which is the description of the BBS($J$,$K$) dynamics when $J<K$ in terms of a Pitman-type transformation (cf. \cite{Pitman}) of a certain path-encoding of the particle configuration. Whilst such a viewpoint will not be as central to this study as it was in \cite{CKST}, it is still useful for providing an explicit description of the carrier, and identifying a subset of $\mathcal{C}^{inv}_{J,K}$ upon which the i.i.d.\ measures we consider are supported. See Subsection \ref{J<K} for details; a graphical depiction of the path encoding and its role in characterising the dynamics is presented below as Figure \ref{bbs35}.

\begin{rem}
BBS $(J,K)$ can be generalized to a model where the configuration takes values in $[0,J]^\mathbb{Z}$, and the carrier takes values in $[0,K]^\mathbb{Z}$, where $J,K\in(0,\infty]$, with the dynamics again being determined by \eqref{carrierdef} and \eqref{wdyn}. We expect that a number of the arguments of this article will readily extend to this setting. In particular, for $J<K$, it should be possible to describe the dynamics in terms of the path encoding, and use this to study the system in a similar way. It would be an interesting future project to explore to what extent the results we prove here can be adapted to the more general setting.
\end{rem}

The remainder of the article is organised as follows. Section \ref{carriersec} contains the deterministic part of the article, which is where we introduce the notion of a canonical carrier, and prove Theorem \ref{mainthm1}. In Section \ref{probdualsec} we study duality in a probabilistic sense, establishing Theorem \ref{dualthm} in particular. This is followed in Section \ref{iidsec} by an investigation of i.i.d.\ measures that are invariant for the box-ball system, which is where Theorem \ref{iidthm} and Corollary \ref{iidcor} are proved. Finally, we establish the speed theorem for the tagged particle, Theorem \ref{speedthm}, in Section \ref{speedsec}.

\section{Existence and uniqueness of carrier}\label{carriersec}

In this section we investigate the issues of whether a given configuration $\eta\in\mathcal{C}_{J,K}$ admits a carrier, and, if so, whether it is unique. We also define what it means for a carrier to be `canonical', which is a concept that will be crucial to our study. Considering the cases $J>K$, $J=K$, $J<K=\infty$ and $J<K<\infty$ separately, we completely characterise the sets of configurations for which a canonical carrier exists, see Propositions \ref{ch1}, \ref{ch2}, \ref{ch3} and \ref{ch4}, respectively. Furthermore, as the main conclusion of this section, we prove Theorem \ref{mainthm1}.

To begin with, we list some key basic properties of the map $(a,b)\mapsto F_{J,K}(a,b):=(F^{(1)}_{J,K}(a,b),F^{(2)}_{J,K}(a,b))$ that are easily checked from the definition at \eqref{fdef}, and which will be applied later. Again, the reader might find it helpful to refer to Figure \ref{bbs75} for an illustration of the various symmetries. Moreover, a reader wishing to check the properties, and certain steps in the subsequent arguments, might find Figure \ref{diagrams} useful.

\begin{figure}[t]
\begin{multicols}{2}
Diagrams are of form:
\[\xymatrix@C-15pt@R-15pt{ & F^{(1)}_{J,K}(a,b) & \\
            b \ar[rr]& & F^{(2)}_{J,K}(a,b)\\
             & a \ar[uu]&}\]

Case 1: $0\leq a+b\leq \min\{K,J\}$
\[\xymatrix@C-15pt@R-15pt{ & b & \\
            b \ar[rr]& & a\\
             & a \ar[uu]&}\]
\end{multicols}

\begin{multicols}{2}
Case 2(a): $J\leq a+b\leq K$
\[\xymatrix@C-15pt@R-15pt{ & J-a & \\
            b \ar[rr]& & b+2a-J\\
             & a \ar[uu]&}\]

Case 2(b): $K\leq a+b\leq J$
\[\xymatrix@C-15pt@R-15pt{ & 2b+a-K & \\
            b \ar[rr]& & K-b\\
             & a \ar[uu]&}\]
\end{multicols}

Case 3: $a+b\geq\max\{K,J\}$
\[\xymatrix@C-15pt@R-15pt{ & b+J-K & \\
            b\ar[rr]& & a+K-J\\
             & a \ar[uu]&}\]
\vspace{-20pt}
\caption{Summary of the output of $F_{J,K}$, as defined at \eqref{fdef}.}\label{diagrams}
\end{figure}

\begin{description}
  \item[Involution] For any $(a,b)\in \{0,1,\dots,J\}\times \{0,1,\dots,K\}$, it holds that
  \begin{equation}\label{involution}
  F_{J,K}\circ F_{J,K}(a,b)=(a,b).
  \end{equation}
  \item[Configuration-carrier duality] For any $(a,b)\in \{0,1,\dots,J\}\times \{0,1,\dots,K\}$, it holds that  $F^{(1)}_{J,K}(a,b)=F^{(2)}_{K,J}(b,a)$ and  $F^{(2)}_{J,K}(a,b)=F^{(1)}_{K,J}(b,a)$. Equivalently,
      \begin{equation}\label{dualityrel}
        \pi\circ F_{J,K}= F_{K,J}\circ\pi,
      \end{equation}
      where $\pi$ is the permutation map given by $\pi(a,b)=(b,a)$.
\end{description}

\begin{rem} Similarly to the remark following Theorem \ref{mainthm1}, there is an alternative presentation of the previous properties that also has its advantages (and which is used in the integrable systems literature). Indeed, suppose we had chosen to state the properties in terms of the map $\tilde{F}_{J,K}:=\pi\circ F_{J,K}$, then \eqref{involution} would become $\tilde{F}_{J,K}^{-1}=\pi\circ \tilde{F}_{J,K}\circ\pi$, and \eqref{dualityrel} is given by $\pi\circ \tilde{F}_{J,K}= \tilde{F}_{K,J}\circ\pi$; combining these yields $\tilde{F}_{J,K}^{-1}=\tilde{F}_{K,J}$.
\end{rem}

\begin{description}
  \item[Reducibility] If $\min\{J,K\}> 2r$
   for some $r\in\mathbb{N}$, then for any $(a,b) \in \{r,\dots,J-r\}\times\{r,\dots,K-r\}$ it holds that
      \begin{equation}\label{reducibility}
F_{J-2r,K-2r}(a-r,b-r)=\left(F^{(1)}_{J,K}(a,b)-r,F^{(2)}_{J,K}(a,b)-r\right).
\end{equation}
  \item[Empty box-ball duality] If $J,K<\infty$, then it holds that
  \begin{equation}\label{spaceball}
  (\sigma_J\times\sigma_K)\circ F_{J,K} = F_{J,K}\circ (\sigma_J\times\sigma_K),
  \end{equation}
    where $(\sigma_J\times\sigma_K)(a,b):=(\sigma_J(a),\sigma_K(b))$ with $\sigma_J(a):=J-a$ and $\sigma_K(b):=K-b$.
\end{description}

We next introduce formally a carrier, and what it means for this to be canonical.

\begin{dfn} For $\eta=(\eta_n)_{n\in\mathbb{Z}} \in \mathcal{C}_{J,K}$, we say $Y=(Y_n)_{n\in\mathbb{Z}} \in \{0,\dots,K\}^{\mathbb{Z}}$ is a \emph{BBS($J$,$K$) carrier} for $\eta$ if
\begin{equation}\label{carrier}
Y_n=F^{(2)}_{J,K}(\eta_n,Y_{n-1}),\qquad \forall n \in \mathbb{Z}.
\end{equation}
NB. We will simply say $Y$ is a \emph{carrier} for $\eta$ when it is clear which particular model is being considered.
\end{dfn}

As we will see below, a BBS($J$,$K$) carrier does not necessarily exist, nor is it unique if it does. Hence we introduce subsets of the configuration space $\mathcal{C}_{J,K}$ as follows:
\begin{align*}
\mathcal{C}_{J,K}^{\exists}&:=\left\{ \eta \in \mathcal{C}_{J,K}:\: \mbox{there exists a BBS($J$,$K$) carrier for $\eta$}\right\};\\
\mathcal{C}_{J,K}^{\exists !}&:=\left\{ \eta \in \mathcal{C}_{J,K}:\: \mbox{there exists a unique BBS($J$,$K$) carrier for $\eta$}\right\}.
\end{align*}
We will characterize these subsets in the following subsections (in Propositions \ref{ch1}, \ref{ch2}, \ref{ch3}, \ref{ch4}). Note that, for a given $\eta \in \mathcal{C}_{J,K}$ and $N \in \mathbb{Z}$, if $Y=(Y_n)_{-\infty}^N \in \{0,\dots,K\}^{\mathbb{Z}_{\le N}}$ satisfies $Y_n=F^{(2)}_{J,K}(\eta_n,Y_{n-1})$ for all $n\le N$, then $Y$ is uniquely extended to a carrier $Y=(Y_n) \in \{0,\dots,K\}^{\mathbb{Z}}$. Hence, the existence and the uniqueness of the carrier is a `tail' problem.

The subset $\mathcal{C}_{J,K}^{\exists !}$ seems a natural domain for the BBS($J$,$K$) dynamics; indeed, for $\eta\in\mathcal{C}_{J,K}^{\exists}$ and an associated carrier $Y$, one could define the related dynamics by setting
\[T^{Y}\eta_n:=F^{(1)}_{J,K}(\eta_n,Y_{n-1}),\]
and for $\eta\in\mathcal{C}_{J,K}^{\exists !}$, this uniquely determines the updated configuration. However, as is a consequence of Lemma \ref{keylemma} and our characterisation of $\mathcal{C}_{J,K}^{\exists!}$, there can exist configurations in $\mathcal{C}_{J,K}^{\exists!}$ that admit a particular form of degenerate behaviour that persists for all time. Moreover, the set $\mathcal{C}_{J,K}^{\exists !}$ excludes certain configurations for which a natural choice of carrier still exists. Specifically, as we show in Proposition \ref{ch3}, $\mathcal{C}_{J,\infty}^{\exists !}=\emptyset$ for any $J<\infty$, yet for a class of configurations in $\mathcal{C}_{J,\infty}$ one can still make sense of the dynamics by choosing a carrier appropriately. As described in \cite[Remark 2.11]{CKST} with regard to the BBS($1$,$\infty$) model in particular, our choice excludes the possibility of transporting particles into the system from $-\infty$.

Towards presenting our choice of domain for the dynamics, we first introduce an essential boundary of the carrier $Y$ for $\eta$ by setting
\begin{equation}\label{essbound}
B_Y\equiv B_{Y,\eta}:=\sup\left\{N \in \mathbb{Z}:\: \min\{J,K\}  \le \eta_{n}+Y_{n-1} \le \max\{J,K\},\:\forall n \le N\right\},
\end{equation}
with the convention that $\sup\emptyset =-\infty$. Note that if $\eta \in \mathcal{C}_{J,K}^{\exists}$ and $Y$ is a carrier for $\eta$, then $N \le B_Y$ is equivalent to
\begin{equation}\label{condeq0}
T^{Y}\eta_n=J-\eta_n,\qquad\forall n\le N,
\end{equation}
when $J<K$, and to
\begin{equation}\label{condeq}
Y_n=K-Y_{n-1},\qquad\forall n\le N,
\end{equation}
when $J>K$. We are now ready to introduce the notion of a canonical carrier for the BBS($J$,$K$), that excludes the behaviour described by \eqref{condeq0} and \eqref{condeq}.

\begin{dfn}\label{cancarr}
For $\eta=(\eta_n)_{n\in\mathbb{Z}} \in \mathcal{C}_{J,K}$, we say $Y=(Y_n)_{n\in\mathbb{Z}} \in \{0,\dots,K\}^{\mathbb{Z}}$ is a \emph{canonical BBS($J$,$K$) carrier} for $\eta$ if one of the following conditions hold:\\
(a) $J=K$ and $Y$ is a BBS($J$,$K$) carrier for $\eta$;\\
(b) $J \neq K$ and $Y$ is a BBS($J$,$K$) carrier for $\eta$ with $B_Y=-\infty$.
\end{dfn}

Given this definition, we introduce a corresponding subset of configurations, which we argue is a natural domain of the BBS($J$,$K$) dynamics by setting:
\[\mathcal{C}_{J,K}^{can}:=\left\{ \eta \in \mathcal{C}_{J,K}:\: \mbox{there exists a canonical BBS($J$,$K$) carrier for $\eta$}\right\}.\]
As we will see below in Propositions \ref{ch1}, \ref{ch3} and \ref{ch4}, neither $\mathcal{C}_{J,K}^{can} \subseteq \mathcal{C}_{J,K}^{\exists !}$ nor $\mathcal{C}_{J,K}^{can} \supseteq \mathcal{C}_{J,K}^{\exists !}$ hold in general. However, it is possible to check the following.

\begin{lmm}\label{unique}
For any $\eta \in \mathcal{C}_{J,K}^{can}$, a canonical BBS($J$,$K$) carrier exists uniquely. Moreover, if $W$ is the canonical carrier, then $W_{n}$ is a (measurable) function of $(\eta_{m})_{m\leq n}$.
\end{lmm}
\begin{proof} See Propositions \ref{ch1}, \ref{ch2}, \ref{ch3}, \ref{ch4} (and their proofs) below.
\end{proof}

Typically, for $\eta \in \mathcal{C}_{J,K}^{can}$, we will denote the unique associated canonical BBS($J$,$K$) carrier by $W=(W_n)_{n\in\mathbb{Z}}$, and define the BBS($J$,$K$) operator $T=T_{J,K}$ by setting
\[T\eta_n=F^{(1)}_{J,K}(\eta_n,W_{n-1}).\]
In our study of this map, the following property of non-canonical carriers will be central. The result makes precise the claim made above that the degenerate dynamics allowed by non-canonical carriers, as given by \eqref{condeq0} and \eqref{condeq}, persists for all time.

\begin{lmm}\label{keylemma}
If $\eta \in \mathcal{C}_{J,K}^{\exists}$ and $Y$ is a carrier for $\eta$ but not canonical, then the following statements hold. \\
(a) $T^{Y}\eta \notin \mathcal{C}_{J,K}^{can}$.\\
(b) Suppose either $J\leq K$ or $\infty>J>K$. If $(\eta^{(i)})_{i \in \mathbb{N}}$ and $(Y^{(i)})_{i \in \mathbb{N}}$ is a sequence such that $\eta^{(1)}=\eta$, $Y^{(1)}=Y$, and, for each $i$, $Y^{(i)}$ is a carrier for $\eta^{(i)}$ and $T^{Y^{(i)}}\eta^{(i)}=\eta^{(i+1)}$ for all $i$, then $\inf_{i \in \mathbb{N}}B_{Y^{(i)}} >-\infty$.
\end{lmm}
\begin{proof}
See proofs in the following subsections.
\end{proof}

With the above preparations in place, we are now in a position to study the reversibility of the BBS($J$,$K$) dynamics, as well as the duality between this system and BBS($K$,$J$). To this end, recall the notation for the spatially reversed configuration $R\eta$, where $\eta \in \mathcal{C}_{J,K}$, from \eqref{rdef}, and also the subset $\mathcal{C}_{J,K}^{rev}$ of configurations for which both $\eta$ and $R\eta$ admit a canonical carrier from \eqref{crevdef}. The following proposition establishes the claim from the introduction that if we define
\begin{equation}\label{tinv}
T^{-1}\eta:=RTR\eta
\end{equation}
for configurations $\eta$ such that $R\eta\in \mathcal{C}_{J,K}^{can}$, then $T^{-1}$ is the inverse of $T$ on $\mathcal{C}_{J,K}^{rev}$.

\begin{prp}\label{tinvprp} Suppose $\eta \in \mathcal{C}_{J,K}^{rev}$. It is then the case that $T^{-1}\eta, RT\eta \in \mathcal{C}_{J,K}^{can}$, and moreover $TT^{-1}\eta=T^{-1}T\eta=\eta$.
\end{prp}
\begin{proof} Let $\bar{W}_n:=W_{-n}$
where $W$ is the canonical carrier for $\eta$. From the involutive property of $F_{J,K}$, i.e.\ \eqref{involution}, it follows that $F^{(2)}_{J,K}(RT\eta_n,\bar{W}_{n-1})=\bar{W}_n$. In particular, this implies that $\bar{W}$ is a carrier for $RT\eta$. The same property further implies that $T^{\bar{W}}(RT\eta)=R\eta$. Hence, since $R\eta \in \mathcal{C}_{J,K}^{can}$, Lemma \ref{keylemma}(a) yields that $\bar{W}$ must in fact be a (and hence the) canonical carrier for $RT\eta$. Thus we have established $RT\eta \in \mathcal{C}_{J,K}^{can}$, and moreover $T^{-1}T\eta=RTRT\eta=R^2\eta=\eta$. In the same way, it is possible to check that if $V$ is the canonical carrier for $R\eta$ and we set $\bar{V}_n:=V_{-n}$, then $\bar{V}$ is the canonical carrier for $T^{-1}\eta$. So, $T^{-1}\eta \in \mathcal{C}_{J,K}^{can}$, and $TT^{-1}\eta =T^{\bar{V}}T^{-1}\eta=\eta$.
\end{proof}

We next prove Theorem \ref{mainthm1}. We recall the invariant set of configurations $\mathcal{C}_{J,K}^{inv}$ from \eqref{cinvdef}, and note that on $\mathcal{C}_{J,K}^{inv}$, $(T^tW_n)_{n \in \mathbb{Z}, t \in \mathbb{Z}}$ is well-defined, where $T^tW$ is the canonical carrier for $T^t \eta$. We also recall the duality map $\eta\mapsto\mathcal{D}_{J,K}(\eta)=((T^{t}W)_0)_{t \in \mathbb{Z}}$ that is defined on $\mathcal{C}_{J,K}^{inv}$ from \eqref{dualitymap}.

\begin{proof}[Proof of Theorem \ref{mainthm1}] (a) We first consider the case $J=K$. One then has that $F_{J,K}(a,b)=(b,a)$, from which it follows that $T^tW_{n}=T^t\eta_n=\eta_{n-t}$. (Clearly the carrier exists and is unique, and it is canonical by Definition \ref{cancarr}.) Thus we obtain that $\mathcal{D}_{J,K}(\eta)=\theta^{-1} R\eta$, where $\theta$ is the left-shift, as defined in the statement of the theorem. Since $\theta^{-1}R=R\theta$, $\theta T=T\theta$ and $TR=RT^{-1}$, we thus have that $T^t\mathcal{D}_{J,K}(\eta)=\theta^{-1}RT^{-t}\eta$ and $RT^t\mathcal{D}_{J,K}(\eta)=\theta T^{-t}\eta$, which are both in $\mathcal{C}_{K,J}^{can}$ by assumption (and the fact that the latter set is invariant under spatial shifts), i.e.\ $\mathcal{D}_{J,K}(\eta)\in\mathcal{C}_{K,J}^{inv}$. Moreover, it is clear from the identity $F_{J,K}(a,b)=(b,a)$ that the carrier for $(T^tW_n)_{t\in\mathbb{Z}}=(\eta_{n-t})_{t\in\mathbb{Z}}$ is $(\eta_{n-t})_{t\in\mathbb{Z}}=(T^{t+1}\eta_{n+1})_{t\in\mathbb{Z}}$, as desired.

Now suppose $J>K$ or $J<K<\infty$. By the duality of the model function, as stated at \eqref{dualityrel}, it holds that $F^{(2)}_{K,J}(T^tW_{n},T^t\eta_{n+1})=F^{(1)}_{J,K}(T^t\eta_{n+1},T^tW_{n})=T^{t+1}\eta_{n+1}$. Hence, for each $n$, $(T^{t+1}\eta_{n+1})_{t \in  \mathbb{Z}}$ is a BBS($K$,$J$) carrier for $(T^{t}W_{n})_{t \in \mathbb{Z}}$, and moreover, it holds that $T_{K,J}^{(T^{t+1}\eta_{n+1})_{t\in\mathbb{Z}}}\left((T^tW_{n})_{t\in\mathbb{Z}}\right)=(T^tW_{n+1})_{t\in\mathbb{Z}}$. We need to show these are canonical carriers. To this end, observe that by assumption, for each $t$, it is not possible to find an $N\in\mathbb{Z}$ such that
\begin{equation}\label{e1}
\min\{J,K\} \le T^t\eta_n+T^tW_{n-1} = T^{t+1}\eta_n+T^tW_n  \le \max\{J,K\},\qquad \forall n \le N.
\end{equation}
(NB. The central equality is an application of \eqref{consmass}.) Similarly, by considering the carrier of the reversed configuration $RT^t\eta$ as in the proof of Proposition \ref{tinvprp}, for each $t$, it is not possible to find an $N\in\mathbb{Z}$ such that
\begin{equation}\label{e2}
\min\{J,K\} \le T^{t+1}\eta_n+T^tW_n   =T^t\eta_n+T^tW_{n-1} \le \max\{J,K\},\qquad \forall n \ge N.
\end{equation}
Moreover, note that if for some $n \in \mathbb{Z}$, $(T^{t+1}\eta_{n+1})_{t \in  \mathbb{Z}}$ is not a canonical carrier, then Lemma \ref{keylemma}(b) implies that $A:=\inf_{m \ge n} B_{(T^{t+1}\eta_{m+1})_{t\in\mathbb{Z}}} >-\infty$. However, a consequence of this is that $\min\{J,K\} \le T^{t}\eta_{m+1}+T^tW_m \le \max\{J,K\}$ for all $t \le A$, $m \ge n$, which contradicts \eqref{e2}. Hence $(T^{t+1}\eta_{n+1})_{t \in \mathbb{Z}}$ is the canonical BBS($K$,$J$) carrier for $(T^{t}W_{n})_{t \in \mathbb{Z}}$ for each $n$. In the same way, but appealing to \eqref{e1} in place of \eqref{e2}, $(T^{1-t}\eta_n)_{t \in \mathbb{Z}}$ is shown to be the canonical BBS($K$,$J$) carrier for $R((T^{t}W_n)_{t \in \mathbb{Z}})=(T^{1-t}W_n)_{t \in \mathbb{Z}}$. Hence $\mathcal{D}_{J,K}(\eta)\in\mathcal{C}_{K,J}^{inv}$, as is required to complete the proof of part (a).

(b) This follows easily from part (a).
\end{proof}

To conclude this part of our discussion, we present a result concerning the duality of carriers and canonical carriers under swapping the roles of empty boxes and balls, i.e.\ \eqref{spaceball}, in the case when both box and carrier capacities are finite. Since the proof of the result is straightforward, we omit this. NB. For a configuration $\eta\in\mathcal{C}_{J,K}$, we define $\sigma_J\eta:=(\sigma_J\eta_n)_{n\in\mathbb{Z}}$ (where, as defined above, $\sigma_J(a):=J-a$), and for a subset $\mathcal{C}\subseteq\mathcal{C}_{J,K}$, we write $\sigma_J\mathcal{C}:=\{\sigma_J\eta:\:\eta\in\mathcal{C}\}$. For a carrier $Y$, we similarly define $\sigma_KY$.

\begin{prp}\label{bsdual} Suppose $J,K < \infty$.\\
(a) For $\eta  \in \mathcal{C}_{J,K}$, $Y$ is a carrier for $\eta$ if and only if $\sigma_K Y$ is a carrier for $\sigma_J\eta$.\\
(b) For $\eta  \in \mathcal{C}_{J,K}$, $Y$ is a canonical carrier for $\eta$ if and only if $\sigma_K Y$ is a canonical carrier for $\sigma_J\eta$.\\
(c) It holds that $\mathcal{C}_{J,K}^{\exists}= \sigma_J \mathcal{C}_{J,K}^{\exists}$, $\mathcal{C}_{J,K}^{\exists !} = \sigma_J \mathcal{C}_{J,K}^{\exists !}$, and also $\mathcal{C}_{J,K}^{can}= \sigma_J \mathcal{C}_{J,K}^{can}$.
\end{prp}

In the following subsections, we characterise $\mathcal{C}_{J,K}^{\exists}$, $\mathcal{C}_{J,K}^{\exists !}$ and $\mathcal{C}_{J,K}^{can}$ in the cases $J>K$ (Subsection \ref{case1}), $J=K$ (Subsection \ref{case2}), $J<K=\infty$ (Subsection \ref{case3a}) and $J<K<\infty$ (Subsection \ref{case3b}). For each case, we also give the proofs of Lemmas \ref{unique} and \ref{keylemma}, and either characterise or describe a certain subset of $\mathcal{C}^{inv}_{J,K}$.

\subsection{Case 1: $J>K$}\label{case1}

In this subsection we assume $J>K$, and start by making some simple observations about the possible outcomes of the model. The following lemma is easily deduced from the definition of $F_{J,K}$, see also Figure \ref{diagrams}.

\begin{lmm}\label{0lmm} Suppose $J>K$. If $\eta \in \mathcal{C}_{J,K}^{\exists}$, and $Y$ is a carrier for $\eta$, then $\eta_n=0$ implies $Y_n=0$, and $\eta_n=J$ implies $Y_n=K$.
\end{lmm}

From the preceding lemma, we can readily show that on the following set a unique carrier exists: $\mathcal{C}_{J,K}^{0}:=\{\eta \in \mathcal{C}_{J,K}:\:\liminf_{n \to -\infty} \min\left\{ \eta_n, J-\eta_n\right\}=0\}$.

\begin{lmm}\label{c0lem} If $J>K$, then $ \mathcal{C}_{J,K}^{0} \subseteq \mathcal{C}_{J,K}^{\exists !}$.
\end{lmm}
\begin{proof} If $\eta \in \mathcal{C}_{J,K}^{0}$, then there exists a decreasing sequence $(N_k)_{k\in\mathbb{N}}$ of integers with $\eta_{N_k}\in\{0,J\}$. Thus one can uniquely define a carrier by applying Lemma \ref{0lmm} to deduce the values of $(Y_{N_k})_{k\in\mathbb{N}}$, and defining the values of $Y_n$ for $n\neq N_k$ inductively by \eqref{carrier}. Hence a carrier exists and it is unique.
\end{proof}

Towards understanding the general situation, for $r \in \mathbb{Z}_+$ satisfying $r \le \frac{J}{2}$, we let
\begin{equation}\label{crdef}
\mathcal{C}_{J,K}^{r}:=\left\{ \eta \in \mathcal{C}_{J,K}:\:\liminf_{n \to -\infty} \min\left\{ \eta_n,J-\eta_n \right\}=r \right\}.
\end{equation}
Note that $\mathcal{C}_{J,K}=(\cup_{r=0}^{\lfloor\frac{J}{2}\rfloor}\mathcal{C}_{J,K}^{r})\cup \mathcal{C}^\infty_{J,K}$, where $\mathcal{C}^\infty_{J,K}:=\{ \eta \in \mathcal{C}_{J,K}:\:\liminf_{n \to -\infty} \eta_n=\infty\}$. In the main result of the section, Proposition \ref{ch1}, these sets will appear in our characterisation of $\mathcal{C}_{J,K}^{\exists}$, $\mathcal{C}_{J,K}^{\exists !}$ and $\mathcal{C}_{J,K}^{can}$. Prior to this, we proceed to present a sequence of preparatory lemmas.

\begin{lmm}\label{basiclmm} Suppose $J>K$. Let $\eta \in \mathcal{C}_{J,K}^{\exists}$ and $Y$ be a carrier for $\eta$. If $\eta_n \in \{r,r+1,\dots,J-r\}$ for some $r \in \{0,1,\dots,\lfloor\frac{J}{2}\rfloor\}$, and $\min\{Y_n,K-Y_n\} <r$, then $\min\{Y_{n-1},K-Y_{n-1}\} \le \min\{Y_n,K-Y_n\}$. Moreover, $Y_{n-1} = Y_n$ holds if and only if $Y_{n-1}=Y_{n}=\frac{K}{2}$.
\end{lmm}
\begin{proof} Suppose $Y_n <r$ and $\eta_n \ge r$. From \eqref{fdef}, we then have that
\[r>Y_n=Y_{n-1}-\min\{Y_{n-1},J-\eta_n\}+\min\{\eta_n,K-Y_{n-1}\}\geq \min\{r,K-Y_{n-1}\},\]
which implies $K-Y_{n-1}\leq Y_n$. In the same way, if $K-Y_n <r$ and $\eta_n \le J-r$, then $Y_{n-1} \le  K-Y_n$. The first part of the lemma follows.

Suppose $Y_n <r$ and $Y_n=Y_{n-1}$. It then follows from the assumptions and \eqref{fdef} that $Y_n=\min\{Y_n,J-\eta_n\}=\min\{\eta_n,K-Y_n\}=K-Y_n$, and so $Y_n=\frac{K}{2}$. The result follows similarly if $K-Y_n <r $ and $Y_n=Y_{n-1}$.
\end{proof}

\begin{lmm}\label{chcan}
Suppose $J>K$. For $\eta \in \mathcal{C}_{J,K}^{\exists}$, suppose that there exist $r \in \{0,1,\dots,\lfloor\frac{J}{2}\rfloor\}$ and $N \in \mathbb{Z}$ such that $\eta_n \in \{r,r+1,\dots,J-r\}$ for all $n \le N$. If $Y$ is a carrier for $\eta$ such that
\begin{equation}\label{ccc1}
\min\{Y_{n_0}, K-Y_{n_0}\} <r
\end{equation}
for some $n_0 \le N$, then $Y$ is not canonical.
\end{lmm}
\begin{proof}
From Lemma \ref{basiclmm}, \eqref{ccc1} implies that $\min\{Y_{n_0-n},K-Y_{n_0-n}\}$ is non-increasing for $n\geq 0$, and so $\lim_{n \to -\infty}\min\{Y_n,K-Y_n\}=q$ for some $q\in\mathbb{Z}_+$. Namely, there exists $N_1 \le N$ such that for all $n \le N_1$, $Y_n=q$ or $Y_n=K-q$. Now, if $Y_n=Y_{n-1}$ for some $n \le N_1$, then Lemma \ref{basiclmm} tells us that $Y_n=Y_{n-1}=\frac{K}{2}$, and so $q=\frac{K}{2}$. Hence $Y_n=\frac{K}{2}=K-Y_{n-1}$ for all $n \le N_1$, and so $Y$ is not canonical (recall \eqref{condeq}). Similarly, if $Y_n \neq Y_{n-1}$ for all $n \le N_1$, then $Y_n=K-Y_{n-1}$ for all $n \le N_1$, and so $Y$ is not canonical (again by \ref{condeq}).
\end{proof}

\begin{lmm}\label{le2r}
Suppose $J>K$, and let $r\in\{0,1,\dots,\lfloor\frac{J}{2}\rfloor\}\cup\{\infty\}$.\\
(a) If $K \le 2r$, then $\mathcal{C}_{J,K}^{r} \subseteq \mathcal{C}_{J,K}^{\exists}$ and $\mathcal{C}_{J,K}^{r} \cap \mathcal{C}_{J,K}^{can} = \emptyset$.\\
(b) If $K<2r$, then $\mathcal{C}_{J,K}^{r} \cap \mathcal{C}_{J,K}^{\exists !} = \emptyset$.
\end{lmm}
\begin{proof} Throughout the proof, we suppose that $K\leq 2r$. Moreover, we start by considering the case when $r\in\{0,1,\dots,\lfloor\frac{J}{2}\rfloor\}$. For $\eta \in \mathcal{C}_{J,K}^{r}$, there exists an $N$ such that for all $n \le N$, $r\leq \eta_n \leq J-r$. Hence, taking $Y_{2n}=\min\{r,K\}$ and $Y_{2n+1}=K-\min\{r,K\}$ for $n$ satisfying $2n,2n+1 \le N$,
one can check from the definition at \eqref{fdef} (see also Figure \ref{diagrams}) that $F^{(2)}_{J,K}(Y_{n-1}, \eta_n)=Y_{n}$ holds for all $n \le N$. Indeed, for this choice of $Y$, one has that $K \le Y_{n-1}+\eta_n \le J$ for all $n \le N$. Since $Y$ can be extended to a carrier by applying \eqref{carrier}, it follows that $\eta\in\mathcal{C}_{J,K}^{\exists}$. Furthermore, if $K <2r$, then by taking $Y_{2n}=K-\min\{r,K\}$ and $Y_{2n+1}=\min\{r,K\}$ for all $n$ satisfying $2n,2n+1 \le N$, we have another carrier. Hence, in this case, $\eta \notin \mathcal{C}_{J,K}^{\exists !}$. To complete the proof when $r\in\{0,1,\dots,\lfloor\frac{J}{2}\rfloor\}$, we need to check that when $K\leq 2r$, any $\eta \in  \mathcal{C}_{J,K}^{r}$ is not an element of $\mathcal{C}_{J,K}^{can}$. From Lemma \ref{chcan}, if $Y$ is a canonical carrier for $\eta$, then $\min\{Y_n,K-Y_n\} \ge r$ for all $n \le N$. However, since $K \le 2r$, this holds if and only if $K=2r$ and $Y_n=r$ for all $n \le N$. Given this implies $K \le Y_{n-1}+\eta_n \le J$ for all $n \le N$, we have reached a contradiction. Thus $\eta \notin \mathcal{C}_{J,K}^{can}$.

The result for $r=\infty$ is proved in a similar fashion. The only difference is that we now start by noting that for $\eta\in C_{\infty,K}^\infty$, there exists an $N$ such that for all $n \le N$, $\eta_n \geq K$. (Also, $C_{J,K}^\infty=\emptyset$ for $J<\infty$.)
\end{proof}

\begin{lmm}\label{>2r} Suppose $J>K>2r$ for some $r\in\{0,1,\dots,\lfloor\frac{J}{2}\rfloor\}$.\\
(a) For each $\eta \in  \mathcal{C}_{J,K}^{r}$, there exists a unique carrier $Y$ such that
\begin{equation}\label{geqr}
\liminf_{n \to -\infty}\min\{Y_n,K-Y_n\} \ge r,
\end{equation}
and hence $\mathcal{C}_{J,K}^{r} \subseteq \mathcal{C}_{J,K}^{\exists}$. Moreover, this carrier satisfies
\begin{equation}\label{equalsr}
\liminf_{n \to -\infty}\min\{Y_n,K-Y_n\} = r,
\end{equation}
and if $J=\infty$, then
\begin{equation}\label{equalsr2}
\liminf_{n \to -\infty}Y_n = r.
\end{equation}
(b) For $\eta \in  \mathcal{C}_{J,K}^{r}$, it holds that $\eta \in \mathcal{C}_{J,K}^{can}$ if and only if the carrier of part (a) is canonical. Moreover, the canonical carrier is unique if it exists.
\end{lmm}
\begin{proof} Let $K > 2r$ and $\eta \in  \mathcal{C}_{J,K}^{r}$. There then exists an $N \in \mathbb{Z}$ such that $\eta_n \in \{r,\dots, J-r\}$ for all $n \le N$. Hence $(\tilde{\eta}_n)_{n \le N}:=(\eta_n-r)_{n \le N}$ satisfies $\tilde{\eta}_n \in \{0,\dots,J-2r\}$ for all $n \le N$ and $\liminf_{n \to -\infty}\min\{\tilde{\eta}_n,J-2r-\tilde{\eta}_n\} =0$. Therefore, by Lemma \ref{c0lem}, there exists a unique BBS($J-2r,K-2r$) carrier $\tilde{Y} \in \{0,\dots, K-2r\}$ for $\tilde{\eta}$ such that $F^{(2)}_{J-2r,K-2r}(\tilde{\eta}_n,\tilde{Y}_{n-1})=\tilde{Y}_n$ for all $n \le N$. (So that $\tilde{Y}$ is well-defined for all $n$, one may assume we have extended $\tilde{\eta}_n$ to an element of $\mathcal{C}_{J-2r,K-2r}$, though how this is done is unimportant for this proof.) By the reducibility property \eqref{reducibility}, it follows that if we define $Y_n=\tilde{Y}_n+r$ for $n\leq N$, then $F^{(2)}_{J,K}(\eta_n,Y_{n-1})=Y_n$ for all $n \le N$. By way of construction, $Y_n$ satisfies \eqref{equalsr} and \eqref{equalsr2}. Moreover, if there is another carrier $Y'$ for $\eta$ satisfying \eqref{geqr} (with $Y$ replaced by $Y'$), then $\tilde{Y}':=Y'-r$ satisfies $F^{(2)}_{J-2r,K-2r}(\tilde{\eta}_n,\tilde{Y}'_{n-1})=\tilde{Y}'_n$ for enough small $n$. However, the uniqueness of the carrier for $\tilde{\eta}$ implies $\tilde{Y}=\tilde{Y}'$, and so $Y=Y'$. Thus we have proved (a). Part (b) follows obviously from part (a) and Lemma \ref{chcan}.
\end{proof}

For our characterization of the sets $\mathcal{C}_{J,K}^{\exists !}$ and $\mathcal{C}_{J,K}^{can}$ in Proposition \ref{ch1}, we introduce a further family of subsets of the configuration space, as given by
\begin{equation}\label{cralt}
\mathcal{C}_{J,K}^{r,alt}:=\left\{\eta \in \mathcal{C}_{J,K}:\:\max_{i\in\{0,1\}}  \liminf_{n \to -\infty}  \mathbf{1}_{\{ \eta_{2n+i} \ge K-r,\:\eta_{2n+1+i}\le J-K+r \}}=1\right\}
\end{equation}
for $r\in \{0,1,\dots,K\}$. Note that $\mathcal{C}_{J,K}^{0,alt} \subseteq \mathcal{C}_{J,K}^{1,alt} \subseteq \cdots \subseteq \mathcal{C}_{J,K}^{K,alt}=\mathcal{C}_{J,K}$.

\begin{lmm}\label{chunique}
Suppose $J>K\geq 2r$, where $r\in\{0,1,\dots,\lfloor\frac{J}{2}\rfloor\}$. Let $\eta \in \mathcal{C}_{J,K}^{r}$.\\
(a) It holds that $\eta \notin \mathcal{C}_{J,K}^{can}$ if and only if $\eta \in \mathcal{C}_{J,K}^{r,alt}$.\\
(b) If $r \ge 1$, then $\eta \notin \mathcal{C}_{J,K}^{\exists !}$ if and only if $\eta \in \mathcal{C}_{J,K}^{r-1,alt}$.
\end{lmm}
\begin{proof}
Let $\eta \in \mathcal{C}_{J,K}^{r}$ and $K \ge 2r$. For the case $K=2r$, we already showed in Lemma \ref{le2r} that $\eta \notin \mathcal{C}_{J,K}^{can}$. On the other hand, since for suitably small $n$ we have that $K-r = r\leq \eta_n\leq J-r=J-K+r$, we have that $\eta \in \mathcal{C}_{J,K}^{r,alt}$. Hence part (a) holds. Next, suppose $K >2r$. By Lemma \ref{>2r}, $\eta \notin \mathcal{C}_{J,K}^{can}$ if and only if the carrier satisfying \eqref{equalsr} is not canonical. Moreover, by \eqref{condeq}, $\eta \notin \mathcal{C}_{J,K}^{can}$ is equivalent to $Y_n=K-Y_{n-1}$ for small $n$. Hence, $\eta \notin \mathcal{C}_{J,K}^{can}$ if and only if the carrier satisfying \eqref{equalsr} is given by, for some $i\in\{0,1\}$, $Y_{2n+i}=r$, $Y_{2n+1+i}=K-r$ for all small $n$. Since $Y$ is a non-canonical carrier for $\eta$ if and only if $K \le \eta_{n}+Y_{n-1} \le J$ for small $n$, we obtain that the condition $\eta \notin \mathcal{C}_{J,K}^{can}$ is equivalent to there existing an $i\in\{0,1\}$ such that $K-r \le \eta_{2n+i} \le J-r$ and $r \le \eta_{2n+1+i} \le J-K+r$ for small $n$. Since we are assuming $\eta \in \mathcal{C}_{J,K}^r$, the latter condition holds if and only if $\eta \in \mathcal{C}_{J,K}^{r,alt}$. Hence we have completed the proof of part (a).

We now look to prove (b). In the case $K=2r$, from the proof of Lemma \ref{le2r}, there exists a carrier satisfying $Y_n=r$ for small $n$. If there exists another carrier, then it must satisfy $\liminf_{n \to -\infty}\min\{Y_n,K-Y_n\} <r$. Furthermore, in the case $K >2r$, by Lemma \ref{>2r}, if $\eta \notin \mathcal{C}_{J,K}^{\exists !}$, then there must be a carrier such that $\liminf_{n \to -\infty}\min\{Y_n,K-Y_n\} <r$. Hence, in both cases $\eta \notin \mathcal{C}_{J,K}^{\exists !}$ is equivalent to there existing a carrier $Y$ such that $\liminf_{n \to -\infty}\min\{Y_n,K-Y_n\} <r$. If the latter condition holds for some carrier $Y$, then Lemma \ref{chcan} (and \eqref{condeq}) yields that $Y_n=K-Y_{n-1}$ for small $n$, and in particular there exists $q\in\{0,1,\dots,r-1\}$ such that,  for some $i\in\{0,1\}$, $Y_{2n+i}=q$ and $Y_{2n+1+i}=K-q$ hold for all small $n$. Such a $Y$ is a carrier for $\eta$ if and only if $K \le \eta_{n}+Y_{n-1} \le J$ for small $n$. Thus we arrive at the conclusion that $\eta \notin \mathcal{C}_{J,K}^{\exists !}$ if and only if there exists an $i\in\{0,1\}$ such that $K-q \le \eta_{2n+i} \le J-q$ and $q \le \eta_{2n+1+i} \le J-K+q$ for small $n$. Since $\eta \in \mathcal{C}_{J,K}^r$, this holds if and only if $\eta \in \mathcal{C}_{J,K}^{q,alt}$. Hence, $\eta \notin \mathcal{C}_{J,K}^{\exists !}$ is equivalent to $\eta \in \cup_{q=0}^{r-1} \mathcal{C}_{J,K}^{q,alt}$, but $\mathcal{C}_{J,K}^{q,alt}$ is increasing in $q$, and so this is equivalent to $\eta \in \mathcal{C}_{J,K}^{r-1,alt}$.
\end{proof}

\begin{prp}\label{ch1}
Suppose $J>K$. It then holds that:\\
(a) $\mathcal{C}_{J,K}^{\exists} =\mathcal{C}_{J,K}$;\\
(b) $\mathcal{C}_{J,K}^{\exists !}= \cup_{r =0}^{\lfloor\frac{K}{2}\rfloor} (\mathcal{C}_{J,K}^r \cap (\mathcal{C}_{J,K}^{r-1,alt})^c)$;\\
(c) $\mathcal{C}_{J,K}^{can}= \cup_{r =0}^{\lfloor\frac{K}{2}\rfloor} (\mathcal{C}_{J,K}^r \cap (\mathcal{C}_{J,K}^{r,alt})^c)$,\\
with convention that $\mathcal{C}_{J,K}^{-1,alt}=\emptyset$. In particular, $\mathcal{C}_{J,K}^{can} \subseteq  \mathcal{C}_{J,K}^{\exists !}$.
\end{prp}
\begin{proof} Part (a) follows from Lemmas \ref{le2r} and \ref{>2r}. For part (b), first note that
\[\mathcal{C}_{J,K}^{\exists !}=\left(\cup_{r=0}^{\lfloor\frac{J}{2}\rfloor}
\left(\mathcal{C}_{J,K}^{\exists !}\cap
\mathcal{C}_{J,K}^{r}\right)\right)\cup\left(\mathcal{C}_{J,K}^{\exists !}\cap \mathcal{C}^\infty_{J,K}\right).\]
From Lemma \ref{c0lem}, $\mathcal{C}_{J,K}^{\exists !} \cap \mathcal{C}_{J,K}^0=\mathcal{C}_{J,K}^0$; from Lemma \ref{chunique},
$\mathcal{C}_{J,K}^{\exists !} \cap \mathcal{C}_{J,K}^r=\mathcal{C}_{J,K}^r \cap (\mathcal{C}_{J,K}^{r-1,alt})^c$ for $1\leq r \leq \lfloor\frac{K}{2}\rfloor$; and from Lemma \ref{le2r}, $\mathcal{C}_{J,K}^{\exists !} \cap \mathcal{C}_{J,K}^r=\emptyset$ for $r \ge \lfloor\frac{K}{2}\rfloor+1$. This establishes (b), and part (c) is shown in a similar way.
\end{proof}

\begin{proof}[Proof of Lemma \ref{unique} in the case $J>K$.]
Since $\mathcal{C}_{J,K}^{can} \subseteq  \mathcal{C}_{J,K}^{\exists !}$, the result is clear.
\end{proof}

\begin{proof}[Proof of Lemma \ref{keylemma} in the case $J>K$.]
(a) Suppose $\eta \in \mathcal{C}_{J,K}^{\exists}$ and $Y$ is a carrier for $\eta$ but not canonical. By \eqref{condeq}, there exists a $q\in\{0,1,\dots,\lfloor\frac{K}{2}\rfloor\}$ and $i\in\{0,1\}$ such that $Y_{2n+i}=q$ and $Y_{2n+1+i}=K-q$ for small $n$. Moreover, $K \le \eta_{n}+Y_{n-1} \le J$ for small $n$, and, by \eqref{consmass}, this implies $K \le T^Y\eta_{n}+Y_{n} \le J$. Hence $K-q \le T^Y\eta_{2n+i} \le J-q$ and $q \le T^Y\eta_{2n+1+i} \le J-K+q$ for small $n$. Since $q \le K-q$ and $J-K+q \le J-q$ by assumption, this implies $T^Y\eta \in \mathcal{C}_{J,K}^{q} \cap \mathcal{C}_{J,K}^{q,alt}$. Thus, by Lemma \ref{chunique}, $T^Y\eta \notin \mathcal{C}_{J,K}^{can}$.

(b) For $\eta \in \mathcal{C}_{J,K}$, $q\in\{0,1,\dots,\lfloor\frac{K}{2}\rfloor\}$, $i\in\{0,1\}$, let
\[B_{q,i}\equiv B_{q,i,\eta}:=\sup\left\{N \in \mathbb{Z}:\:\begin{array}{c}
                                                              \eta_{2n+i}\in[q,J-K+q],\:  \eta_{2n+1+i}\in[K-q, J-q], \\
                                                              \mbox{whenever }2n+i,2n+1+i\le N
                                                            \end{array}\right\}.\]
Note that $B_{q,i} >-\infty$ if and only if $\eta$ has a carrier such that $Y_{2n+i}=q$, $Y_{2n+1+i}=K-q$ for $2n+i,2n+1+i \le B_{q,i}$. We further introduce, for any $c\in\{0,1,\dots,J\}$ and $i\in\{0,1\}$,
\[B_{c,i}^{alt}\equiv B_{c,i,\eta}^{alt}:=\sup\left\{N \in \mathbb{Z}:\: \eta_{2n+i} \le J-c,\:c \le \eta_{2n+1+i},\:\forall 2n+i,2n+1+i \le N\right\}.\]
Note that $B_{q,i}=\min\{B_{K-q,i}^{alt}, B_{q,1-i}^{alt}\}$ for  $q\in\{0,1,\dots,\lfloor\frac{K}{2}\rfloor\}$ and $i\in\{0,1\}$.

Suppose now $\eta\in\mathcal{C}_{J,K}$ has a carrier $Y$ that is not canonical. There then exists some $q \in\{0,1,\dots,\lfloor\frac{K}{2}\rfloor\}$ and $i \in \{0,1\}$ such that $B_{q,i} >-\infty$. For such $\eta$, let
\begin{equation}\label{beta}
B_{\eta}:=\inf\left\{ B_{c,i}^{alt}:\:c\in\{0,1,\dots,J\},\: i\in\{0,1\},\:B_{c,i}^{alt} > -\infty\right\}.
\end{equation}
Since we are assuming $J<\infty$, we have that $B_{\eta}>-\infty$. Moreover, note that $B_{Y,\eta} \ge B_{\eta}$ for any carrier $Y$ for $\eta$. Thus to establish the result it will be sufficient to check that $B_{T^Y\eta} \ge B_{\eta}$; indeed, it then holds that $\inf_{i\in\mathbb{N}}B_{Y^{(i)}}\ge \inf_{i\in\mathbb{N}}B_{\eta^{(i)}} \ge B_{\eta}>-\infty$. To prove $B_{T^Y\eta} \ge B_{\eta}$, we let $\tilde{B}_{c,i}^{alt}:=B_{c,i,T^Y\eta}^{alt}$, and will show that either $\tilde{B}_{c,i}^{alt}=-\infty$ or $\tilde{B}_{c,i}^{alt} \ge B_{\eta}$ for any $c,i$. (NB.\ Since we assume $T^Y\eta$ has a carrier and, by (a), it is not a canonical carrier, we know that there exists some $q\in\{0,1,\dots,\lfloor \frac{K}{2}\rfloor\}$ and $i \in \{0,1\}$ such that $\tilde{B}_{q,i} >-\infty$.) Now, since $Y$ is not canonical, there exists $q\in\{0,1,\dots,\lfloor\frac{K}{2}\rfloor\}$ and $i\in\{0,1\}$ such that $Y_{2n+i}=q$, $Y_{2n+1+i}=K-q$ for $2n+i,2n+1+i\le B_{q,i}$. Then, $T^{Y}\eta_{2n+i}=\eta_{2n+i}+K-2q$ and $T^{Y}\eta_{2n+1+i}=\eta_{2n+1+i}-K+2q$. Hence
\begin{align*}
\lefteqn{\min\left\{\tilde{B}_{c,i}^{alt},B_{q,i}\right\} =\min\left\{\sup\left\{N \in \mathbb{Z}:\:\begin{array}{c}
T^Y\eta_{2n+i} \le J-c,\: c \le T^Y\eta_{2n+1+i},\\ \forall 2n+i,2n+1+i\le N\end{array}\right\},B_{q,i}\right\}}\\
&=\min\left\{\sup\left\{N \in \mathbb{Z}:\:\begin{array}{c}
                                   \eta_{2n+i} \le J-(c+K-2q),\:c+K-2q \le \eta_{2n+1+i},\\
                                   \forall 2n+i,2n+1+i \le N
                                \end{array}\right\},B_{q,i}\right\},
\end{align*}
i.e.\ $\min\{\tilde{B}_{c,i}^{alt},B_{q,i}\}=\min\{B_{c+K-2q,i}^{alt},B_{q,i}\}$. It follows that if $c+K-2q \le J$ and $B_{c+K-2q,i}^{alt}>-\infty$, then $\tilde{B}_{c,i}^{alt}\geq B_{\eta}$, and otherwise, $\tilde{B}_{c,i}^{alt}=-\infty$. Thus the proof is complete.
\end{proof}

We next show that in this case $\mathcal{C}_{J,K}^{rev}=\mathcal{C}_{J,K}^{inv}$.

\begin{lmm}\label{12345} Suppose $J>K$. For $\eta \in \mathcal{C}_{J,K}^{can}$, $T\eta \in \mathcal{C}_{J,K}^{can}$. Moreover, $\mathcal{C}_{J,K}^{rev}=\mathcal{C}_{J,K}^{inv}$.
\end{lmm}
\begin{proof} We will start by showing that if $\eta \in \mathcal{C}_{J,K}^{can} \cap \mathcal{C}_{J,K}^0$, then $T\eta \in \mathcal{C}_{J,K}^{can} \cap \mathcal{C}_{J,K}^0$. Suppose $\eta \in \mathcal{C}_{J,K}^{can} \cap \mathcal{C}_{J,K}^0$. It is then the case that $\eta_n \in \{0,J\}$ infinitely often as $n \to -\infty$. We will assume that $\eta_n=0$ infinitely often as $n \to -\infty$, and we simply note that the case $\eta_n=J$ infinitely often can be obtained from this by considering the empty box-ball duality of \eqref{spaceball}. If there exists a decreasing subsequence $(n_k)_{k \in \mathbb{N}}$ of integers such that $n_k \to -\infty$, and $n_k-n_{k+1}$ is odd and $\eta_{n_k}=0$ for all $k$, then, for each $k$, there exists $n_{k+1} < \ell_k  \le n_k$ such that either $T\eta_{\ell_k}=0$ and $\ell_k-n_{k+1}$ is odd, or $T\eta_{\ell_k}=J$ and $\ell_k-n_{k+1}$ is even. Indeed, from $n_{k+1}$, the carrier alternates between $0$ and $K$ as follows (drawing the dynamics as in Figure \ref{diagrams}):
\begin{equation}\label{sequence}
\xymatrix@C-15pt@R-15pt{          & b                       &         &a_1-K           &   &a_2+K&\\
                           b \ar[rr]&                         & 0\ar[rr]&                & K\ar[rr] &&      \dots,\\
                                    & \eta_{n_{k+1}}=0 \ar[uu]&         &a_1\geq K\ar[uu]&   &a_2\leq J-K\ar[uu]& }
\end{equation}
with the sequence being ended by one of the following possibilities:
\[\xymatrix@C-15pt@R-15pt{ & 0 & &&  & J & \\
            0 \ar[rr]& & a,&&K \ar[rr]& & a+K-J.\\
             & a<K \ar[uu]&&&& a>J-K \ar[uu]&}\]
Observe that $\ell_k-\ell_{k+1}$ is odd if $\eta_{\ell_{k+1}}=\eta_{\ell_k}$, and $\ell_k-\ell_{k+1}$ is even if $\eta_{\ell_{k+1}} \neq \eta_{\ell_k}$. Hence $T\eta$ satisfies $\eta_n \in \{0,J\}$ infinitely often as $n\to-\infty$, and  $T\eta_n \notin \mathcal{C}_{J,K}^{0,alt}$. Thus, by Proposition \ref{ch1}, $T\eta \in \mathcal{C}_{J,K}^{can} \cap \mathcal{C}_{J,K}^0$. Next, if such a subsequence $(n_k)_{k \in \mathbb{N}}$ does not exist, then there exists $N$ such that for $n,m \le N$, $\eta_n=\eta_m=0$ implies $n-m$ is even. We can therefore find a decreasing sequence $(n_k)_{k \in \mathbb{N}}$ of integers such that $n_k \to -\infty$, and, for all $k$, $n_k-n_{k+1}$ is even, $\eta_{n_k}=0$, and $\eta_{\ell} \neq 0$ for any $n_{k+1} < \ell <n_k$. Since $\eta \notin \mathcal{C}_{J,K}^{0,alt}$, we know the sequence sketched at \eqref{sequence} terminates prior to $n_k$ infinitely often as $k\to\infty$, and so the following happens infinitely often as $k \to \infty$: there exists $1 \le m \le \frac{n_k-n_{k+1}}{2}$ such that either $\eta_{2m-1+n_{k+1}} < K$, or $\eta_{2m+n_{k+1}} > J-K$. For such $k$, there exists an integer $\ell_k\in(n_{k+1},n_k]$ such that $T\eta_{\ell}\not\in \{0,J\}$ for all $n_{k+1} <\ell<\ell_k$, and for which $T\eta_{\ell_k}=0$ if $\ell_k-n_{k+1}$ is odd, and $T\eta_{\ell_k}=J$ if $\ell_k-n_{k+1}$ is even. Moreover, if $T\eta_{n_k}=K$, then by considering the picture at \eqref{sequence} reversed using the involution property of \eqref{involution} we see that the relevant part of the system looks like:
\[\xymatrix@C-15pt@R-15pt{          & a_2\geq K                      &         &a_1\leq J-K          &   &K&\\
                           \dots \ar[rr]&                         & 0\ar[rr]&                & K\ar[rr] &&      0.\\
                                    & a_2-K \ar[uu]&         &a_1+K\ar[uu]&   &\eta_{n_{k}}=0\ar[uu]& }\]
Since $T\eta_{n_k}\leq K$, it follows that, if $T\eta_{\ell_k}=0$, then there exists $\tilde{\ell}_k\in(\ell_k,n_k]$ such that either $T\eta_{\tilde{\ell}_k} <K$ and $\tilde{\ell}_k-\ell_k$ is odd, or $T\eta_{\tilde{\ell}_k} > J-K$ and $\tilde{\ell}_k-\ell_k$ is even. Similarly, if $T\eta_{\ell_k}=J$, then there exists $\tilde{\ell}_k\in(\ell_k,n_k]$ such that either $T\eta_{\tilde{\ell}_k} >J-K$ and $\tilde{\ell}_k-\ell_k$ is odd, or $T\eta_{\tilde{\ell}_k} < K$ and $\tilde{\ell}_k-\ell_k$ is even. Therefore $T\eta \notin \mathcal{C}_{J,K}^{0,alt}$, and so $T\eta \in \mathcal{C}_{J,K}^{can}$.

Now, suppose $\eta \in \mathcal{C}_{J,K}^{can} \cap \mathcal{C}_{J,K}^r$ for some $r<\frac{K}{2}$. From the reducibility property \eqref{reducibility}, it is then easily checked that $\tilde{\eta}:=\min\{\max\{0,\eta-r\},J-2r\}$ satisfies $\tilde{\eta}\in \mathcal{C}_{J-2r,K-2r}^{can} \cap \mathcal{C}_{J-2r,K-2r}^0$. Hence $T_{J-2r,K-2r}\tilde{\eta} \in \mathcal{C}_{J-2r,K-2r}^{can} \cap \mathcal{C}_{J-2r,K-2r}^0$ by the first part of the proof. Since $\tilde{\eta}_n=\eta_n$ for small $n$, it follows that $T_{J,K}\eta=T_{J-2r,K-2r}\tilde{\eta}+r$ satisfies $T\eta \in \mathcal{C}_{J,K}^{can} \cap \mathcal{C}_{J,K}^r$. Furthermore, recall Lemma \ref{le2r} yields that for $r=\frac{K}{2}$, $\mathcal{C}_{J,K}^{can} \cap \mathcal{C}_{J,K}^r=\emptyset$. Hence no configuration satisfying $\eta\in \mathcal{C}_{J,K}^{can} \cap \mathcal{C}_{J,K}^r$ exists in this case, and so the proof of the first claim is complete.

The claim that $\mathcal{C}_{J,K}^{rev}=\mathcal{C}_{J,K}^{inv}$ follows by applying the first part of the lemma in conjunction with Proposition \ref{tinvprp}.
\end{proof}

To conclude the subsection, we check a basic property that holds for elements of $\mathcal{C}^{inv}_{J,\infty}\setminus \tilde{\mathcal{C}}^{inv}_{J,\infty}$, which will be applied in the proof of Lemma \ref{pinvlmm} below. The notation $B_{\cdot}$ in the statement is defined at \eqref{beta}.

\begin{lmm}\label{lmmB}
Suppose $J <\infty$. If $\eta \in \mathcal{C}^{inv}_{J,\infty}\setminus \tilde{\mathcal{C}}^{inv}_{J,\infty}$, then either $B_{\mathcal{D}_{J,\infty}(\eta)}=-\infty$ or $B_{R \mathcal{D}_{J,\infty}(\eta)}=-\infty$ holds.
\end{lmm}
\begin{proof} By definition, $\mathcal{D}_{J,\infty}(\eta) \notin \mathcal{C}_{\infty,J}^{inv}$, and so Lemma \ref{12345} yields further that $\mathcal{D}_{J,\infty}(\eta) \notin \mathcal{C}_{\infty,J}^{rev}$. In particular, either $\mathcal{D}_{J,\infty}(\eta) \notin \mathcal{C}_{\infty,J}^{can}$ or $R\mathcal{D}_{J,\infty}(\eta) \notin \mathcal{C}_{\infty,J}^{can}$ holds. Suppose $\mathcal{D}_{J,\infty}(\eta) \notin \mathcal{C}_{\infty,J}^{can}$. If $B_{\mathcal{D}_{J,\infty}(\eta)}>-\infty$, then by the argument applied in the proof of Lemma \ref{keylemma}(b) and the proof of Theorem \ref{mainthm1}, we can conclude that $\eta \notin \mathcal{C}_{J,\infty}^{inv}$, which contradicts the assumption. Thus $\mathcal{D}_{J,\infty}(\eta) \notin \mathcal{C}_{\infty,J}^{can}$ implies $B_{\mathcal{D}_{J,\infty}(\eta)}=-\infty$. In the same way, one can also show that if $R\mathcal{D}_{J,\infty}(\eta) \notin \mathcal{C}_{\infty,J}^{can}$, then $B_{R\mathcal{D}_{J,\infty}(\eta)}=-\infty$.
\end{proof}

\subsection{Case 2: $J=K$}\label{case2}

In this subsection, we assume $J=K$, which is the easiest case to deal with.

\begin{prp}\label{ch2} Suppose $J=K$.\\
(a) It holds that $\mathcal{C}_{J,K}=\mathcal{C}_{J,K}^{\exists}=\mathcal{C}_{J,K}^{\exists !}= \mathcal{C}_{J,K}^{can}=\mathcal{C}_{J,K}^{rev}=\mathcal{C}_{J,K}^{inv}$.\\
(b) For any $\eta\in\mathcal{C}_{J,K}$, $T\eta=\theta^{-1}\eta$.
\end{prp}
\begin{proof} Since $F^{(2)}_{J,K}(a,b)=a$ for any $a,b\in\{0,1,\dots,J=K\}$, we immediately see from \eqref{carrier} that $Y$ is a BBS($J$,$K$) carrier if and only if $Y_n=\eta_n$ for all $n\in\mathbb{Z}$. In particular, the carrier always exists and is unique. Moreover, by Definition \ref{cancarr}, this carrier is canonical. Thus we have proved part (a). For part (b), we simply observe that $T\eta_n=F^{(1)}_{J,K}(\eta_n,W_{n-1})=W_{n-1}=\eta_{n-1}$, and so the dynamics are given by the right-shift map $\theta^{-1}$.
\end{proof}

\begin{proof}[Proof of Lemma \ref{unique} in the case $J=K$]
This follows from Proposition \ref{ch2}.
\end{proof}

\begin{proof}[Proof of Lemma \ref{keylemma} in the case $J=K$]
Since any carrier is a canonical carrier, no configuration satisfies the assumption.
\end{proof}

\subsection{Case 3: $J<K$}\label{J<K}

In this subsection, we assume $J<K$, and will later subdivide into the cases when $K=\infty$, and when $K<\infty$. As in the study of the original box-ball system, BBS(1,$\infty$), in \cite{CKST}, it turns out that when $J<K$ it is useful to describe the configuration and dynamics in terms of a path encoding and an appropriate operation on this. Specifically, for a configuration $\eta \in \mathcal{C}_{J,K}$, the associated path encoding $S=(S_n)_{n\in\mathbb{Z}}$ will be given by setting $S_0=0$ and
\[S_n-S_{n-1}=J-2\eta_n.\]
Clearly the map $\eta\mapsto S$ gives a bijection between the configuration space $\mathcal{C}_{J,K}$ and the path space $\mathcal{S}_{J,K}:=\{S : \mathbb{Z} \to \mathbb{Z}:\: S_0=0,\:S_n-S_{n-1} \in \{-J,-J+2,\dots, J-2,J\}\}$. In the following lemma, we translate the criteria for the existence of a carrier or a canonical carrier to ones involving the path encoding. For the statement of the result, it will be convenient to introduce the two-point running average of $S$, which we will denote $\tilde{S}=(\tilde{S}_n)_{n\in\mathbb{Z}}$ and define by setting
\[\tilde{S}_n:=\frac{S_{n-1}+S_n}{2}\]

\begin{lmm}\label{pathlemma} Suppose $J<K$. For $Y \in \{0,\dots,K\}^{\mathbb{Z}}$, $Y$ is a carrier for $\eta\in\mathcal{C}_{J,K}$ with path encoding $S$ if and only if the path $M=(M_n)_{n\in \mathbb{Z}}$ given by $M_n:=Y_n+S_n-\frac{J}{2}$ satisfies
\begin{equation}\label{pathm}
M_n=\min\left\{\max\left\{M_{n-1},\tilde{S}_n\right\}, \tilde{S}_n+K-J\right\},\qquad\forall n \in \mathbb{Z}.
\end{equation}
Moreover, $Y$ is a canonical carrier for $\eta$ with path encoding $S$ if and only if $Y$ is a carrier and $M_n$ does not converge in $\mathbb{R}$ as $n\rightarrow-\infty$.
\end{lmm}
\begin{proof}
Suppose that $Y$ is a carrier for $\eta$. Then, by definition,
\[Y_n=Y_{n-1}-\min\{Y_{n-1},J-\eta_n\}+\min\{\eta_n,K-Y_{n-1}\}.\]
Changing variables from $(\eta,Y)$ to $(S,M)$, we have
\begin{eqnarray*}
M_n-S_n +\frac{J}{2}& =&M_{n-1}-S_{n-1}+\frac{J}{2} -\min\left\{M_{n-1}-S_{n-1}+\frac{J}{2}, \frac{J+S_n-S_{n-1}}{2}\right\}\\
 &&+\min\left\{ \frac{J-S_n+S_{n-1}}{2},K-M_{n-1}+S_{n-1}-\frac{J}{2}\right\},
\end{eqnarray*}
from which some elementary rearrangement yields
\[M_n = \max\left\{ \tilde{S}_n, M_{n-1}\right\}+\min\left\{ 0,K-J-M_{n-1}+\tilde{S}_n\right\}.\]
If $M_{n-1} \ge  \tilde{S}_n$, then
\[M_n =\min\left\{ M_{n-1},K-J+\tilde{S}_n\right\} =\min\left\{\max\left\{M_{n-1},\tilde{S}_n\right\}, \tilde{S}_n+K-J\right\}.\]
On the other hand, if $M_{n-1} \le \tilde{S}_n$, then $ 0 \le K-J-M_{n-1}+\tilde{S}_n$, and so we have
\[M_n=\tilde{S}_n=\min\left\{\max\left\{M_{n-1},\tilde{S}_n\right\}, \tilde{S}_n+K-J\right\}.\]
Hence, in both cases, \eqref{pathm} holds. Reversing the steps of the argument, one obtains that the condition \eqref{pathm} is also necessary for $Y$ to be a carrier.

Next we consider when a carrier $Y$ is canonical. For this, we first claim the equivalence of the conditions $J \le \eta_n+Y_{n-1} \le K$ and $M_{n-1}=M_n$. Indeed, this readily follows from \eqref{pathm} and the observation that
\begin{equation}\label{iff}
J \le \eta_n+Y_{n-1} \le K \qquad \Leftrightarrow \qquad \tilde{S}_n \le M_{n-1} \le \tilde{S}_n+K-J,
\end{equation}
which one can deduce by applying the definition of $M$ and rearranging. Now, recall that $Y$ is a canonical carrier if and only if $B_Y=-\infty$, where $B_Y$ was defined at \eqref{essbound}. From the preceding argument, we see that $B_Y \neq -\infty$ implies $M_n$ is eventually constant as $n\rightarrow-\infty$, and moreover, $B_Y= -\infty$ implies $M_n$ does not converge as $n \to -\infty$ (since $|M_n-M_{n-1}| \ge 1$ if $M_n \neq M_{n-1}$). This completes the proof of the second part.
\end{proof}

We now describe how the BBS($J$,$K$) dynamics can be expressed in terms of the path encoding via a Pitman-type transformation. We recall that the original Pitman transformation of a path involves reflection in the past maximum \cite{Pitman}; up to a shift, our path transformation is also a reflection, but in the path $M$, as defined in the statement of Lemma \ref{pathlemma}. For the model BBS(1,$\infty$), this reduces to Pitman's original definition. Figure \ref{bbs35} shows an example of the transformation when $J=3$ and $K=5$.

\begin{figure}[t]
\begin{center}
\scalebox{0.7}{
\includegraphics[width=0.6\textwidth]{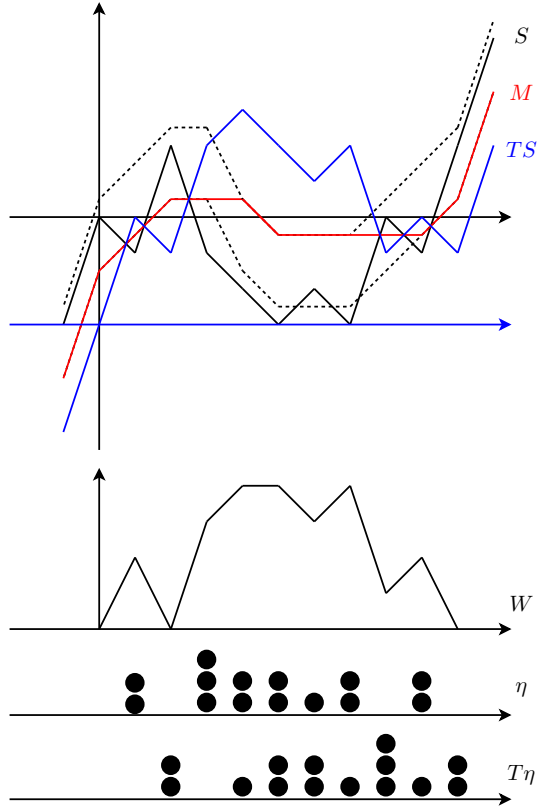}
\rput(0,12.2){$S$}
\rput(0,11.3){\color{red} $M$}
\rput(0,10.4){\color{blue} $TS$}
\rput(0,3.2){$W$}
\rput(0,1.85){$\eta$}
\rput(0,0.5){$T\eta$}}
\vspace{-15pt}
\end{center}
\caption{Path encodings for the BBS(3,5). The top graph shows the path encoding of the original configuration $S$ (black), $M$ (red), and the path encoding of the updated configuration $TS$ (blue). The dotted lines show the two-point running average $\tilde{S}$, and the version of this shifted upwards by $2$. The second graph shows the canonical carrier process $W$ corresponding to the initial configuration. The final two graphs show the initial configuration $\eta$ and the updated one $T\eta$, respectively. We acknowledge Satoshi Tsujimoto as providing the inspiration for this depiction.}\label{bbs35}
\end{figure}

\begin{lmm}\label{pitman}
Suppose $J<K$. If $\eta \in \mathcal{C}_{J,K}^{\exists}$ and $Y$ is an associated carrier, then $T^YS=2M-S-2M_0$, where $T^YS$ is the path encoding of $T^Y\eta$, and $M$ is defined as in Lemma \ref{pathlemma}.
\end{lmm}
\begin{proof}
Since $T^YS_0=0=2M_0-S_0-2M_0$, we only need to show that the increments of $T^YS$ and $2M-S-2M_0$ coincide. By definition and \eqref{consmass},
\[T^YS_n-T^YS_{n-1}=J-2T^{Y}\eta_n=J-2(Y_{n-1}+\eta_n-Y_n)=2Y_n-2Y_{n-1}+S_n-S_{n-1}.\]
Moreover,
\begin{align*}
2M_n-S_n-(2M_{n-1}-S_{n-1})&=2Y_n+2S_n-J-S_n-(2Y_{n-1}+2S_{n-1}-J)+S_{n-1}\\
 &=2Y_n-2Y_{n-1}+S_n-S_{n-1},
\end{align*}
and so the proof is complete.
\end{proof}

\subsubsection{Case 3(a): $J<K=\infty$}\label{case3a}

We now focus on the case that $J<K=\infty$. In this case, the condition \eqref{pathm} simplifies to
\begin{equation}\label{pathminf}
M_n=\max\{M_{n-1},\tilde{S}_n\},\qquad \forall n\in\mathbb{Z},
\end{equation}
one solution of which is given by taking $M$ to be the past maximum of $\tilde{S}$, if this exists. This observation will be key in the following arguments. For these, we also let
\begin{equation}\label{overlineS}
\overline{S}_{-\infty}:=\limsup_{n \to -\infty}S_n, \qquad \underline{S}_{-\infty}:=\liminf_{n \to -\infty}S_n,
\end{equation}
and define $\overline{\tilde{S}}_{-\infty}$ and $\underline{\tilde{S}}_{-\infty}$ from $\tilde{S}$ similarly. Note that, since $|S_n-\tilde{S}_n|  \le \frac{J}{2}$, $\overline{S}_{-\infty}=\pm \infty$ if and only if $\overline{\tilde{S}}_{-\infty}=\pm \infty$, and the same holds for $\underline{S}_{-\infty}$ and $ \underline{\tilde{S}}_{-\infty}$. We continue by describing $\mathcal{C}_{J,\infty}^{\exists}$, $\mathcal{C}_{J,\infty}^{\exists !}$ and $\mathcal{C}_{J,\infty}^{can}$.

\begin{prp}\label{ch3} Suppose $J<K=\infty$. It then holds that:\\
(a) $\mathcal{C}_{J,\infty}^{\exists}=\{\eta \in \mathcal{C}_{J,\infty}:\:\overline{S}_{-\infty} < \infty \}$;\\
(b) $\mathcal{C}_{J,\infty}^{\exists !}= \emptyset$;\\
(c) $\mathcal{C}_{J,\infty}^{can}=\{\eta \in \mathcal{C}_{J,\infty}:\:\overline{S}_{-\infty}=-\infty\}$.
\end{prp}
\begin{proof}
(a) We first show that if $\overline{S}_{-\infty} < \infty$, then there exists a carrier. Under the latter condition, we also have that $\overline{\tilde{S}}_{-\infty} < \infty$. Hence $\tilde{M}_n:=\max_{m\le n}\tilde{S}_m$ is finite for $n\in\mathbb{Z}$, and, as per the remark preceding the proposition, $\tilde{M}=(\tilde{M}_n)_{n\in\mathbb{Z}}$ satisfies \eqref{pathminf}. For this $\tilde{M}$, define $Y=(Y_n)_{n\in\mathbb{Z}}$ by setting $Y_n:=\tilde{M}_n-S_n+\frac{J}{2}$. To establish that this $Y$ is a carrier for $\eta$ with path encoding $S$, it will be enough to check that $Y_n\in\mathbb{Z}_+$ for each $n\in\mathbb{Z}_+$. Since $\tilde{M}_n \ge \tilde{S}_n$, we have that $Y_n \ge 0$. Also, if $J$ is even, then $S_n \in 2\mathbb{Z}$ for all $n$, and it readily follows that $Y_n \in \mathbb{Z}$. If $J$ is odd, then $S_n$ is odd for odd $n$ and $S_n$ is even for even $n$. It follows that $\tilde{M}_n \in \mathbb{Z}+\frac{1}{2}$, and so $Y_n \in \mathbb{Z}$, as desired. We next show that if $\overline{S}_{-\infty} = \infty$, then a carrier does not exist. Indeed, if $(M_n)_{n\in\mathbb{Z}}$ satisfies \eqref{pathminf}, then $M_n \ge \tilde{S}_n$ for all $n$. Hence, applying \eqref{pathminf} repeatedly, $M_n \ge \max_{m \le n }\tilde{S}_m$. Since $\overline{S}_{-\infty} = \infty$ implies $\overline{\tilde{S}}_{-\infty} = \infty$, $M_n=\infty$ in this case, and thus a carrier does not exist.\\
(b) Suppose $\overline{S}_{-\infty} < \infty$ and again let $\tilde{M}_n=\max_{m \le n }\tilde{S}_m$. Then, for any $A \ge \overline{\tilde{S}}_{-\infty}$ satisfying $A \in \mathbb{Z}$ if $J$ is even, and $A \in \mathbb{Z}+\frac{1}{2}$ if $J$ is odd, $M_n:=\max\{A,\tilde{M}_n\}$ also satisfies \eqref{pathminf} and $Y_n:=M_n-S_n+\frac{J}{2}$ is a carrier. Hence uniqueness does not hold for any $\eta$.\\
(c) From Lemma \ref{pathlemma}, a carrier $Y$ is canonical if and only if the associated $M$ does not converge in $\mathbb{R}$ as $n \to -\infty$. On the other hand, if $M$ satisfies \eqref{pathminf}, then $M$ is non-decreasing. Hence a carrier $Y$ is canonical if and only if $\lim_{n \to \infty}M_n= -\infty$. Since $M_n \ge \max_{m \le n }\tilde{S}_m$ for any $n$, $\lim_{n \to \infty}M_n= -\infty$ implies $\overline{S}_{-\infty}= -\infty$. Moreover, if $\overline{S}_{-\infty}= -\infty$, then $Y_n:=\tilde{M}_n-S_n+\frac{J}{2}$ is a canonical carrier, where $\tilde{M}_n=\max_{m \le n }\tilde{S}_m$.
(It is further straightforward to check that the only other carriers are of the form described in the proof of (b), and that none of these are canonical.)
\end{proof}

\begin{proof}[Proof of Lemma \ref{unique} in the case $J<K=\infty$] Suppose $\eta \in \mathcal{C}_{J,\infty}^{can}$, and $Y$ is a carrier with $M_n$ satisfying $M_n > \tilde{M}_n:=\max_{m \le n }\tilde{S}_m$ for some $n \in \mathbb{Z}$. Since $M_n=\max\{M_{n-1},\tilde{S}_n\}$ and $\tilde{S}_n < M_n$, it must hold that $M_{n-1}=M_n$. Similarly, we obtain that $M_m=M_n$ for all $m \le n$, and so, by Lemma \ref{pathlemma}, $Y$ is not a canonical carrier. Hence $Y$ is a canonical carrier if and only if $M_n=\tilde{M}_n$ for all $n$, and thus it is unique.
\end{proof}

\begin{proof}[Proof of Lemma \ref{keylemma} in the case $J<K=\infty$]
(a) Suppose $\eta \in \mathcal{C}_{J,\infty}^{\exists}$ and $Y$ is a carrier for $\eta$, but not canonical. By Lemma \ref{pathlemma}, $M_n$ is a constant for $n \le B_Y$. Hence we obtain from Lemma \ref{pitman} that $T^YS_n-T^YS_{n-1}=-(S_n-S_{n-1})$ for $n \le B_Y$, and so $\limsup_{n \to -\infty}T^YS_n \geq \liminf_{n \to -\infty}T^YS_n = - \overline{S}_{-\infty}+C$ for some constant $C\in\mathbb{Z}$. Since $\overline{S}_{-\infty}<\infty$, we have that $\limsup_{n \to -\infty}T^YS_n>-\infty$, and so $T^Y\eta \notin \mathcal{C}_{J,K}^{can}$.\\
(b) Let $(\eta^{(i)})_{i \in \mathbb{N}}$ and $(Y^{(i)})_{i \in \mathbb{N}}$ satisfy the assumptions of the lemma, so that in particular $Y^{(1)}=Y$ is not canonical. By Lemma \ref{ch3}, $\overline{S}_{-\infty} <\infty$. Moreover, if $\underline{S}_{-\infty}=-\infty$, then arguing as in (a) allows us to deduce that $\limsup_{n \to -\infty}T^YS_n=-\underline{S}_{-\infty}+C=\infty$, which yields in turn that $T^Y\eta \notin \mathcal{C}_{J,\infty}^{\exists}$; this contradicts the assumption that $Y^{(2)}$ is a carrier for $\eta^{(2)}$. Hence $\underline{S}_{-\infty} >-\infty$, and so there exists a $B_{\eta} \in \mathbb{Z}$ such that $\sup_{m \le n}\tilde{S}_m=\overline{\tilde{S}}_{-\infty}$ and $\inf_{m \le n}\tilde{S}_m=\underline{\tilde{S}}_{-\infty}$ for all  $n\le B_{\eta}$, where $\overline{\tilde{S}}_{-\infty},\underline{\tilde{S}}_{-\infty}\in\mathbb{R}$. Now, since $M_n$ satisfies \eqref{pathminf}, then it must be the case that $M_n=Y_n+S_n-\frac{J}{2}$ is constant for $n \le B_{\eta}$. Recalling \eqref{iff}, it follows that $B_Y \ge B_{\eta}$. Moreover, by Lemma \ref{pitman}, $S^{(2)}:=T^YS$, the path encoding of $\eta^{(2)}=T^Y\eta^{(1)}$ also satisfies $\sup_{m \le n}\tilde{S}^{(2)}_m=\overline{\tilde{S}^{(2)}}_{-\infty}$ and $\inf_{m \le n}\tilde{S}^{(2)}_m=\underline{\tilde{S}^{(2)}}_{-\infty}$ for all $n \le B_{\eta}$, where $\overline{\tilde{S}^{(2)}}_{-\infty}, \underline{\tilde{S}^{(2)}}_{-\infty}\in\mathbb{R}$. Arguing as before, it must be the case that $B_{Y^{(2)}} \ge B_{\eta}$. Repeating the same argument, we conclude that $\inf_{i}B_{Y^{(i)}} \ge B_{\eta} > -\infty$.
\end{proof}

We complete the subsection by identifying an explicit subset of $\mathcal{C}^{inv}_{J,K}$, which is natural to consider for suitably homogeneous random configurations, and also present an example to complete the discussion of the remark following Theorem \ref{mainthm1}.

\begin{crl}\label{invsubs} For $J<K=\infty$, $\{\eta \in \mathcal{C}_{J,\infty}:\: \exists \lim_{n \to \pm\infty}\frac{S_n}{n}\in(0,\infty)\}\subseteq\mathcal{C}_{J,\infty}^{inv}$.
\end{crl}
\begin{proof} If $\lim_{n \to \pm\infty}\frac{S_n}{n}=c_{\pm}\in(0,\infty)$, then Proposition \ref{ch3} gives that both $\eta$ and $R\eta$, which has path encoding $RS=(-S_{-n})_{n\in\mathbb{Z}}$, are in $\mathcal{C}^{can}_{J,\infty}$, and so $\eta\in\mathcal{C}^{rev}_{J,K}$. Moreover, from the assumption, we deduce that $\tilde{M}_n:=\sup_{m \le n}\tilde{S}_m$ also satisfies $\lim_{n \to \pm\infty}\frac{\tilde{M}_n}{n}=c_{\pm}$. This implies
\[\lim_{n \to \pm\infty}\frac{TS_n}{n}=\lim_{n \to \pm\infty}\frac{2\tilde{M}_n-S_n-2\tilde{M}_0}{n}=c_{\pm},\]
and hence $T\eta\in \mathcal{C}^{rev}_{J,K}$. (Recall from the proof of Proposition \ref{ch3} that $\tilde{M}$ is the process satisfying \eqref{pathminf} that corresponds to the canonical carrier.) Proceeding similarly with $RS$ in place of $S$, we find that, writing $T^{-1}S$ as the path encoding of $T^{-1}\eta$,
\[\lim_{n \to \pm\infty}\frac{T^{-1}S_n}{n}=\lim_{n \to \pm\infty}\frac{RTRS_n}{n}=c_{\pm},\]
and so $T^{-1}\eta\in \mathcal{C}^{rev}_{J,K}$. Iterating these arguments yields that $T^t\eta\in \mathcal{C}^{rev}_{J,K}$ for all $t\in\mathbb{Z}$. Thus the proof is complete.
\end{proof}

\newcommand{\BA}{\circle{10}}
\newcommand{\BB}{\circle*{10}}
\begin{exa}\label{notbijection} Consider the following configuration $\eta\in\mathcal{C}_{1,\infty}$:
\begin{center}
\dots\:\:\:{\BA\BB\BB\BB\BA\BA\BA\BA\BA\BA\BB\BB\BA\BA\BA\BB\BA\BA\BA\BA\BA\BA\BA\BA\BA}\dots,
\end{center}
where for each $n\in\mathbb{N}$ a string of particles of length $n$ is placed in the interval $\{-(2n+1)(n-1),\dots,-2n(n-1)\}$, and all other sites are vacant. It is elementary to check that $\frac{S_n}{n}\rightarrow 1$ as $n\rightarrow\infty$, and $\frac{S_n}{n}\rightarrow \frac{1}{2}$ as $n\rightarrow-\infty$, and so Lemma \ref{invsubs} yields that $\eta\in\mathcal{C}^{inv}_{1,\infty}$. It is further possible to check that the current sequence $((T^tW)_0)_{t\in\mathbb{Z}}$ is given by $(\dots,0,0,0,W_0=1,0,2,0,3,0,4,\dots)$, and the image of this configuration under $R$ is clearly an element of $\mathcal{C}^0_{\infty,1}\cap\mathcal{C}^{0,alt}_{\infty,1}$. Hence, by Proposition \ref{ch1} and Lemma \ref{12345}, $((T^tW)_0)_{t\in\mathbb{Z}}\notin \mathcal{C}^{rev}_{\infty,1}=\mathcal{C}^{inv}_{\infty,1}$. To construct a similar example of $\eta\in\mathcal{C}^{inv}_{J,\infty}$ for $J<\infty$ with $\mathcal{D}_{J,\infty}(\eta)\not\in \mathcal{C}_{\infty,J}^{inv}$, simply replace the individual particles with boxes filled to their capacity $J$, and leave all other sites vacant; in this case the current is given by $(\dots,0,0,0,W_0=J,0,2J,0,3J,0,4J,\dots)$, and the same argument applies.
\end{exa}

\begin{rem}
We note that $\mathcal{C}^{inv}_{1,\infty}$ is the same set as $\mathcal{S}_{sub-critical}$ introduced in \cite{CKST}.
\end{rem}

\subsubsection{Case 3(b): $J<K<\infty$}\label{case3b}

We now come to the final case, which is when $J<K<\infty$. Towards describing the sets $\mathcal{C}_{J,K}^{\exists}$, $\mathcal{C}_{J,K}^{\exists !}$ and $\mathcal{C}_{J,K}^{can}$, we first show that whenever the two-point running average of the path encoding $\tilde{S}$ fluctuates more than $K-J$, the value of the process $M$ (as described at \eqref{pathm}) can be determined from $\tilde{S}$ uniquely; this is because the carrier sees greater than or equal to $K$ more empty boxes than balls, or vice versa, over the relevant part of the configuration, which means that it essentially empties or fills itself.

\begin{lmm}\label{fluctuation}
Suppose $J<K <\infty$. Let $\eta \in \mathcal{C}_{J,K}^{\exists}$, $Y$ be a carrier for $\eta$, and $M$ be the process given by $M_n=Y_n+S_n-\frac{J}{2}$. If $|\tilde{S}_N -\tilde{S}_n| > K-J$ for some $n<N$, and $|\tilde{S}_m-\tilde{S}_n|\le K-J$ for all $m\in\{n,n+1,\dots,N-1\}$, then the following hold:\\
(a) $M_N \neq M_{N-1}$;\\
(b) if $\tilde{S}_N -\tilde{S}_n > K-J$, then $M_N=\tilde{S}_N$;\\
(c) if $\tilde{S}_N -\tilde{S}_n <-(K-J)$, then $M_N=\tilde{S}_N+K-J$.
\end{lmm}
\begin{proof} First observe that \eqref{pathm} implies $\tilde{S}_n \le M_n\le \tilde{S}_n+K-J$ for all $n \in \mathbb{Z}$.  Suppose $n <N$, $\tilde{S}_N -\tilde{S}_n > K-J$ and $|\tilde{S}_m-\tilde{S}_n|\le K-J$ for all $m\in\{n,n+1,\dots,N-1\}$, then $\tilde{S}_m < \tilde{S}_N$ for all $m$ in the latter range, and also $M_n \le \tilde{S}_n + K-J <\tilde{S}_N$. Hence, since $M_{n+1} \le \max\{M_n, \tilde{S}_{n+1}\}$, if $n+1 <N$, then $M_{n+1} < \tilde{S}_N$. Recursively, we deduce that $M_{m} < \tilde{S}_N$ for any $m\in\{n,n+1,\dots,N-1\}$. It follows that $M_N=\tilde{S}_N$, and so $M_N \neq M_{N-1}$. The case when $\tilde{S}_N -\tilde{S}_n <-(K-J)$ can be dealt with in a similar way, and thus we establish (a). Note that the proofs of part (b) and (c) are contained in the argument already given.
\end{proof}

\begin{prp}\label{ch4}
Suppose $J<K < \infty$. It then holds that:\\
(a) $\mathcal{C}_{J,K}^{\exists}= \mathcal{C}_{J,K}$;\\
(b) $\mathcal{C}_{J,K}^{\exists !}= \{\eta \in \mathcal{C}_{J,K}:\:\overline{\tilde{S}}_{-\infty}=\infty\mbox{ or } \underline{\tilde{S}}_{-\infty} = -\infty\mbox{ or }\overline{\tilde{S}}_{-\infty} \ge \underline{\tilde{S}}_{-\infty}+K-J\}$;\\
(c)  $\mathcal{C}_{J,K}^{can}=\{\eta \in \mathcal{C}_{J,K}:\:\overline{\tilde{S}}_{-\infty}=\infty\mbox{ or } \underline{\tilde{S}}_{-\infty} = -\infty\mbox{ or }\overline{\tilde{S}}_{-\infty} > \underline{\tilde{S}}_{-\infty} + K-J \}$,\\
where we recall the notation $\overline{\tilde{S}}_{-\infty}$ and $\underline{\tilde{S}}_{-\infty}$ from below \eqref{overlineS}.
\end{prp}
\begin{proof}
We consider three cases separately: (i) $\overline{\tilde{S}}_{-\infty}=\infty$ or $\underline{\tilde{S}}_{-\infty} = -\infty$ or $\overline{\tilde{S}}_{-\infty} > \underline{\tilde{S}}_{-\infty} + K-J $; (ii) $\overline{\tilde{S}}_{-\infty} = \underline{\tilde{S}}_{-\infty} + K-J \in \mathbb{R}$; (iii) $\overline{\tilde{S}}_{-\infty} < \underline{\tilde{S}}_{-\infty} + K-J$. Specifically, we will show that: in case (i), there exists a unique carrier which is canonical; in case (ii), there exists a unique carrier which is not canonical; and in case (iii), there are multiple carriers where any of them is not canonical. From this, the result follows.

Suppose (i) holds. There then exists a decreasing divergent sequence $(N_i)_{i\in\mathbb{N}}$ of integers such that $|\tilde{S}_{N_{i}} -\tilde{S}_{N_{i+1}}| > K-J$ and also $|\tilde{S}_m-\tilde{S}_{N_i}|\le K-J$ for all $m\in\{N_{i+1},N_{i+1}+1,\dots,N_{i}-1\}$. Now, define a process $M$ by setting $M_{N_i}:=\tilde{S}_{N_i}$ if $\tilde{S}_{N_i}-\tilde{S}_{N_{i+1}}>K-J$ and $M_{N_i}:=\tilde{S}_{N_i}+K-J$ if $\tilde{S}_{N_i}-\tilde{S}_{N_{i+1}}<-(K-J)$, and then defining $M_n$ for $n\neq N_i$ from \eqref{pathm} recursively. Arguing as in the proof of Lemma \ref{fluctuation}, it follows that $M$ satisfies \eqref{pathm} everywhere. Since $\tilde{S}_n \le M_n \le \tilde{S}_n +K-J$ and $M-\tilde{S}$ takes integer values, it readily follows that $Y:=M-S+\frac{J}{2}$ takes values in $\{0,1,\dots,K\}$, and hence is a carrier for $\eta$. Lemma \ref{fluctuation} further tells us that this carrier is unique. Finally, applying Lemma \ref{fluctuation} again yields $M_{N_i} \neq M_{N_i-1}$, which means $M_n$ does not converge to a constant as $n \to -\infty$. Hence, by Lemma \ref{pathlemma}, the carrier is canonical.

Next, suppose (ii) holds. There then exists $B_{\eta} \in \mathbb{Z}$ such that $\sup_{m \le n}\tilde{S}_m=\overline{\tilde{S}}_{-\infty}=\underline{\tilde{S}}_{-\infty}+K-J$ and $\inf_{m \le n}\tilde{S}_m=\underline{\tilde{S}}_{-\infty}$ for all $n \le B_{\eta}$. Define $M_n:=\overline{\tilde{S}}_{-\infty}$ for $n \le B_{\eta}$, and by \eqref{pathm} for $n> B_{\eta}$ recursively. Since $\tilde{S}_n \le M_n\leq \tilde{S}_n+K-J$ for $n \le B_{\eta}$, $M$ satisfies \eqref{pathm}, and so there exists a carrier $Y$. Moreover, by Lemma \ref{pathlemma}, this is not canonical. If there exists another carrier, then $M_n \neq \overline{\tilde{S}}_{-\infty}$ for some $n \le B_{\eta}$. Suppose $M_n > \overline{\tilde{S}}_{-\infty}$. Since $\overline{\tilde{S}}_{-\infty}=\underline{\tilde{S}}_{-\infty}+K-J$, there exists an $N < n$ such that $M_N\le S_N+K-J=\underline{\tilde{S}}_{-\infty}+K-J=\overline{\tilde{S}}_{-\infty}$. Thus, since $M$ satisfies \eqref{pathm}, $M_{m} \le \max\{M_{m-1},\tilde{S}_m\} \le \overline{\tilde{S}}_{-\infty}$ for $m\in\{N+1,N+2,\dots,n\}$. In particular, we have that $M_n \le \overline{\tilde{S}}_{-\infty}$, which contradicts the assumption. In the same way, we can show that $M_n <\overline{\tilde{S}}_{-\infty}$ for some $n \le B_{\eta}$ yields a contradiction. Hence the carrier is unique.

Finally, suppose (iii) holds. Then,  there exists $B_{\eta} \in \mathbb{Z}$ such that $\sup_{m \le n}\tilde{S}_m=\overline{\tilde{S}}_{-\infty}$ and $\inf_{m \le n}\tilde{S}_m=\underline{\tilde{S}}_{-\infty}$ for all $n \le B_{\eta}$. For any $C \in \{0,1,\dots,\underline{\tilde{S}}_{-\infty} + K-J-\overline{\tilde{S}}_{-\infty}\}$, define $M_n=\overline{\tilde{S}}_{-\infty}+C$ for $n \le B_{\eta}$, and by \eqref{pathm} for $n> B_{\eta}$ recursively. Then, as in case (ii), each such $M$ satisfies \eqref{pathm} and is associated with a carrier. Hence, there exists multiple carriers and none are canonical. (It is also possible to check, as in case (ii), that there are no other carriers.)
\end{proof}

\begin{proof}[Proof of Lemma \ref{unique} in the case $J<K<\infty$]
Since $\mathcal{C}_{J,K}^{can} \subseteq \mathcal{C}_{J,K}^{\exists !}$, this is clear.
\end{proof}

\begin{proof}[Proof of Lemma \ref{keylemma} in the case $J<K<\infty$]
(a) Suppose $\eta \in \mathcal{C}_{J,K}^{\exists}$ and $Y$ is a carrier for $\eta$ but not a canonical carrier. This implies case (ii) or case (iii) from the proof of Proposition \ref{ch4} hold, namely $\overline{\tilde{S}}_{-\infty}, \underline{\tilde{S}}_{-\infty} \in \mathbb{R}$ and $\overline{\tilde{S}}_{-\infty} \le \underline{\tilde{S}}_{-\infty}+K-J$. Since, by Lemma \ref{pathlemma}, $M_n$ is a constant for $n \le B_Y$, Lemma \ref{pitman} implies $T^YS_n-T^YS_{n-1}=-(S_n-S_{n-1})$ for $n \le B_Y$. It follows that $\limsup_{n \to -\infty}\widetilde{T^YS}_n = - \underline{\tilde{S}}_{-\infty}+C$ and $\liminf_{n \to -\infty}\widetilde{T^YS}_n = - \overline{\tilde{S}}_{-\infty}+C$ for some constant $C$. Thus $\limsup_{n \to -\infty}\widetilde{T^YS}_n\in\mathbb{R}$, $\liminf_{n \to -\infty}\widetilde{T^YS}_n \in \mathbb{R}$ and
$\limsup_{n \to -\infty}\widetilde{T^YS}_n \le \liminf_{n \to -\infty}\widetilde{T^YS}_n+K-J$. Therefore, by Proposition \ref{ch4}, $T^Y\eta \notin \mathcal{C}_{J,K}^{can}$.\\
(b) Let $(\eta^{(i)})_{i \in \mathbb{N}}$ and $(Y^{(i)})_{i \in \mathbb{N}}$ satisfy the assumptions of the lemma, so that in particular $Y^{(1)}=Y$ is not canonical. Then, as in the argument for (a), $\overline{S}_{-\infty}, \underline{S}_{-\infty} \in \mathbb{R}$, and so there exists $B_{\eta} \in \mathbb{Z}$ such that $\sup_{m \le n}\tilde{S}_m=\overline{\tilde{S}}_{-\infty}$ and $\inf_{m \le n}\tilde{S}_m=\underline{\tilde{S}}_{-\infty}$ for all $n \le B_{\eta}$. Moreover, by the proof of Proposition \ref{ch4}, $M_n=Y_n+S_n-\frac{J}{2}$ is constant for $n \le B_{\eta}$, and so $B_Y \ge B_{\eta}$ (recall \eqref{iff}). Now, by Lemma \ref{pitman}, $S^{(2)}:=T^YS$, the path encoding of $\eta^{(2)}=T^Y\eta^{(1)}$, also satisfies $\sup_{m \le n}\tilde{S}^{(2)}_m=\overline{\tilde{S}^{(2)}}_{-\infty}$ and $\inf_{m \le n}\tilde{S}^{(2)}_m=\underline{\tilde{S}^{(2)}}_{-\infty}$ for all $n\le B_{\eta}$. In particular, $B_{Y^{(2)}} \ge B_{\eta}$. Repeating the same argument, we conclude that $\inf_{i}B_{Y^{(i)}} \ge B_{\eta} > -\infty$, as required.
\end{proof}

We complete the section by describing an explicit subset of $\mathcal{C}^{inv}_{J,K}$, which, similarly to \eqref{invsubs}, is natural for homogeneous random configurations.

\begin{crl}\label{invsubs2} For $J<K<\infty$, $\{\eta \in \mathcal{C}_{J,K}:\:\limsup_{n \to \pm \infty}|S_n|=\infty\}\subseteq\mathcal{C}_{J,K}^{inv}$.
\end{crl}
\begin{proof} If $\limsup_{n \to \pm \infty}|S_n|=\infty$, then Proposition \ref{ch4} readily yields that $\eta\in\mathcal{C}^{rev}_{J,K}$. Moreover, from \eqref{pathm} and the definition of $\tilde{S}$ we know that $|M_n-S_n| \le |M_n-\tilde{S}_n|+|\tilde{S}_n-S_n| \le K-\frac{J}{2}$ for any $n$. Hence, if $\limsup_{n \to \pm \infty}|S_n|=\infty$, then $\lim_{n \to \pm\infty}|TS_n|=\lim_{n \to \pm\infty}|2M_n-S_n-2M_0|=\infty$, and it follows that $T\eta\in\mathcal{C}^{rev}_{J,K}$. The remainder of the proof is identical to that of Corollary \ref{invsubs}.
\end{proof}

\section{Duality between invariant properties of probability measures}\label{probdualsec}

With the deterministic preparations in place, we are now ready to study probability measures on configurations. In particular, the main aim of this section is to prove Theorem \ref{dualthm}. We also give a lemma that shows how independence between the two sides of the configuration transfers into a corresponding property for the current sequence (see Lemma \ref{indep}). We recall the definitions of $\mathcal{P}_{J,K}$, $\mathcal{P}_{J,K}^{rev}$, $\mathcal{P}_{J,K}^{inv}$, $\tilde{\mathcal{P}}_{J,K}^{inv}$ from above the statement of Theorem \ref{dualthm}, as well as the maps $\mathcal{D}_{J,K}$ and $\mathcal{D}^P_{J,K}$ from \eqref{dualitymap} and \eqref{djkpdef}, respectively. We start with a simple lemma that shows if the dual measure of a measure $P_{J,\infty} \in \mathcal{P}_{J,\infty}^{inv}$ is spatially stationary, then $P_{J,\infty} \in \tilde{\mathcal{P}}_{J,\infty}^{inv}$, which will be important in allowing us to appeal to the bijectivity of $\mathcal{D}^P_{J,\infty}$ on this smaller set.

\begin{lmm}\label{pinvlmm}
If $P_{J,\infty} \in \mathcal{P}_{J,\infty}^{inv}$, and $P_{\infty,J}:= \mathcal{D}^P_{J,\infty}(P_{J,\infty})$ satisfies $P_{\infty,J}\circ \theta^{-1}=P_{\infty,J}$, then $P_{J,\infty} \in \tilde{\mathcal{P}}_{J,\infty}^{inv}$.
\end{lmm}
\begin{proof} Suppose $\eta$ has distribution given by $P_{J, \infty}$, and assume $P_{\infty,J}\circ \theta^{-1}=P_{\infty,J}$ but $P_{J,\infty}(\tilde{\mathcal{C}}_{J,\infty}^{inv}) <1$. By Proposition \ref{ch1} and Lemma \ref{lmmB}, either
\begin{equation}\label{p11}
P_{J,\infty}\hspace{-2pt}\left(B_{\mathcal{D}_{J,\infty}(\eta)} =-\infty, \exists q \in \left\{0,1,\dots, \lfloor\tfrac{J}{2}\rfloor\right\}, i \in \{0,1\}\mbox{ such that }B_{q,i,\mathcal{D}_{J,\infty}(\eta)} > -\infty\right)
\end{equation}
or
\begin{equation}\label{p12}
P_{J,\infty}\hspace{-2pt}\left(B_{R\mathcal{D}_{J,\infty}(\eta)} =-\infty, \exists q \in \left\{0,1,\dots, \lfloor\tfrac{J}{2}\rfloor\right\}, i \in \{0,1\}\mbox{ such that }B_{q,i,R\mathcal{D}_{J,\infty}(\eta)} > -\infty\right)
\end{equation}
is strictly positive, where we recall the definitions of $B_{\eta}, B_{q,i,\eta}$ and $B_{c,i,\eta}$ from Subsection \ref{case1}. Now, since $P_{\infty,J}\circ \theta^{-1}=P_{\infty,J}$, we have that $P_{\infty,J}(B_{c,i,\eta}^{alt} \in \{\infty,-\infty\})=1$ for all $c \in \mathbb{Z}_+$ and $i \in \{0,1\}$. Moreover, since $P_{\infty,J}\circ R \circ \theta^{-1}=P_{\infty,J}\circ R$, it must also be the case that $P_{\infty,J}(B_{c,i,R\eta}^{alt} \in \{\infty,-\infty\})=1$ for all $c \in \mathbb{Z}_+$ and $i \in \{0,1\}$. Hence the probabilities at \eqref{p11} and \eqref{p12} must both be equal to 0, but this is a contradiction.
\end{proof}

\begin{proof}[Proof of Theorem \ref{dualthm}] (a) By assumption if $\eta\sim P_{J,K}$, then $T_{J,K}\eta\sim P_{J,K}$, and so $T\eta\in \mathcal{C}^{rev}_{J,K}$, $P_{J,K}$-a.s. Iterating this and appealing to countability yields $T_{J,K}^t\eta\in \mathcal{C}^{rev}_{J,K}$ for all $t\geq 0$, $P_{J,K}$-a.s. Moreover, from Proposition \ref{tinvprp}, we also obtain that $T_{J,K}^{-1}\eta\sim P_{J,K}$, and so we can argue as before to extend the previous conclusion to $T_{J,K}^t\eta\in \mathcal{C}^{rev}_{J,K}$ for all $t\in\mathbb{Z}$, $P_{J,K}$-a.s. Thus we have established $P_{J,K}\in\mathcal{P}^{inv}_{J,K}$. Now, recall from \eqref{djkident} that $\mathcal{D}_{J,K}\circ T_{J,K}=\theta\circ\mathcal{D}_{J,K}$ on $\mathcal{C}^{inv}_{J,K}$. Thus, defining $P_{K,J}:= P_{J,K} \circ \mathcal{D}_{J,K}^{-1}\circ\theta$, we find that
\[P_{K,J}\circ \theta^{-1}=P_{J,K}\circ \mathcal{D}_{J,K}^{-1}=P_{J,K}\circ T_{J,K}^{-1}\circ \mathcal{D}_{J,K}^{-1}\circ\theta=P_{J,K}\circ \mathcal{D}_{J,K}^{-1}\circ\theta=P_{K,J},\]
and so we can apply Lemma \ref{pinvlmm} to deduce that $P_{J,K}\in\tilde{\mathcal{P}}^{inv}_{J,K}$. (NB. The latter lemma deals with the only case in which $\tilde{\mathcal{P}}^{inv}_{J,K}$ is not equal to ${\mathcal{P}}^{inv}_{J,K}$.)

(b) We already established the `only if' part of the claim in the proof of part (a), and so we need to prove the converse. Moreover, we note that if $P_{K,J}\circ\theta^{-1}=P_{K,J}$, then Lemma \ref{pinvlmm} gives us that $P_{J,K}\in\tilde{\mathcal{P}}^{inv}_{J,K}$. Hence we also have that $P_{K,J}\in\tilde{\mathcal{P}}_{K,J}^{inv}$ and $P_{J,K}=P_{K,J}\circ\mathcal{D}_{K,J}^{-1}\circ\theta$ (see comment below \eqref{djkpdef}). It follows that
\begin{align*}
P_{J,K}\circ T_{J,K}^{-1}&=P_{J,K}\circ \mathcal{D}_{J,K}^{-1}\circ\theta^{-1}\circ \mathcal{D}_{J,K}=P_{K,J}\circ\theta^{-2}\circ \mathcal{D}_{J,K}\\
&=P_{K,J}\circ\theta^{-1}\circ \mathcal{D}_{J,K}
=P_{K,J}\circ\mathcal{D}_{J,K}^{-1}\circ\theta
=P_{J,K},
\end{align*}
where we have applied \eqref{djkident}, the definition of $P_{K,J}$, the invariance of $P_{K,J}$ under $\theta$, \eqref{dualinv}, and the identity $P_{J,K}=P_{K,J}\circ\mathcal{D}_{K,J}^{-1}\circ\theta$, respectively.

(c) Again, under the invariance of either side, then we know that both $P_{J,K}\in\tilde{\mathcal{P}}^{inv}_{J,K}$ and $P_{K,J}\in\tilde{\mathcal{P}}^{inv}_{K,J}$, and so we only need to deal with the underlying spaces $\tilde{\mathcal{C}}_{J,K}^{inv}$ and $\tilde{\mathcal{C}}_{K,J}^{inv}$. First suppose that $T_{J,K}$ is ergodic for $P_{J,K}$. To check that $\theta$ is ergodic for $P_{K,J}$, we are required to check that if a measurable subset $A\subseteq\tilde{\mathcal{C}}^{inv}_{K,J}$ satisfies $\theta^{-1}(A)=A$, then $P_{K,J}(A)\in\{0,1\}$. To this end, note that
\[T_{J,K}^{-1}\circ\mathcal{D}_{J,K}^{-1}\circ\theta(A)=
\mathcal{D}_{J,K}^{-1}\circ\theta^{-1}\circ\theta(A)=
\mathcal{D}_{J,K}^{-1}\circ\theta\circ\theta^{-1}(A)=
\mathcal{D}_{J,K}^{-1}\circ\theta(A),\]
i.e.\ $\mathcal{D}_{J,K}^{-1}\circ\theta(A)$ is invariant for $T_{J,K}$, and hence the ergodicity of $T_{J,K}$ for $P_{J,K}$ implies $P_{K,J}(A)=P_{J,K}(\mathcal{D}_{J,K}^{-1}\circ\theta(A))\in\{0,1\}$. Similarly, if $\theta$ is ergodic for $T_{K,J}$, and a measurable subset $A\subseteq\tilde{\mathcal{C}}^{inv}_{J,K}$ satisfies $T_{J,K}^{-1}(A)=A$, then
\begin{eqnarray*}
\theta^{-1}\circ\mathcal{D}_{K,J}^{-1}\circ\theta(A)&=&\theta^{-2}\circ\mathcal{D}_{J,K}(A)=\theta^{-1}\circ\mathcal{D}_{J,K}\circ T_{J,K}^{-1}(A)\\
&=&\theta^{-1}\circ\mathcal{D}_{J,K}(A)
=\mathcal{D}_{K,J}^{-1}\circ\theta(A),
\end{eqnarray*}
which implies $P_{J,K}(A)=P_{K,J}(\mathcal{D}_{K,J}^{-1}\circ\theta(A))\in\{0,1\}$, and this completes the proof.
\end{proof}

We now give a lemma that will be useful for studying i.i.d.\ measures, as we do in the next section.

\begin{lmm}\label{indep}
Suppose that $\eta$ is a random configuration with distribution whose support is contained within $\mathcal{C}_{J,K}^{inv}$. If $(\eta_n)_{n \le 0}$ and $(\eta_n)_{n \ge 1}$ are independent, then so are $(T^tW_0)_{t \ge 0}$ and $(T^tW_0)_{t \le -1}$.
\end{lmm}
\begin{proof} We have from Lemma \ref{unique} that $(W_n)_{n\leq 0}$ is $(\eta_n)_{n\leq 0}$ measurable, and hence from \eqref{carrierdef}, we also have that $((T\eta)_n)_{n\leq 0}$ is $(\eta_n)_{n\leq 0}$ measurable. Iterating this yields that $(T^tW_0)_{t \ge 0}$ is $(\eta_n)_{n\leq 0}$ measurable. Conversely, write $V$ for the canonical carrier associated with $R\eta$. We then similarly have that $(T^tV_0)_{t \ge 0}$ is $(R\eta_n)_{n\leq 0}$ measurable, where we use the notation $T^tV$ for the canonical carrier of $T^tR\eta=RT^{-t}\eta$, with the equality here being a consequence of \eqref{tinv}. Now, using the notation of the proof of Proposition \ref{tinvprp}, we have that the reversed carrier $\overline{T^tV}$ is the canonical carrier for the configuration $T^{-1}R^2T^{-t}\eta=T^{-(t+1)}\eta$. In particular, it follows that $(T^tW_0)_{t \le -1}=(T^{-(t+1)}V_0)_{t \le -1}$ is $(\eta_{1-n})_{n\leq 0}$ measurable. The result follows.
\end{proof}

As the last result of this section, we prove the claim in the introduction about non-canonical carriers leading to trivial dynamics when $J<K$. To this end, let us suppose we have a procedure for choosing a unique carrier $\hat{W}$ for each $\eta \in \mathcal{C}_{J,K}^{\exists}$ which extends the definition of $W$, i.e.\ $\hat{W}=W$ on $\mathcal{C}_{J,K}^{can}$. (In what follows, we assume this choice is measurable with respect to the underlying probability space.) We can then define the generalized dynamics $\hat{T}$ and the reversible set $\hat{\mathcal{C}}_{J,K}^{rev}$ by setting $\hat{T}\eta  =T^{\hat{W}}\eta$ for all $\eta \in \mathcal{C}_{J,K}^{\exists}$, and $\hat{\mathcal{C}}_{J,K}^{rev} :=\{ \eta \in \mathcal{C}_{J,K}^{\exists}:\:R\eta \in \mathcal{C}_{J,K}^{\exists},\: \hat{T}R\hat{T}R\eta=R\hat{T}R\hat{T}\eta =\eta\}$. We then have the following result.

\begin{prp}\label{p33} Suppose $J <K$, and $\eta$ is a random configuration with distribution supported on $\hat{\mathcal{C}}_{J,K}^{rev} \backslash\mathcal{C}_{J,K}^{rev}$. If the distribution of $\eta$ is invariant under $\hat{T}$, namely $\hat{T}\eta \buildrel{d}\over{=}\eta$, then $\hat{T}\eta =\sigma_J\eta$, almost-surely.
\end{prp}
\begin{proof}
If $\hat{T}\eta \buildrel{d}\over{=} \eta$, then we can define $\hat{T}^t\eta$ for all $t \ge 0$, almost-surely. Also, since $\eta$ is supported on $\hat{\mathcal{C}}_{J,K}^{rev}$, $\hat{T}^t \eta$ is well-defined for all $t \le 0$, almost-surely, we where write $\hat{T}^{-1}\eta:=R\hat{T}R\eta$ and $\hat{T}^{t-1}\eta:=\hat{T}^{-1}\hat{T}^t\eta$. (Cf.\ the proof of Theorem \ref{dualthm}(a).) Furthermore, we will denote $\hat{W}(\hat{T}^{t}\eta)$ (which might be a canonical carrier for $\hat{T}^t\eta$) by $\hat{T}^t\hat{W}$. Now, since $\eta$ is supported on $\hat{\mathcal{C}}_{J,K}^{rev} \backslash \mathcal{C}_{J,K}^{rev}$, it holds that, almost-surely, either $\eta \notin \mathcal{C}_{J,K}^{can}$ or $R\eta \notin \mathcal{C}_{J,K}^{can}$. If $\eta \notin \mathcal{C}_{J,K}^{can}$, then by Lemma \ref{keylemma}(b) for $J <K$, $ \inf_{t \ge 0} B_{\hat{T}^t\hat{W}} >-\infty$. Hence there exists an $N \in \mathbb{Z}$ such that $ J \le \hat{T}^t \eta_{n+1} + \hat{T}^t\hat{W}_n \le K$ for all $n \le N$ and $t \ge 0$. It follows that, for each $n \le N$, there exists $q \in \{0,1,\dots, \lfloor\frac{J}{2}\rfloor\}$ and $i \in \{0,1\}$ such that $B_{q,i,R(\hat{T}^t\hat{W}_n)_{t \in \mathbb{Z}}} >-\infty$, and this implies in turn that $B_{R(\hat{T}^t\hat{W}_n)_{t \in \mathbb{Z}}}>\infty$. Since $\hat{T}\eta \buildrel{d}\over{=} \eta$, we can appeal to shift invariance as in the proof of Lemma \ref{pinvlmm} to conclude from this that $B_{R(\hat{T}^t\hat{W}_n)_{t \in \mathbb{Z}}}=\infty$, which yields $J \le \hat{T}^t \eta_n + \hat{T}^t\hat{W}_{n-1} \le K$ for all $n \le N$ and $t \in \mathbb{Z}$. Thus we find that $(\hat{T}^{t+1}\eta_{n+1})_{t \in \mathbb{Z}}$ is a carrier for $(\hat{T}^t\hat{W}_n)_{t \in \mathbb{Z}}$, but not a canonical one for $n \le N$. By Lemma \ref{keylemma}(a), we can extend this to the conclusion that $(\hat{T}^t\hat{W}_n)_{t \in \mathbb{Z}} \notin \mathcal{C}_{K,J}^{can}$ for all $n \in \mathbb{Z}$. So, by the same argument, $J \le \hat{T}^t \eta_n + \hat{T}^t\hat{W}_{n-1} \le K$ for all $n \in \mathbb{Z}$ and $t \in \mathbb{Z}$, which tells us that $\hat{T}\eta_n=J-\eta_n$ for all $n$. The same argument applies if $R\eta \notin \mathcal{C}_{J,K}^{can}$, and thus we complete the proof.
\end{proof}

\section{I.i.d.\ invariant configurations}\label{iidsec}

In this section we take up the study of i.i.d.\ measures, and in particular tackle the question of which of these are invariant for the BBS($J$,$K$) dynamics. As well as proving the characterisation of invariance in terms of detailed balance that was stated as Theorem \ref{iidthm}, and the result that invariant i.i.d.\ measures are also ergodic for the BBS($J$,$K$) transformation of Corollary \ref{iidcor}, we present our complete characterisation of invariant i.i.d.\ measures in Theorem \ref{iidclass}. We recall from the introduction the space of probability measures on $\{0,1,\dots,J\}$, that is $\mathcal{M}_{J,K}$, the subsets $\mathcal{M}_{J,K}^{rev}$ and $\mathcal{M}_{J,K}^{inv}$ from \eqref{mrevdef} and \eqref{minvdef}, respectively, the duality map $\mathcal{D}^{\mu}_{J,K}$ from \eqref{dmudef}, and the quantity $r(\mu_{J,K})$ from \eqref{rjkdef}. As the first result of the section, we give a description of $\mathcal{M}_{J,K}^{rev}$, and show that it is equal to $\mathcal{M}^{inv}_{J,K}$. For this, we introduce notation for the mean of a measure, by setting $m(\mu_{J,K}):= \sum_{a=0}^{J}a\mu_{J,K}(a)$.

\begin{prp}\label{mrevprp} It holds that
\[\mathcal{M}_{J,K}^{rev}=\left\{
                            \begin{array}{ll}
                            \left\{\mu_{J,K} \in \mathcal{M}_{J,K}:\:2r(\mu_{J,K})<K\right\}, & \hbox{if $J>K$;} \\
                            \mathcal{M}_{J,K}, & \hbox{if $J=K$;} \\
                            \left\{\mu_{J,K} \in \mathcal{M}_{J,K}:\:2m(\mu_{J,K}) < J\right\}, & \hbox{if $J<K=\infty$;}\\
                            \left\{\mu_{J,K} \in \mathcal{M}_{J,K}:\:2r(\mu_{J,K}) <J\right\}, & \hbox{if $J<K<\infty$.}
                            \end{array}
                          \right.\]
Moreover, $\mathcal{M}_{J,K}^{rev}=\mathcal{M}_{J,K}^{inv}$.
\end{prp}
\begin{proof} In the case $J>K$, note that $\mu_{J,K}^{\otimes \mathbb{Z}}(\mathcal{C}_{J,K}^{r(\mu_{J,K})})=1$ (where we recall the notation $\mathcal{C}^r_{J,K}$ from \eqref{crdef}). Hence, if $2r(\mu_{J,K}) \ge {K}$, then, from Lemma \ref{le2r}, $\mu_{J,K}^{\otimes \mathbb{Z}}(\mathcal{C}_{J,K}^{can})=0$, and so $\mu_{J,K} \notin \mathcal{M}_{J,K}^{rev}$. Conversely, if $2r(\mu_{J,K}) < {K}$, then $K-r(\mu_{J,K}) >r(\mu_{J,K})$ and, if $J<\infty$, $J-K+r(\mu_{J,K}) <J-r(\mu_{J,K})$. Hence at least one of $\mu_{J,K}([K-r(\mu_{J,K}),J])$ or $\mu_{J,K}([0,J-K+r(\mu_{J,K})])$ is strictly less than 1. It follows that,
\begin{eqnarray*}
\lefteqn{\mu_{J,K}^{\otimes \mathbb{Z}}\left(\mathcal{C}_{J,K}^{r(\mu_{J,K}),alt}\right)}\\
& \le&  \mu_{J,K}^{\otimes \mathbb{Z}}\left(\cup_{n \in \mathbb{Z}, i\in\{0,1\}}\cap_{m \le n}\{\eta_{2m+i} \ge K-r(\mu_{J,K}), \eta_{2m+1+i} \le J-K+r(\mu_{J,K})\}\right)=0,
\end{eqnarray*}
where $\mathcal{C}_{J,K}^{r,alt}$ was defined at \eqref{cralt}. Therefore $\mu_{J,K}^{\otimes \mathbb{Z}}(\mathcal{C}_{J,K}^{r(\mu_{J,K})} \cap( \mathcal{C}_{J,K}^{r(\mu_{J,K}),alt})^c)=1$, and Proposition \ref{ch1} thus yields $\mu_{J,K}^{\otimes \mathbb{Z}}(\mathcal{C}_{J,K}^{can})=1$. Since the distribution of $\mu_{J,K}^{\otimes \mathbb{Z}}$ is invariant under the reflection $R$, we further obtain that $\mu_{J,K}^{\otimes \mathbb{Z}}(\mathcal{C}_{J,K}^{rev})=1$. Moreover, from Lemma \ref{12345}, $\mu_{J,K}^{\otimes \mathbb{Z}}(\mathcal{C}_{J,K}^{inv})=1$. For $J=K$, the claim is obvious from Proposition \ref{ch2}. Next, suppose that $J<K=\infty$. Note that, under $\mu_{J,K}^{\otimes \mathbb{Z}}$, $\mathbf{E}(S_n-S_{n-1})=\mathbf{E}(J-2\eta_n)=J-2m(\mu_{J,K})$ and $(S_n-S_{n-1})_{n\in\mathbb{Z}}$ is i.i.d. In particular, $S$  is a random walk, and $\limsup_{n \to -\infty}S_n =-\infty$ with $\mu_{J,K}^{\otimes \mathbb{Z}}$-probability one if and only if $J-2m(\mu_{J,K})>0$. Hence Proposition \ref{ch3} and the invariance of the distribution of $\mu_{J,K}^{\otimes \mathbb{Z}}$ under the reflection $R$ allow us to conclude that $\mu_{J,K} \in \mathcal{M}_{J,K}^{rev}$ if and only if $2m(\mu_{J,K}) < J$. Moreover, if $\mu_{J,K} \in \mathcal{M}_{J,K}^{rev}$, then $\lim_{n \to \pm \infty}\frac{S_n}{n}=J-2m(\mu_{J,K}) >0$, and Corollary \ref{invsubs} yields $\mu_{J,K} \in \mathcal{M}_{J,K}^{inv}$. Finally, suppose $J <K<\infty$. Since $S$ is a random walk, if the distribution of $S_n-S_{n-1}=J-2\eta_n$ is not a delta measure on $0$, then $\limsup_{n \to \pm \infty}|S_n|=\infty$. Hence if $2r(\mu_{J,K}) <J$, then Corollary \ref{invsubs2} gives us that $\mu_{J,K} \in \mathcal{M}_{J,K}^{inv}$, and in particular $\mu_{J,K} \in \mathcal{M}_{J,K}^{rev}$. On the other hand, if $2r(\mu_{J,K}) =J$, then $\mu_{J,K}$ is a delta measure on $\frac{J}{2}$, and so $\limsup_{n \to -\infty}\tilde{S}_n=\liminf_{n \to -\infty}\tilde{S}_n=0$, $\mu_{J,K}^{\otimes \mathbb{Z}}$-almost-surely. From Proposition \ref{ch4}, it follows that $\mu_{J,K}^{\otimes \mathbb{Z}}(\mathcal{C}_{J,K}^{can})=0$, which means that $\mu_{J,K} \notin \mathcal{M}_{J,K}^{rev}$.
\end{proof}

In our next result, we relate invariance under $T_{J,K}$ with the detailed balance equation \eqref{dbalance}. In particular, we show that, when the dual measure considered in the latter equation is given by the distribution of the current, the two conditions are equivalent. We moreover obtain that the dual product measure is invariant under $T_{K,J}$.

\begin{prp}\label{dbalance1}
Fix $J,K\in\mathbb{N}\cup\{\infty\}$.  Let $\mu_{J,K}\in\mathcal{M}^{rev}_{J,K}$, and define $\mu_{K,J}:=\mathcal{D}^\mu_{J,K}(\mu_{J,K})$. It is then the case that
\begin{eqnarray}
\lefteqn{\mu_{J,K}^{\otimes \mathbb{Z}} \circ T_{J,K}^{-1} = \mu_{J,K}^{\otimes \mathbb{Z}}}\label{line1}\\
&\Leftrightarrow & \mu_{J,K}\times\mu_{K,J}\circ F_{J,K}^{-1}=\mu_{J,K}\times\mu_{K,J}\label{line2}\\
&\Rightarrow & \mu_{K,J}\in\mathcal{M}^{rev}_{K,J}\mbox{ and }\mu_{K,J}^{\otimes \mathbb{Z}} \circ T_{K,J}^{-1} = \mu_{K,J}^{\otimes \mathbb{Z}}.\label{line3}
\end{eqnarray}
\end{prp}
\begin{proof} We start by showing that \eqref{line1} implies \eqref{line3}. Suppose $\mu_{J,K}\in\mathcal{M}^{rev}_{J,K}$ and that $\mu_{J,K}^{\otimes \mathbb{Z}} \circ T_{J,K}^{-1} = \mu_{J,K}^{\otimes \mathbb{Z}}$. From Theorem \ref{dualthm}(a), we  then have that $P_{K,J}:=\mathcal{D}^P_{J,K}(\mu_{J,K}^{\otimes \mathbb{Z}})$ is well-defined and satisfies $P_{K,J} \in \tilde{\mathcal{P}}_{K,J}^{inv}$. Moreover, since $\mu_{J,K}^{\otimes \mathbb{Z}} \circ T_{J,K}^{-1} = \mu_{J,K}^{\otimes \mathbb{Z}}$ and $\mu_{J,K}^{\otimes \mathbb{Z}} \circ \theta^{-1} = \mu_{J,K}^{\otimes \mathbb{Z}}$, Theorem \ref{dualthm} (b) yields that $P_{K,J}  \circ T_{K,J}^{-1} =  P_{K,J}$ and $P_{K,J}  \circ \theta^{-1} =  P_{K,J}$ both hold. Hence to conclude \eqref{line3}, it is sufficient to show that $P_{K,J} = \mu_{K,J}^{\otimes \mathbb{Z}}$. By definition and stationarity under $\theta$, $P_{K,J}$ is the distribution of $(T_{J,K}^tW_0)_{t \in \mathbb{Z}}$ and $T_{J,K}^tW_0\sim \mu_{K,J}$ for all $t \in \mathbb{Z}$. Moreover, from Lemma \ref{indep}, $(T_{J,K}^tW_0)_{t \ge0}$ and $(T_{J,K}^tW_0)_{t \le -1}$ are independent, and together with stationarity, this gives that $(T_{J,K}^tW_0)_{t \in \mathbb{Z}}$ is an independent sequence. This confirms the implication.

We next show that \eqref{line1} implies \eqref{line2}. Recall from \eqref{prodmeasure} that the joint distribution of $(\eta_1,W_0)$ is $ \mu_{J,K}\times\mu_{K,J}$. Assuming that $\mu_{J,K}^{\otimes \mathbb{Z}} \circ T_{J,K}^{-1} = \mu_{J,K}^{\otimes \mathbb{Z}}$ and also appealing to the stationarity of $\mu_{J,K}^{\otimes \mathbb{Z}}$ under the spatial shift, we find that $T_{J,K}\eta_1\sim\mu_{J,K}$ and $W_1 \sim \mu_{K,J}$. Since $F_{J,K}(\eta_1,W_0)=(T_{J,K}\eta_1,W_1)$, to establish \eqref{line2} it is enough to show that $T_{J,K}\eta_1$ and $W_1$ are independent. Now, as $\mu_{J,K}^{\otimes \mathbb{Z}}(\mathcal{C}_{J,K}^{rev})=1$, we can argue as in the proof of Proposition \ref{tinvprp} to show that $(W_{-n})_{n\in\mathbb{Z}}$ is the canonical carrier of $RT_{J,K}\eta$, and thus $W_1$ is measurable with respect to $(T_{J,K}\eta_n)_{n \ge2}$. Given that, under \eqref{line1}, $(T_{J,K}\eta_n)_{n \ge2}$ and $T_{J,K}\eta_1$ are independent, we obtain that $W_1$ and $T_{J,K}\eta_1$ are independent, as desired.

Finally, we show \eqref{line2} implies \eqref{line1}. To this end, suppose that \eqref{line2} holds. Under $\mu_{J,K}^{\otimes \mathbb{Z}}$, $(W_n)_{n\in\mathbb{Z}}$ is a stationary process and $W_0 \sim \mu_{K,J}$, so $W_n \sim \mu_{K,J}$ for all $n$. Hence, for any $n$, the joint distribution of $(\eta_n,W_{n-1})$ is $ \mu_{J,K}\times\mu_{K,J}$. Moreover, the detailed balance equation \eqref{line2} shows that, for all $n$: $T_{J,K}\eta_n \sim \mu_{J,K}$, and $T_{J,K}\eta_n$ and $W_n$ are independent. Moreover, the latter independence implies in turn that $T_{J,K}\eta_n$ and $\sigma(W_n, (\eta_m)_{m \ge n+1})$ are independent. Since $(T_{J,K}\eta_m)_{m \ge n+1}$ is measurable with respect to $\sigma(W_n, (\eta_m)_{m \ge n+1})$, we find that $T_{J,K}\eta_n$ and $(T_{J,K}\eta_m)_{m \ge n+1}$ are independent. Hence we conclude that $(T_{J,K}\eta_n)_{n} \sim \mu_{J,K}^{\otimes \mathbb{Z}}$, and so $\mu_{J,K}^{\otimes \mathbb{Z}} \circ T_{J,K}^{-1} = \mu_{J,K}^{\otimes \mathbb{Z}}$.
\end{proof}

\begin{proof}[Proof of Corollary \ref{iidcor}] As in the proof of Proposition \ref{dbalance1}, we have under the assumptions of the corollary that $\mathcal{D}^P_{J,K}(\mu_{J,K}^{\otimes\mathbb{Z}})=\mu_{K,J}^{\otimes\mathbb{Z}}$. In particular, this implies that $\mathcal{D}^P_{J,K}(\mu_{J,K}^{\otimes\mathbb{Z}})$ is invariant and ergodic under $\theta$. Hence Theorem \ref{dualthm}(c) yields the result.
\end{proof}

Whilst the previous proposition gives some insight into the role of the detailed balance equation, it is slightly unsatisfactory in that it depends on the dual measure being assumed to be the distribution of the current. As was stated in Theorem \ref{iidthm}, we can avoid this assumption when we include instead conditions that link the support of the original and dual measures under consideration. We now prove the latter version of the result, which will play a crucial role in characterizing all invariant measures of product form.

\begin{proof}[Proof of Theorem \ref{iidthm}]
The claim is obvious for the case $J=K$, so we assume $J\neq K$. First suppose $\mu_{J,K}^{\otimes \mathbb{Z}} \circ T_{J,K}^{-1} = \mu_{J,K}^{\otimes \mathbb{Z}}$; we then have from Theorem \ref{dualthm}(a) that $\mu_{J,K}^{\otimes \mathbb{Z}}\in\tilde{\mathcal{P}}_{J,K}^{inv}$. From Proposition \ref{dbalance1}, to prove the `only if' direction of the first part of the theorem, we only need to show that if $\mu_{K,J} := \mathcal{D}^{\mu}_{J,K}(\mu_{J,K})$, then $r(\mu_{J,K})=r(\mu_{K,J})$, and also $\underline{r}(\mu_{J,K})=\underline{r}(\mu_{K,J})$ when $\max\{J,K\}=\infty$. If $J>K$, then since $\mu_{J,K}^{\otimes \mathbb{Z}}(\mathcal{C}_{J,K}^{r(J,K)})=1$, we have from Lemma \ref{le2r} that $K>2r(\mu_{J,K})$. Consequently we can appeal to Lemma \ref{>2r}, and in particular \eqref{equalsr} and \eqref{equalsr2} applied to the canonical carrier imply the desired conclusion. Next, suppose $J<K$. In this case, Proposition \ref{dbalance1} yields that $\mu_{K,J}^{\otimes \mathbb{Z}} \circ T_{K,J}^{-1} = \mu_{K,J}^{\otimes \mathbb{Z}}$. Moreover, from the proof of Proposition \ref{dbalance1}, we have that $\mathcal{D}^{P}_{J,K}(\mu_{J,K}^{\otimes \mathbb{Z}})=\mu_{K,J}^{\otimes \mathbb{Z}}$, and since the latter map is a bijection between $\tilde{\mathcal{P}}_{J,K}^{inv}$ and $\tilde{\mathcal{P}}_{K,J}^{inv}$ with inverse $\mathcal{D}^{P}_{K,J}$, it follows that $\mu_{J,K}^{\otimes \mathbb{Z}}=\mathcal{D}^{P}_{K,J}(\mu_{K,J}^{\otimes \mathbb{Z}})=\mathcal{D}^{\mu}_{K,J}(\mu_{K,J})^{\otimes \mathbb{Z}}$. This implies in turn that $\mu_{J,K}= \mathcal{D}^{\mu}_{K,J}(\mu_{K,J})$. Thus we can apply the result for the case $J>K$ to complete this part of the proof.

We now prove the `if' direction. From Proposition \ref{dbalance1}, to establish this, and the remaining claims of the theorem, it is enough to show that $\mathcal{D}^{\mu}_{J,K}(\mu_{J,K})=\mu_{K,J}$. Under $\mu_{J,K}^{\otimes \mathbb{Z}}$, $(W_n)_{n\in\mathbb{Z}}$ is a stationary Markov process with the transition probabilities
\begin{equation}\label{wnprobs}
\mathbf{P}\left(W_n=b\:\vline\:W_{n-1}=a\right)=\mu_{J,K}\left(\{x:\:F^{(2)}(x,a)=b\}\right),\qquad\forall a,b\in\{0,1,\dots,K\}.
\end{equation}
By the detailed balance equation and the involutive property of $F_{J,K}$, it is straightforward to check from this expression that $\mu_{K,J}$ is a reversible measure for this Markov process. To prove $W_0 \sim \mu_{K,J}$, we show that if $r:=r(\mu_{J,K})=r(\mu_{K,J})$, and also $\underline{r}(\mu_{J,K})=\underline{r}(\mu_{K,J})$ when $K=\infty$, then $W_0$ is concentrated on the subset $\{r,r+1,\dots,K-r\}$, and the Markov process with the state space $\{r,r+1,\dots,K-r\}$ and the transition probabilities \eqref{wnprobs} has at most one stationary measure. If $J>K$, then since $W$ is a canonical carrier, $W_n$ is concentrated on $\{r,r+1,\dots,K-r\}$ (recall Lemmas \ref{le2r} and \ref{>2r}). Moreover, $\mathbf{P}(W_n=r\:|\:W_{n-1}=a) \ge \mu_{J,K}(r)$ and $\mathbf{P}(W_n=K-r\:|\:W_{n-1}=a)\ge\mu_{J,K}(J-r)$ for all $a \in \{r,r+1,\dots,K-r\}$. Hence either $\mathbf{P}(W_n=r\:|\:W_{n-1}=a)>0$ for all $a \in \{r,r+1,\dots,K-r\}$, or $\mathbf{P}(W_n=K-r\:|\:W_{n-1}=a)>0$ for all $a \in \{r,r+1,\dots,K-r\}$, and so the stationary measure is unique if it exists. We next consider the case $J <K$. When $K=\infty$, recall from the proof of Proposition \ref{ch3} that $W_n=M_n-S_n+\frac{J}{2}$, where $M_n=\max_{m \le n}\tilde{S}_m$. Hence $W_n \ge \tilde{S}_n-S_n+\frac{J}{2}=\eta_n$, which implies $W_n$ is concentrated on $\{r,r+1,\dots\}$. Moreover, since $r=r(\mu_{K,J})=\underline{r}(\mu_{K,J})=\underline{r}(\mu_{J,K})$, it holds that $\mu_{J,K}^{\otimes \mathbb{Z}} (\eta_{0}=\eta_1=\dots=\eta_N=r)>0$ for any $N \in \mathbb{N}$. By Proposition \ref{mrevprp} we also know that $r <\frac{J}{2}$, and so for each $a \in \{r,r+1,\dots,\}$ there exists an $N\in \mathbb{N}$ such that
\begin{eqnarray*}
\mathbf{P}\left(W_{n+N}=r\:|\:W_n=a\right) &\ge &\mathbf{P}\left(\eta_{n+1}=\dots=\eta_{n+N}=r, \eta_{n+N}=W_{n+N}|W_n=a\right)\\
&= & \mathbf{P}\left(\eta_{n+1}=\dots=\eta_{n+N}=r,\tilde{S}_{n+N}=M_{n+N}|W_n=a\right)\\
&= & \mathbf{P}\left(\eta_{n+1}=\dots=\eta_{n+N}=r|W_n=a\right) >0.
\end{eqnarray*}
Hence there is at most one stationary measure. For the case $K<\infty$, one can check from \eqref{pathm} that any carrier $Y$ satisfies $Y_n \in \{r,\dots,K-r\}$ for all $n\in\mathbb{Z}$, $\mu_{J,K}^{\otimes \mathbb{Z}}$-almost-surely. However, in this case $\mathcal{C}_{J,K}^{can} \subseteq \mathcal{C}_{J,K}^{\exists !}$ holds by Proposition \ref{ch4}, and so $W \in \{r,\dots,K-r\}$. Uniqueness of the stationary measure is shown similarly to the case $K=\infty$.
\end{proof}

Having this characterization of the collection of invariant measures $\mathcal{I}_{J,K}^{inv}$ (as defined at \eqref{iinvdef}), we now work towards its explicit description. Since for the case $J=K$, $\mathcal{M}_{J,K}=\mathcal{I}_{J,K}^{inv}$ holds, we will focus on the case $J \neq K$. Our next step is to apply the reducibility property for $F_{J,K}$ from \eqref{reducibility} (and the empty box-ball duality of \eqref{spaceball}) to obtain a reducibility property for measures. In particular, we will show that if a measure $\mu_{J,K} \in \mathcal{M}_{J,K}$ satisfies, together with a dual measure, the detailed balance equation and has support bounded away from 0 and $J$ (in the sense that $r(\mu_{J,K})>0$), then the detailed balance equation is also satisfied by a pair of corresponding dual measures on smaller state spaces. To state the result precisely, we need some further notation. For any $\mu_{J,K} \in \mathcal{M}_{J,K}$, we define $\tilde{\mu}_{J,K} \in \mathcal{M}_{J-2r(\mu_{J,K}),K-2r(\mu_{J,K})}$ by setting
\[\tilde{\mu}_{J,K}(a)=\left\{
\begin{array}{ll}
  \mu_{J,K}\left(a+r(\mu_{J,K})\right), & \mbox{if $r(\mu_{J,K})=\underline{r}(\mu_{J,K})$},\\
  \mu_{J,K}\left(\sigma_J(a+r(\mu_{J,K}))\right), & \mbox{otherwise}.
\end{array}
\right.
\]
Note that if we define $\mathcal{E}_r:\mathcal{M}_{J,K} \to \mathcal{M}_{J+2r,K+2r}$ for $r \in \mathbb{Z}_+$ by setting
\[\mathcal{E}_r(\mu)(a)=\left\{\begin{array}{ll}
                              \mu(a-r), & \mbox{if }a \in \{ r, \dots,J+r\},\\
                              0, & \mbox{otherwise},\end{array}\right.\]
then $\mu_{J,K}=\mathcal{E}_{r(\mu_{J,K})}(\tilde{\mu}_{J,K})$ if $r(\mu_{J,K})=\underline{r}(\mu_{J,K})$, else $\mu_{J,K}=\sigma_J\mathcal{E}_{r(\mu_{J,K})}(\tilde{\mu}_{J,K})$. We moreover observe that $\tilde{\mu}_{J,K}(0)>0$, and hence the following lemma will allow us to concentrate on measures satisfying the latter property when looking for solutions to the detailed balance equation.

\begin{lmm}\label{reducible} Suppose $J \neq K$ and $\mu_{J,K} \in \mathcal{M}_{J,K}^{rev}$. If  $(\tilde{J},\tilde{K})=(J-2r(\mu_{J,K}),K-2r(\mu_{J,K}))$, then $\min\{\tilde{J},\tilde{K}\} \ge 1$.  Moreover, the following two conditions are equivalent:\\
(a) there exists $\mu_{K,J} \in \mathcal{M}_{K,J}$ such that the detailed balance equation \eqref{dbalance} holds, $r(\mu_{J,K})=r(\mu_{K,J})$, and also $\underline{r}(\mu_{J,K})=\underline{r}(\mu_{K,J})$ when $\max\{J,K\}=\infty$;\\
(b) there exists $\tilde{\mu}_{K,J}'  \in \mathcal{M}_{\tilde{K},\tilde{J}}$ such that
\begin{equation}\label{dbother}
\tilde{\mu}_{J,K} \times \tilde{\mu}_{K,J}'  \circ F_{\tilde{J},\tilde{K}}^{-1} = \tilde{\mu}_{J,K} \times \tilde{\mu}_{K,J}'
\end{equation}
holds, $r(\tilde{\mu}_{K,J}')=0$, and also $\underline{r}(\mu_{J,K})-r(\mu_{J,K})=\underline{r}(\tilde{\mu}_{K,J}')=0$ when $\max\{J,K\}=\infty$.
\end{lmm}
\begin{proof} If $J>K$, then $K > 2r(\mu_{J,K})$ and if $J<K$, then $J > 2r(\mu_{J,K})$ (by Proposition \ref{mrevprp}). Hence $\min\{\tilde{J},\tilde{K}\} \ge 1$. Suppose there exists $\mu_{K,J} \in \mathcal{M}_{K,J}$ such that condition (a) is met. Defining  $\tilde{\mu}_{K,J}'(a):=\mu_{K,J}(a+r(\mu_{J,K}))$ when $r(\mu_{J,K})=\underline{r}(\mu_{J,K})$, and $\tilde{\mu}_{K,J}'(a):=\mu_{K,J}(\sigma_K(a+r(\mu_{J,K})))$ otherwise, we obtain the measure required for (b) to hold. Indeed, checking \eqref{dbother} is an elementary application of \eqref{involution}, \eqref{reducibility} and \eqref{spaceball} (cf.\ Proposition \ref{bsdual}). Moreover, since $r(\mu_{J,K})=r(\mu_{K,J})$, we have that $\tilde{\mu}_{K,J}'(\{0,\tilde{K}\})=\mu_{K,J}(\{r(\mu_{K,J}),{K}-r(\mu_{K,J})\})>0$, and so $r(\tilde{\mu}_{K,J}')=0$. Also, if $J=\infty$ or $K=\infty$, then we use that $\underline{r}(\mu_{K,J})=r(\mu_{K,J})$ (which easily follows from the assumptions) to check that $\tilde{\mu}_{K,J}'(0)=\mu_{K,J}(r(\mu_{K,J}))>0$, and thus $\underline{r}(\tilde{\mu}_{K,J}')=0$. The opposite direction is similarly clear.
\end{proof}

In our next result, we describe solutions of the detailed balance equation when $\mu_{J,K}$ and $\mu_{K,J}$ have full support. For this purpose, we first introduce a variation of the geometric distribution that will arise naturally in the result.

\begin{dfn} For $N\in\mathbb{Z}_+\cup\{\infty\}$, $\alpha>0$, $\beta>0$ and $m\in\mathbb{N}$, we say $X$ has \emph{scaled truncated bipartite geometric distribution} with parameters $N$, $1-\alpha$, $\beta$ and $m$ if
\[\mathbf{P}\left(X=mx\right)=C_{N,1-\alpha,\beta,m}\alpha^{x}\beta^{\iota(x)},\qquad x\in\{0,1,\dots,N\},\]
where $\iota(2x)=0, \iota(2x+1)=1$ and $C_{N,1-\alpha,\beta,m}$ is a normalising constant; in this case we write $X\sim \mathrm{stbGeo}(N,1-\alpha,\beta,m)$. Note that, if $N=\infty$, then we require that $\alpha<1$ for the distribution to be defined. We observe that $\mathrm{stbGeo}(N,1-\alpha,1,1)$ is simply the distribution of the usual parameter $1-\alpha$ geometric distribution conditioned to take a value in $\{0,1,\dots,N\}$.
\end{dfn}

\begin{lmm}\label{geolmm}
Suppose $J \neq K$. Two measures $\mu_{J,K}\in\mathcal{M}_{J,K}$ and $\mu_{K,J} \in \mathcal{M}_{K,J}$ satisfy the detailed balance equation, i.e.\ \eqref{dbalance} holds, and they both have full support, if and only if $\mu_{J,K}$ and $\mu_{K,J}$ are given by $\mathrm{stbGeo}(J,1-\alpha,\beta,1)$ and $\mathrm{stbGeo}(K,1-\alpha,\beta,1)$, respectively, where one of the following conditions is satisfied:\\
(i) $J,K \in \mathbb{N}\cup\{\infty\}$, $\alpha\in(0,1)$, $\beta=1$, \\
(ii) $J,K\in 2\mathbb{N}\cup\{\infty\}$, $\alpha\in(0,1)$, $\beta\in(0,\infty)\backslash\{1\}$, \\
(iii) $J,K \in \mathbb{N}$, $\alpha\geq 1$, $\beta=1$, \\
(iv) $J,K\in 2\mathbb{N}$, $\alpha\geq1$, $\beta\in(0,\infty)\backslash\{1\}$.
\end{lmm}
\begin{proof}
It is elementary to check that if $\mu_{J,K}=\mathrm{stbGeo}(J,1-\alpha,\beta,1)$ and $\mu_{K,J}=\mathrm{stbGeo}(K,1-\alpha,\beta,1)$ with one of the conditions (i)-(iv) holding, then the detailed balance equation holds. Hence it remains to establish the converse. Without loss of generality, we assume $J <K$. To simplify notation, set $\mu:=\mu_{J,K}$ and $\nu:=\mu_{K,J}$. By the detailed balance equation we have
\begin{equation}\label{munu}
\frac{\mu(a)}{\mu(0)}=\frac{\nu(a)}{\nu(0)}
\end{equation}
for all $0 \le a \le \min\{J,K\}$. Moreover, under the condition \eqref{munu}, the detailed balance equation is equivalent to: for all $a\in\{0,1,\dots,J\}$, $b\in\{0,1,\dots,K\}$,
\begin{equation}\label{ee1}
\nu(a)\nu(b)=\left\{
               \begin{array}{ll}
                 \nu(J-a)\nu(2a+b-J), & \hbox{when }J\leq a+b\leq K,\\
                 \nu(b+J-K)\nu(a+K-J), & \hbox{when }a+b\geq K.
               \end{array}
             \right.
\end{equation}
From \eqref{ee1}, it holds that $\nu(J-a)\nu(a+c)=\nu(a)\nu(J-a+c)$ for all $a\in\{0,1,\dots,J\}$ and $c\in\{0,1,\dots,K-J\}$. Therefore, letting $\ell_a:=\frac{\nu(a+1)}{\nu(a)}$, we have, for all $a\in\{0,1,\dots,J\}$ and $c\in\{1,\dots,K-J\}$,
\begin{equation*}
\prod_{b=a}^{a+c-1}\ell_b=\prod_{b=J-a}^{J-a+c-1}\ell_b.
\end{equation*}
In particular, by considering $c=1$, we have $\ell_a=\ell_{J-a}$ for $a\in\{0,1,\dots,J\}$. Moreover, by induction for $c$, we obtain that
\begin{equation}\label{eec}
\ell_{a+c-1}=\ell_{J-a+c-1}
\end{equation}
for all $a\in\{0,1,\dots,J\}$ and $c\in\{1,\dots,K-J\}$. We will apply this relation to deduce the result, noting that to complete the proof it will be sufficient to establish that $\ell_a=\ell_{a+2}$ for all $a \in \{0,1,\dots, K-3\}$, and also, if $J \notin 2\mathbb{N}$ or $K \notin 2\mathbb{N}\cup\{\infty\}$, then $\ell_0=\ell_1$. To prove the first part, let $a=0$ in \eqref{eec} to see that $\ell_{c}=\ell_{J+c}$ for all $c \in \{0,1,\dots,K-J-1\}$. Thus $(\ell_a)_{a=0}^{K-1}$ is periodic with period $J$. Now, suppose $K-J \ge 2$. Applying \eqref{eec} for $c=1$ and $c=2$ then yields $\ell_a=\ell_{J-a}=\ell_{J-(a+1)+1}=\ell_{a+2}$ for any $a \in \{0,1,\dots, J-1\}$. Combining with the periodicity, it follows that $\ell_a=\ell_{a+2}$ for all $a \in \{0,1,\dots, K-3\}$. Moreover, if $J \notin 2\mathbb{N}$, then $\ell_0=\ell_J$ (which follows from \eqref{eec} with $a=0$ and $c=1$) implies $\ell_0=\ell_1$. And, if $J \in 2\mathbb{N}$ and $K \notin 2\mathbb{N} \cup \{\infty\}$, then we can apply \eqref{ee1} to show $\ell_0=\ell_1$. Indeed, letting $a=2$, $b=K-1$ in \eqref{ee1}, we deduce that
\[\prod_{a=0}^{K-J-1}\ell_{a+2}=\frac{\nu(K-J+2)}{\nu(2)}=\frac{\nu(K-1)}{\nu(J-1)}=\prod_{a=0}^{K-J-1}\ell_{a+J-1}.\]
Since the left-hand expression is given by $\ell_0\ell_1\ell_0\ell_1\dots\ell_0$ and the right-hand expression by $\ell_1\ell_0\ell_1\ell_0\dots\ell_1$ (with the same number of terms in both), this implies $\ell_0=\ell_1$, as desired. Finally, we consider the case $K-J=1$. In this case, \eqref{ee1} implies $\ell_a=\ell_{b-1}$ for any $a \in \{0,1,\dots,J\}$ and $b \in \{1,\dots, K\}$ satisfying $a+b \ge K$. In particular, for $a \ge \frac{K}{2}=\frac{J}{2}+\frac{1}{2}$, $\ell_a=\ell_{a-1}$, and so $\ell_a$ is constant for $a \ge \frac{J}{2}-\frac{1}{2}$. Moreover, applying \eqref{eec} with $c=1$ yields that $\ell_a=\ell_{J-a}$ for all $a \in \{0,1,\dots, J\}$, and we thus conclude that $\ell_0=\ell_1=\dots=\ell_K$, which completes the proof.
\end{proof}

We next prove a straightforward generalisation of the preceding lemma.

\begin{lmm}\label{sgeolmm}
Suppose $J \neq K$. Two measures $\mu_{J,K}\in\mathcal{M}_{J,K}$ and $\mu_{K,J} \in \mathcal{M}_{K,J}$ satisfy the detailed balance equation, i.e.\ \eqref{dbalance} holds, and there exists an $m \in \mathbb{N}$ such that $J,K \in m\mathbb{N} \cup \{\infty\}$ and
\begin{equation}\label{suppm}
\left\{a:\:\mu_{J,K}(a)> 0\right\}=m\mathbb{Z}_+ \cap \{0,1,\dots,J\},\quad\left\{a:\:\mu_{K,J}(a)> 0\right\}=m\mathbb{Z}_+ \cap \{0,1,\dots,K\},
\end{equation}
if and only if $\mu_{J,K}$ and $\mu_{K,J}$ are given by $\mathrm{stbGeo}(\frac{J}{m},1-\alpha,\beta,m)$ and $\mathrm{stbGeo}(\frac{K}{m},1-\alpha,\beta,m)$, respectively, where one of the following conditions is satisfied:\\
(i) $J,K \in m\mathbb{N}\cup\{\infty\}$, $\alpha\in(0,1)$, $\beta=1$, \\
(ii) $J,K\in 2m\mathbb{N}\cup\{\infty\}$, $\alpha\in(0,1)$, $\beta\in(0,\infty)\backslash\{1\}$, \\
(iii) $J,K \in m\mathbb{N}$, $\alpha\geq 1$, $\beta=1$, \\
(iv) $J,K\in 2m\mathbb{N}$, $\alpha\geq1$, $\beta\in(0,\infty)\backslash\{1\}$.
\end{lmm}
\begin{proof} As in the proof of Lemma \ref{geolmm}, it is simple to check the `if' direction of the result, and so it remains to check the converse, i.e.\ that if the detailed balance equation and \eqref{suppm} hold, then the measures have the required form. For this, let $\tilde{J}=\frac{J}{m}, \tilde{K}=\frac{K}{m}$, and define ${\mu}'_{J,K} \in \mathcal{M}_{\tilde{J},\tilde{K}}$ and ${\mu}'_{K,J} \in \mathcal{M}_{\tilde{K},\tilde{J}}$ by setting ${\mu}'_{J,K}(a)=\mu_{J,K}(ma)$ and ${\mu}'_{K,J}(a)=\mu_{K,J}(ma)$. Since $F^{(i)}_{J,K}(ma,mb)=mF^{(i)}_{\tilde{J},\tilde{K}}(a,b)$ for $(a,b)\in\{0,1,\dots,\tilde{J}\}\times \{0,1,\dots,\tilde{K}\}$ and $i=1,2$, it is obvious that the ${\mu}'_{J,K}$ and ${\mu}'_{K,J}$ satisfy the detailed balance equation. It is further clear by construction that both ${\mu}'_{J,K}$ and ${\mu}'_{K,J}$ have full support, and so, by Lemma \ref{geolmm}, ${\mu}'_{J,K}=\mathrm{stbGeo}(\tilde{J},1-\alpha,\beta,1)$ and ${\mu}'_{K,J}=\mathrm{stbGeo}(\tilde{K},1-\alpha,\beta,1)$, where the parameters satisfy the appropriate constraints. Hence ${\mu}_{J,K}=\mathrm{stbGeo}(\tilde{J},1-\alpha,\beta,m)$ and ${\mu}_{K,J}=\mathrm{stbGeo}(\tilde{K},1-\alpha,\beta,m)$, with one of the conditions (i)-(iv) being met.
\end{proof}

As a further step, we continue to weaken our assumption on the support of $\mu_{J,K}$.

\begin{thm}\label{m0thm} Suppose $J \neq K$. If $\mu_{J,K}\in\mathcal{M}_{J,K}$ satisfies $\mu_{J,K}(0)>0$, then there exists a measure $\mu_{K,J} \in \mathcal{M}_{K,J}$ such that the detailed balance equation is satisfied, i.e.\ \eqref{dbalance} holds, if and only if one of the following holds:\\
(a) $\mu_{J,K}$ is supported on $\{0,\dots, \lfloor\frac{\min\{J,K\}}{2}\rfloor\}$, in this case
\[\mu_{K,J}(a)=\left\{\begin{array}{ll}
                        \mu_{J,K}(a), & \mbox{for all }a \le \lfloor\frac{\min\{J,K\}}{2}\rfloor,\\
                        0, & \mbox{otherwise;}\end{array}\right.\]
(b)  $\mu_{J,K}$ and $\mu_{K,J}$ are given by $\mathrm{stbGeo}(\frac{J}{m},1-\alpha,\beta,m)$ and $\mathrm{stbGeo}(\frac{K}{m},1-\alpha,\beta,m)$, respectively, where one of the following conditions is satisfied:\\
(i) $J,K \in m\mathbb{N}\cup\{\infty\}$, $\alpha\in(0,1)$, $\beta=1$, $m\in\mathbb{N}$,\\
(ii) $J,K\in 2m\mathbb{N}\cup\{\infty\}$, $\alpha\in(0,1)$, $\beta\in(0,\infty)\backslash\{1\}$, $m\in\mathbb{N}$,\\
(iii) $J,K \in m\mathbb{N}$, $\alpha\geq 1$, $\beta=1$, $m\in\mathbb{N}$,\\
(iv) $J,K\in 2m\mathbb{N}$, $\alpha\geq1$, $\beta\in(0,\infty)\backslash\{1\}$, $m\in\mathbb{N}$.\\
Hence, in any case, $\mu_{K,J}$ also satisfies $\mu_{K,J}(0) >0$.
\end{thm}
\begin{proof} It is elementary to check that if either (a) or (b) hold, then so does \eqref{dbalance}. Hence it remains to show that if \eqref{dbalance} holds for some $\mu_{K,J} \in \mathcal{M}_{K,J}$, then either (a) or (b) hold. To simplify the notation, let $\mu=\mu_{J,K}$ and $\nu=\mu_{K,J}$. First we prove that $\nu(0)>0$. If this is not the case, then for any $0\leq a \le \min\{J,K\}$, we have that $\mu(0)\nu(a)=\mu(a)\nu(0)=0$, and so $\nu(a)=0$. Hence if $J > K$, then $\nu(\{0,1,\dots,K\})=0$, which contradicts the assumption that $\nu$ is a probability measure. On the other hand, if $J <K$, then there must exist $J<a\le K$ such that $\nu(a)>0$ and $\nu(b)=0$ for all $0\le b <a$. For this choice of $a$, we have that
$0<\mu(0)\nu(a)=\mu(J)\nu(a-J)=0$, which is a contradiction. Thus we obtain in either case that $\nu(0)>0$.

Now, the conclusion of the previous paragraph implies that \eqref{munu} holds. Moreover, without loss of generality, we can assume $J <K$. Given the relation \eqref{munu} and the latter assumption, we recall from the proof of Lemma \ref{geolmm} that the detailed balance equation is equivalent to \eqref{ee1} holding. Suppose for the moment that $\mu(\{0,1,\dots, \lfloor\frac{J}{2}\rfloor\})=1$. From \eqref{munu}, we must therefore have that $\nu(\{\lfloor \frac{J}{2}\rfloor +1,\dots,J\})=0$. Moreover, if $\nu(a)>0$ for some $a\in\{J+1,\dots,K\}$, then \eqref{ee1} yields $0<\mu(0)\nu(a)=\mu(J)\nu(a-J)=0$, which is a contradiction. Hence it is the case that $\nu(a)=\mu(a)$ on $\{0,1,\dots, \lfloor\frac{J}{2}\rfloor\}$, i.e.\ (b) holds.

On the other hand, suppose that $\mu(\{0,1,\dots, \lfloor\frac{J}{2}\rfloor\}) <1$. Let $M$ be the maximum of the support of $\mu$; note that $2M >J$ by assumption. We will next show that it must be the case that $M=J$. If $M <J$, then for any $J \le a \le K$ we have from \eqref{ee1} that $\nu(0)\nu(a)=\nu(J)\nu(a-J)=0$, and so $\nu(a)=0$. Hence $\nu(\{0,1,\dots,M\})=1$. Moreover, if $2M \le K$, then $0<\nu(M)\nu(M)=\nu(J-M)\nu(3M-J)$,
and so $3M-J \le M$. However the latter inequality contradicts $2M >J$. On the other hand, if $2M >K$, then applying \eqref{ee1} again gives $0<\nu(M)\nu(M)=\nu(M+J-K)\nu(M+K-J)=0$, since $M+K-J>M$, and this is again a contradiction. Hence we conclude that $M=J$.

Next, let $m$ be the minimum element of $\{1,\dots,J\}$ such that $\mu(m)>0$. We will show that $m$ is a factor of both $J$ and $K$. If $K<m+J$, then \eqref{ee1} and the result that $M=J$ implies $0<\nu(J)\nu(m)=\nu(m+J-K)\nu(K)$, but $0<m+J-K<m$, and so this contradicts the choice of $m$. Hence $J\le m+J\le K$, and so $0<\nu(m)\nu(J)=\nu(J-m)\nu(2m)$. In particular, we have that both $\nu(2m)>0$ and $\nu(J-m)>0$. Now, the choice of $m$ implies that either $J-m=0$ or $J-m \ge m$. If the latter holds, then $J \ge 2m$ and since $J \le m+J \le K$, $0<\nu(2m)\nu(J-m)=\nu(J-2m)\nu(3m)$. Thus, in the same way, we see that either $J-2m=0$ or $J -2m \ge m$ hold. Repeating these steps inductively, we can conclude that $J=\ell m$ for some $\ell \in \mathbb{N}$, and moreover that $\nu(am)>0$ for $a=0,1,2,\dots, \ell$. Similarly, for any $a=0,1,2,\dots, \ell$, if $am+J \le K$ holds, then $0<\nu(J)\nu(am)=\nu(0)\nu(am+J)$, and so $\nu(am+J)>0$. Iterating the same argument, we find that for $a \in \mathbb{N}$ satisfying $am \le K$, it holds that $\nu(am)>0$. For $K <\infty$, let $p\in \mathbb{Z}_+$ be such that $J+pm \le K$ and $J+(p+1)m>K$, then, since $(p+1)m \le K$, $0<\nu(J)\nu((p+1)m)=\nu(J+(p+1)m-K)\nu(K)$. By the choice of $p$, we have that $0<J+(p+1)m-K \le m$. Hence, by the choice of $m$ and that $\nu(J+(p+1)m-K)>0$, it must hold that $J+(p+1)m-K = m$. This yields $K=\ell' m$ where $\ell'=\ell+p$.

To complete the proof, it remains to show that both $\mu$ and $\nu$ are supported on multiples of $m$ only, as we can then apply Lemma \ref{sgeolmm}. To do this, we will suppose that $\tilde{m}:=\min\{a\in\{0,1,\dots,K\}:\:\nu(a)>0,\:a\not\in m\mathbb{Z}_+\}$ exists and derive a contradiction. By construction, $\tilde{m} >m$. Since $J=\ell m$ and $K=\ell'  m$, we can repeat the same argument as above with $m$ replaced by $\tilde{m}$ to show that $J=\tilde{\ell}\tilde{m}$ and $K=\tilde{\ell}'\tilde{m}$. Hence $J+\tilde{m} \le K$. If $\tilde{m} >2m$, then since $J \le J+\tilde{m}-m \le K$, $0<\nu(J-m)\nu(\tilde{m})=\nu(m)\nu(J+\tilde{m}-2m)$. Since $J \le J+\tilde{m}-2m \le K$, it follows that $0<\nu(0)\nu(J+\tilde{m}-2m)=\nu(J)\nu(\tilde{m}-2m)$. In particular, this implies that $\nu(\tilde{m}-2m)>0$, and since $0<\tilde{m}-2m<\tilde{m}$, this is a contradiction with the choice of $\tilde{m}$. Next, suppose $\tilde{m} < 2m$. It then holds that
\[0<\nu(J-m)\nu(\tilde{m})\nu(2m)/\nu(m)=\nu(J+\tilde{m}-2m)\nu(2m)=\nu(2m-\tilde{m})\nu(J-2m+2\tilde{m}),\]
which implies $\nu(2m-\tilde{m})>0$. Since $0<2m-\tilde{m}<\tilde{m}$, this is a contradiction with the choice of $\tilde{m}$. Therefore we conclude that $\tilde{m}$ does not exist, and so the support of $\mu$ is $\{0,m,2m,\dots,\ell m=J \}$ and the support of $\nu$ is $\{0,m,2m,\dots,\ell' m=K \}$.
\end{proof}

\begin{rem}
Under the measures described in part (a) of the last theorem, it almost-surely holds that $T\eta=\theta^{-1}\eta$. In conjunction with the fact that in the case $J=K$ any measure $\mu_{J,K}$ yields the same dynamics, we see that only the measures of part (b) give non-trivial dynamics in the class $\mathcal{I}_{J,K}^{inv}\cap\{\mu_{J,K}(0)>0\}$.
\end{rem}

We are nearly ready to state the main conclusion of the section, which describes all i.i.d.\ measures that are invariant for the BBS($J$,$K$) dynamics. The result follows from the previous theorem, together with manipulations that undo earlier reductions made using the reducibility property \eqref{reducibility} and empty box-ball duality \eqref{spaceball}. Towards stating the theorem, let $\mathcal{I}_{J,K}^0$ be the set of probability measures characterized in Theorem \ref{m0thm}, namely the collection of $\mu_{J,K}$ satisfying (a) or (b) of the latter result. For $J<\infty$, we also define (slightly abusing notation) $\sigma_{J}(\mathcal{I}_{J,K}^0):=\{\mu_{J,K}\circ\sigma_J:\:\mu_{J,K}\in\mathcal{I}_{J,K}^0\}$, and set
\[\mathcal{I}_{J,K}=\left\{
\begin{array}{ll}
  \mathcal{I}_{J,K}^0\cup\sigma_{J}\left(\mathcal{I}_{J,K}^0\right), & \mbox{if }\max\{J,K\}<\infty,\\
  \mathcal{I}_{J,K}^0, & \mbox{otherwise}.
\end{array}\right.\]
And, we recall from \eqref{iinvdef} that the principal set of interest is denoted $\mathcal{I}_{J,K}^{inv}$. NB. We only define the dynamics on $\mathcal{C}^{can}_{J,K}$, and clearly $\mathcal{M}^{rev}_{J,K}$ is equal to the subset of $\mu_{J,K}\in \mathcal{M}_{J,K}$ such that $\mu_{J,K}^{\otimes \mathbb{Z}}(\mathcal{C}^{can}_{J,K})=1$. Hence $\mathcal{I}_{J,K}^{inv}$ indeed represents all i.i.d.\ measures that are invariant under the operation $T_{J,K}$. We give just one further preparatory lemma.

\begin{lmm}\label{classlmm}
Suppose $J \neq K$ and $\mu \in \mathcal{M}^{rev}_{J,K}$. It then holds that $\mu_{J,K} \in \mathcal{I}^{inv}_{J,K}$ if and only if $\tilde{\mu}_{J,K} \in \mathcal{I}^{0}_{\tilde{J},\tilde{K}}$, where $(\tilde{J},\tilde{K})=(J-2r(\mu_{J,K}),K-2r(\mu_{J,K}))$, and also $r(\mu_{J,K})=\underline{r}(\mu_{J,K})$ when $\max\{J,K\}=\infty$.
\end{lmm}
\begin{proof}
From Theorem \ref{iidthm}, Lemma \ref{reducible} and Theorem \ref{m0thm}, $\mu_{J,K} \in \mathcal{I}^{inv}_{J,K}$ implies $\tilde{\mu}_{J,K} \in \mathcal{I}^{0}_{\tilde{J},\tilde{K}}$, and also $r(\mu_{J,K})=\underline{r}(\mu_{J,K})$ when $\max\{J,K\}=\infty$. Conversely, assuming the latter conditions, from Theorem \ref{m0thm}, we find there exists $\tilde{\mu}_{K,J}'$ such that \eqref{dbother} holds and $\tilde{\mu}_{K,J}'(0)>0$. Thus from Theorem \ref{iidthm} and Lemma \ref{reducible} we obtain $\mu_{J,K} \in \mathcal{I}^{inv}_{J,K}$.
\end{proof}

\begin{thm}\label{iidclass}
It holds for $J=K$ that $\mathcal{I}_{J,K}^{inv}= \mathcal{M}_{J,K}$, and for $J \neq K$ that
\[\mathcal{I}_{J,K}^{inv}= \bigcup_{r=0}^{\lfloor\frac{\min\{J,K\}}{2}\rfloor} \mathcal{E}_r\left(\mathcal{I}_{J-2r,K-2r}\right),\]
with the convention that $\mathcal{I}_{0,K}=\mathcal{I}_{J,0}=\emptyset$.
\end{thm}
\begin{proof} It is enough to show that, in the case $J \neq K$,
\[\mathcal{I}_{J,K}^{inv} \cap \left\{\mu_{J,K} \in \mathcal{M}_{J,K}:\: r(\mu_{J,K})=r\right\}=
\left\{
  \begin{array}{ll}
    \emptyset, & \hbox{when }r >  {\lfloor\frac{\min\{J,K\}}{2}\rfloor},\\
    \mathcal{E}_r\left(\mathcal{I}_{J-2r,K-2r}\right), & \hbox{when }0 \le r \le  {\lfloor\frac{\min\{J,K\}}{2}\rfloor}.
  \end{array}
\right.\]
Since $J \neq K$ and $r >  {\lfloor\frac{\min\{J,K\}}{2}\rfloor}$ implies $J >K$ and $K \le 2r $ or $J<K$ and $J \le 2r$, from Proposition \ref{mrevprp}, $\mathcal{M}_{J,K}^{rev} \cap \{\mu_{J,K} \in \mathcal{M}_{J,K}:\:r(\mu_{J,K})=r\}=\emptyset$. On the other hand, if $r \le {\lfloor\frac{\min\{J,K\}}{2}\rfloor}$, then $\tilde{J}:=J-2r \ge 0$ and $\tilde{K}:=K-2r \ge 0$. If one of them is $0$, then, by the same reasoning as for the case $ r > {\lfloor\frac{\min\{J,K\}}{2}\rfloor}$, it must hold that $\mathcal{M}_{J,K}^{rev} \cap \{\mu_{J,K} \in \mathcal{M}_{J,K}:\: r(\mu_{J,K})=r\}=\emptyset$. If $\min\{\tilde{J}, \tilde{K}\} >0$, then since $\mu_{J,K}=\mathcal{E}_{r(\mu_{J,K})}(\tilde{\mu}_{J,K})$ if $r(\mu_{J,K})=\underline{r}(\mu_{J,K})$, and $\mu_{J,K}=\sigma_J\mathcal{E}_{r(\mu_{J,K})}(\tilde{\mu}_{J,K})$ otherwise, we deduce from Lemma \ref{classlmm} that $\mathcal{I}_{J,K}^{inv} \cap \left\{\mu_{J,K} \in \mathcal{M}_{J,K}:\: r(\mu_{J,K})=r\right\}=\mathcal{E}_r\left(\mathcal{I}_{J-2r,K-2r}\right)$.
\end{proof}

\section{Speed of tagged particle}\label{speedsec}

In this section, we prove Theorem \ref{speedthm}. The argument is an adaptation of the proof of the relevant part of \cite[Theorem 3.38]{CKST}.

\begin{proof}[Proof of Theorem \ref{speedthm}.] At time $t$, there are precisely $\sum_{s=0}^{t-1}(T^sW)_0$ particles between the tagged particle and the origin. Since, by Lemma \ref{indep}, $((T^tW)_0)_{t\in\mathbb{Z}}$ is an i.i.d.\ sequence, the latter sum satisfies
\[t^{-1}\sum_{s=0}^{t-1}(T^sW)_0\rightarrow \int W_0(\eta)\mu_{J,K}^{\otimes \mathbb{Z}}(d\eta)=m_{K,J},\]
almost-surely with respect to $\mu_{J,K}^{\otimes \mathbb{Z}}$. Thus, given any $\varepsilon>0$,  $\mu_{J,K}^{\otimes \mathbb{Z}}$-a.s.\ for large $t$,
\begin{equation}\label{aaaa}
X_{J,K}(t)\left\{
\begin{array}{l}
  \leq \inf\left\{n\geq1:\:\sum_{m=1}^{n}(T^t\eta)_m\geq t\left(m_{K,J}+\varepsilon\right)\right\},\\
  \geq \inf\left\{n\geq1:\:\sum_{m=1}^{n}(T^t\eta)_m\geq t\left(m_{K,J}-\varepsilon\right)\right\}.
\end{array}\right.
\end{equation}
Now, from Theorem \ref{iidclass}, we know that when $J\neq K$ it is the case that $\sum_xx^4\mu_{J,K}(x)<\infty$, and so we can apply a standard fourth moment estimate to deduce that, for any $c,\varepsilon\in(0,\infty)$, there exists a $C<\infty$ such that
\[\mu_{J,K}^{\otimes \mathbb{Z}}\left(\left|\sum_{m=1}^{ct}(T^t\eta)_m-ct m_{J,K}\right|>\varepsilon t\right)\leq Ct^{-2},\qquad \forall t\geq 1.\]
Hence, by Borel-Cantelli and countability, we obtain that $t^{-1}\sum_{m=1}^{ct}(T^t\eta)_m\rightarrow cm_{J,K}$ for all rational $c>0$, $\mu_{J,K}^{\otimes \mathbb{Z}}$-a.s. Combining this with \eqref{aaaa} yields the desired limit.
\end{proof}

\providecommand{\bysame}{\leavevmode\hbox to3em{\hrulefill}\thinspace}
\providecommand{\MR}{\relax\ifhmode\unskip\space\fi MR }
\providecommand{\MRhref}[2]{%
  \href{http://www.ams.org/mathscinet-getitem?mr=#1}{#2}
}
\providecommand{\href}[2]{#2}

\end{document}